\documentclass[reqno]{amsart}
\usepackage{amsmath, amsthm, amssymb, amstext}

\usepackage[left=2.5cm,right=2.5cm,top=2.5cm,bottom=2.5cm]{geometry}
\usepackage{hyperref,xcolor}
\hypersetup{pdfborder={0 0 0},colorlinks}

\usepackage{enumitem}
\allowdisplaybreaks
\newtheorem{theorem}{Theorem}
\newtheorem{remark}[theorem]{Remark}
\newtheorem{lemma}[theorem]{Lemma}
\newtheorem{proposition}[theorem]{Proposition}
\newtheorem{corollary}[theorem]{Corollary}
\newtheorem{definition}[theorem]{Definition}

\numberwithin{theorem}{section}
\numberwithin{equation}{section}

\title[On singularly logarithmic $(p, N)$-Laplace Schr\"{o}dinger equations]{On singularly perturbed $(p, N)$-Laplace Schr\"{o}dinger equation with logarithmic nonlinearity}

\author[D.K. Mahanta]{Deepak Kumar Mahanta}
\address[D.K. Mahanta]{Department of Mathematics, Indian Institute of Technology Jodhpur, Rajasthan 342030, India}
\email{mahanta.1@iitj.ac.in}

\author[T. Mukherjee]{Tuhina Mukherjee}
\address[T. Mukherjee]{Department of Mathematics, Indian Institute of Technology Jodhpur, Rajasthan 342030, India}
\email{tuhina@iitj.ac.in}

\author[P. Winkert]{Patrick Winkert}
\address[P. Winkert]{Technische Universit\"{a}t Berlin, Institut f\"{u}r Mathematik, Stra\ss e des 17.\,Juni 136, 10623 Berlin, Germany}
\email{winkert@math.tu-berlin.de}


\begin{document}

\begin{abstract}
	This article focuses on the study of the existence, multiplicity and concentration behavior of ground states as well as the qualitative aspects of positive solutions for a $(p, N)$-Laplace Schr\"{o}dinger equation with logarithmic nonlinearity and critical exponential nonlinearity in the sense of Trudinger-Moser in the whole Euclidean space $\mathbb{R}^N$. Through the use of smooth variational methods, penalization techniques, and the application of the Lusternik-Schnirelmann category theory, we establish a connection between the number of positive solutions and the topological properties of a set in which the potential function achieves its minimum values.
\end{abstract}

\maketitle

\section{Introduction}

In this article, we deal with the following singularly perturbed $(p, N)$-Laplace Schr\"{o}dinger equation
\begin{align}\label{main problem}\tag{$\mathcal{P}_{\varepsilon}$}
	\begin{cases}
		\mathcal{L}_{p_\varepsilon}(u)+\mathcal{L}_{N_\varepsilon}(u)=|u|^{N-2}u \log |u|^N+f(u)\quad \text{in } \mathbb{R}^N, \\[1ex]
		\displaystyle \int_{\mathbb{R}^N} V(x)\big(|u|^p+|u|^N\big)\,\mathrm{d}x<+\infty, \quad u\in  W^{1,p}(\mathbb{R}^N)\cap W^{1,N}(\mathbb{R}^N) ,
	\end{cases}
\end{align}
where
\begin{align*}
	\mathcal{L}_{t_\varepsilon}(u)=-\varepsilon^t\Delta_t u+V(x)|u|^{t-2}u\quad\text{for }t\in\{p,N\}
\end{align*}
with $N\geq 2$. Further, we assume that $1<p<N$ and $\varepsilon$ is a very small positive parameter. The operator $\Delta_t u=\operatorname{div}(|\nabla u|^{t-2}\nabla u)$ with $t\in\{p,N\}$ is the standard $t$-Laplace operator and the scalar potential $V\colon\mathbb{R}^N\to \mathbb{R} $ is a continuous function. The nonlinearity $f\colon\mathbb{R}\to \mathbb{R} $ has critical exponential growth at infinity, i.e., it behaves like $\exp(\alpha|u|^{\frac{N}{N-1}})$ when $|u|\to\infty$ for some $\alpha>0$, which means that there exists a positive constant $\alpha_0$ such that the following condition holds:
\begin{align*}
	\lim_{|u|\to \infty} |f(u)| \operatorname{exp}(-\alpha|u|^{\frac{N}{N-1}})=
	\begin{cases}
		0       & \text{if } \alpha>\alpha_0, \\
		+\infty & \text{if }\alpha<\alpha_0.
	\end{cases}
\end{align*}
Throughout the paper, we suppose the following assumptions on the scalar potential $V\colon\mathbb{R}^N\to \mathbb{R}$:
\begin{enumerate}
	\item[(V1)]
		$V\in C(\mathbb{R}^N;\mathbb{R})$ and there exists a constant $V_0>0$ such that $\inf_{x\in\mathbb{R}^N} V(x)\geq V_0$.
	\item[(V2)]
		There exists an open and bounded set $\Lambda\subset \mathbb{R}^N$ such that
		\begin{align*}
			V_0=\inf_{x\in \Lambda} V(x)<\min_{\partial \Lambda} V(x).
		\end{align*}
\end{enumerate}
We define
\begin{align*}
	M=\{x\in\Lambda\colon V(x)=V_0\}
	\quad \text{and}\quad
	M_\delta=\{x\in\mathbb{R}^N\colon \operatorname{dist}(x,M)\leq \delta\}
\end{align*}
for $\delta>0$ small enough such that $M_\delta\subset\Lambda$. Moreover, the nonlinearity $f\colon\mathbb{R}\to\mathbb{R}$ is supposed to satisfy the following conditions:
\begin{enumerate}
	\item
		[(f1)] Let
		$f\in C^1(\mathbb{R},\mathbb{R})$ be an odd function such that $f(0)=0, f(s)<0$ for all $s< 0$ and $f(s)>0$ for all $s>0$. Further, there exists a constant $\alpha_0\in(0,\alpha)$ with the property that for all $\tau>0$, there exists $\kappa_\tau>0$ such that for all $s\in \mathbb{R}$, we have
		\begin{align*}
			|f(s)|\leq \tau |s|^{N-1}+\kappa_\tau \Phi\left(\alpha_0 |s|^{N'}\right)
			\quad \text{with }\quad \Phi(t)=\exp(t)-\sum_{j=0}^{N-2}\frac{t^j}{j!}
			\quad\text{and}\quad
			N'=\frac{N}{N-1}.
		\end{align*}
	\item[(f2)]
		There exists $\mu>N$ such that
		\begin{align*}
			sf(s)-\mu F(s)\geq 0\quad \text{for all  }s\in\mathbb{R}, \text{ where}\quad F(s)=\int_{0}^{s}f(t)\,\mathrm{d}t\quad \text{for all } s\in\mathbb{R}.
		\end{align*}
	\item[(f3)]
		The mapping $s\mapsto\displaystyle\frac{f(s)}{|s|^{N-2}s}$ is increasing for all $s>0$ and decreasing for all $s<0$.
	\item[(f4)]
		There exists a constant $\gamma>0$ such that $f(s)\geq \gamma s^{\mu-1}$ for all $s\geq 0$.
\end{enumerate}

\begin{remark}
	A typical example of a function that satisfies \textnormal{(f1)}--\textnormal{(f4)} can be considered as
	\begin{align*}
		f(s)=|s|^{N-2}s~\Phi(|s|^{N'})\quad\text{for all }s\in\mathbb{R},
	\end{align*}
	with $\alpha_0>1$, where $N\geq 2$, $N'=\frac{N}{N-1}$ and $\Phi$ is defined as in \textnormal{(f1)}.
\end{remark}

In order to familiarize the reader with the special behaviors of the classical Sobolev spaces, it is worth pointing out that the space $W^{1,p}(\mathbb{R}^N)$ can be distinguished in three different ways, namely:
\begin{align*}
	\text{(a) the Sobolev case: } p<N,
	\qquad
	\text{(b) the Sobolev limiting case: }p=N,
	\qquad
	\text{(c) the Morrey case: }p>N.
\end{align*}
The Sobolev embedding theorem says that for $p<N$, there holds $W^{1,p}(\mathbb{R}^N)\hookrightarrow L^q (\mathbb{R}^N)$ for any $q\in[p,p^*]$, where $p^*=\frac{Np}{N-p}$ is the critical Sobolev exponent to $p$. In this scenario, to study variational problems, the nonlinearity cannot exceed the polynomial of degree $p^*$. In contrast to this, for the Sobolev limiting case commonly known as the Trudinger-Moser case, one can notice that $p^*$ converges to $\infty$ as $p$ converges to $N$ and thus{,}, we might except that $W^{1, N}(\mathbb{R}^N)$ is continuously embedded in $L^\infty(\mathbb{R}^N)$. This is, however, wrong for $N>1$. In order to see this, let $\varphi\in C^\infty_{c}(\mathbb{R}^N,[0,1])$ be such that $\varphi\equiv 1$ in $B_1(0)$ and $\varphi\equiv 0$ in $ B^c_{2}(0)$, then the function $u(x)=\varphi(x)\log\big(\log \big(1+\frac{1}{|x|}\big)\big)$ belongs to $W^{1,N}(\mathbb{R}^N)$ but not to $L^\infty(\mathbb{R}^N)$. Moreover, in this situation, every polynomial growth is allowed. To fill this gap, it is fairly natural to look for the maximal growth of a function $g\colon\mathbb{R}\to \mathbb{R}^+$ such that
\begin{align*}
	\sup_{\substack{u\in W^{1,N}(\mathbb{R}^N)\\ \|u\|_{W^{1,N}}\leq 1}} ~\int_{\mathbb{R}^N}\cfrac{g(u)}{|x|^\beta}\,\mathrm{d}x<+\infty\quad\text{for all } 0\leq\beta<N,
\end{align*}
where $ \|u\|_{W^{1,N}}=\big(\|\nabla u\|_N^N+\|u\|_{N}^N\big)^{\frac{1}{N}}$ and $\|\cdot\|_N$ is the usual norm of the Lebesgue space $L^N(\mathbb{R}^N)$. It is noteworthy that many authors have independently proved that the maximum growth of such a function $g$ is of exponential type. In that context, we mention the works of Adimurthi-Yang \cite{Adimurthi-Yang-2010} and Li-Ruf \cite{Li-Ruf-2008}. In recent years, the existence and multiplicity of solutions to elliptic equations involving the $N$-Laplace operator with subcritical and critical growth in the sense of Trudinger-Moser inequality have been extensively studied, motivated by their applicability in many fields of modern mathematics. For a detailed study, we refer to Beckner \cite{Beckner-1993}, Chang-Yang \cite{Chang-Yang-2003}, Chen-Lu-Zhu \cite{Chen-Lu-Zhu-2021}, Lam-Lu \cite{Lam-Lu-2012}, Zhang-Zhu \cite{Zhang-Zhu-2024} and the references therein.

It should be pointed out here that the Trudinger-Moser type inequalities and the Adams type inequalities have been widely studied by many authors across diverse domains such as Euclidean spaces, Heisenberg groups, Riemannian manifolds, and so on. In this context, we recommend that readers take a look at some works by Chen-Wang-Zhu \cite{Chen-Wang-Zhu-2023}, Cohn-Lu \cite{Cohn-Lu-2001}, do \'{O}-Lu-Ponciano \cite{do-Lu-Ponciano-2024}, Duy-Phi \cite{Duy-Phi-2022}, Jiang-Xu-Zhang-Zhu \cite{Jiang-Xu-Zhang-Zhu-2025}, Lam-Lu \cite{ Lam-Lu-2013,Lam-Lu-2012-b}, Li-Lu-Zhu \cite{Li-Lu-Zhu-2021}, Wang \cite{Wang-2025}, Xue-Zhang-Zhu \cite{Xue-Zhang-Zhu-2025} and the references cited therein.

In the last decade, great attention has been focused on the study of $(p,q)$-Laplace equations as well as double phase problems in the whole Euclidean space $\mathbb{R}^N$ due to the broad applications in biophysics, plasma physics, solid state physics, and chemical reaction design, see, for example, the books of Aris \cite{Aris-1994}, Fife \cite{Fife-1979} and Murray \cite{Murray-1993} as well as the papers of Myers-Beaghton-Vvedensky \cite{Myers-Beaghton-Vvedensky-1989} and Wilhelmsson \cite{Wilhelmsson-1987} and the references therein. On the other hand, concerning the Sobolev limiting case, that is, $p<q$ and $q=N$, such types of problems are often comparatively less looked upon. This is one of the main motivations for the study in this article. More details on $(p, N)$-Laplace equations can be found in the papers of Carvalho-Figueiredo-Furtado-Medeiros \cite{Carvalho-Figueiredo-Furtado-Medeiros-2021}, Chen-Fiscella-Pucci-Tang \cite{Chen-Fiscella-Pucci-Tang-2020}, Fiscella-Pucci \cite{Fiscella-Pucci-2021}, Mahanta-Mukherjee-Sarkar \cite{Mahanta-Mukherjee-Sarkar-2025} and Mahanta-Winkert \cite{Mahanta-Winkert-2026}, as well as the references therein.

Moreover, we make a note that one of the hypotheses on the potential function $V$ that appears in (V1) says that the corresponding first-order weighted Sobolev spaces make sense and are well-behaved; see, for instance, Lemma \ref{lem2.1} and Lemma \ref{lem2.2}, respectively. To address the challenge posed by the lack of compactness, Bartsch-Wang \cite{Bartsch-Wang-1995} were the first to place assumptions on the potential function $V$. Further, as an application, they studied the existence and multiplicity of solutions for a superlinear Schr\"{o}dinger type equation in $\mathbb{R}^N$. Afterwards,  reducing the conditions on the potential and the nonlinearity, Tang \cite{Tang-2013} achieved some more general results. Later on, Chen-Lu-Zhu \cite{Chen-Lu-Zhu-2020, Chen-Lu-Zhu-2023,Chen-Lu-Zhu-2023-c,  Chen-Lu-Zhu-2021} developed and introduced some more generalized form of assumptions on the potential function $V$, called degenerated and trapping types of potentials, to establish the Trudinger-Moser type inequalities as well as the consequences of the Adams type inequalities, and by employing such inequalities, they studied elliptic and subelliptic PDEs. In addition, we also mention here that Chen-Lu-Zhu \cite{Chen-Lu-Zhu-2023-b} showed the existence of extremals for {T}rudinger-{M}oser inequalities in $\mathbb{R}^2$ in the presence of trapping potential.

Nowadays, there is a great interest in the study of the time-dependent nonlinear logarithmic Schr\"{o}dinger equation of the form
\begin{align}\label{ncl}\tag{NLS}
	i\varepsilon \partial_t \Psi=-\varepsilon^2 \Delta \Psi +(V(x)+E)\Psi-\Psi\log|\Psi|^2 \quad\text{for all }(x,t)\in \mathbb{R}^N\times [0,+\infty),
\end{align}
where $\Psi\colon\mathbb{R}^N\times [0,+\infty)\to\mathbb{C}$, $N\geq 2$, $E\in\mathbb{R}$, $\varepsilon$ is a positive parameter and $V$ is a continuous function satisfying certain hypotheses. It is worth noting that the standing wave solution of \eqref{ncl} is of the form $\Psi(x,t)=\exp{(-iEt/\varepsilon)}u(x)$, where $u$ is a solution of the equation
\begin{align}\label{eq1.1}
	\begin{cases}
		-\varepsilon^2 \Delta u+V(x)u=u\log u^2\quad \text{in }\mathbb{R}^N,\\
		~ u\in H^1(\mathbb{R}^N).
	\end{cases}
\end{align}
From the point of view of the application, such equations are the main tools for studying quantum physics, quantum optics, effective quantum gravity, nuclear physics, transport and diffusion phenomena, theory of superfluidity and Bose-Einstein condensation. For more information in this direction, we refer to Bia\l ynicki-Birula-Mycielski \cite{Bialynicki-Birula-Mycielski-1975}, Carles-Gallagher \cite{Carles-Gallagher-2018}, Cazenave \cite{Cazenave-1983}, Cazenave-Lions \cite{Cazenave-Lions-1982}, Zloshchastiev \cite{Zloshchastiev-2010} and the references therein. In addition, in order to study \eqref{eq1.1}, there have been several technical difficulties due to the presence of logarithmic nonlinearity. For example, let $u$ be a smooth function satisfying
\begin{align*}
	u(x)=
	\begin{cases}
		\big(|x|^{\frac{N}{2}}\log |x|\big)^{-1} &\text{if }|x|\geq 3,\\
		0 & \text{if }|x|\leq 2.
	\end{cases}
\end{align*}
By direct computation, one has $u\in H^1(\mathbb{R}^N)$ but $\int_{\mathbb{R}^N}u^2\log u^2\,\mathrm{d}x=-\infty$. So, the Euler-Lagrange functional associated to \eqref{eq1.1} is not finite and is no longer $C^1$ on $H^1(\mathbb{R}^N)$.  As a result, we cannot directly use the classical critical point theory to study the behavior of solutions of \eqref{eq1.1}. To overcome these difficulties, several approaches have been developed in the mathematical literature so far. We will discuss some of them below.

Initially, Cazenave \cite{Cazenave-1983} studied the following time-dependent logarithmic Schr\"{o}dinger equation
\begin{align}\label{eq1.2}
	iu_t+\Delta u+u\log u^2=0\quad\text{in }\mathbb{R}\times \mathbb{R}^N
\end{align}
by considering the $N$-function $A$ and the function space $W$ defined as
\begin{align*}
	A(s)=\begin{cases}
		-\frac{1}{2} s^2 \log s^2 &\text{if } 0\leq s\leq e^{-3},\\
		3s^2+4e^{-3} s-e^{-6}~& \text{if } s\geq e^{-3},
	\end{cases}
	\quad\text{and}\quad
	W=\left\{u\in H^1(\mathbb{R}^N)\colon \int_{\mathbb{R}^N}|u^2\log u^2|~\mathrm{d}x<+\infty\right\},
\end{align*}
endowed with the Luxemburg norm $\|\cdot\|_W=\|\cdot\|_{H^1(\mathbb{R}^N)}+\|\cdot\|_A$, where
\begin{align*}
	\|u\|_A=\inf\left\{\lambda>0\colon \int_{\mathbb{R}^N}A(\lambda^{-1}|u|)\,\mathrm{d}x\leq 1\right\}.
\end{align*}
The author defined the associated functional $L\colon W\to\mathbb{R}$ given by
\begin{align*}
	L(u)=\frac{1}{2}\int_{\mathbb{R}^N}|\nabla u|^2 \,\mathrm{d}x-\frac{1}{2}\int_{\mathbb{R}^N}u^2\log u^2\,\mathrm{d}x\quad \text{for all }u\in W
\end{align*}
and proved the existence of infinitely many critical points of $L$ on the set $\{u\in W\colon \int_{\mathbb{R}^N}|u|^2\,\mathrm{d}x=1\}$. As a result, he also provided a lot of information about the behavior of the solutions of equation \eqref{eq1.2}.

Later, Squassina-Szulkin \cite{Squassina-Szulkin-2015, Squassina-Szulkin-2017} investigated the following logarithmic Schr\"{o}dinger equation
\begin{equation}\label{eq1.3}
	\begin{cases}
		-\Delta u+V(x)u= Q(x)u\log u^2 & \text{in } \mathbb{R}^N,\\
		~u\in H^1(\mathbb{R}^N),
	\end{cases}
\end{equation}
where $V,Q\in C(\mathbb{R}^N,\mathbb{R})$ are $1$-periodic functions of the variables $x_1,x_2,\cdots,x_N$ satisfying the hypotheses
\begin{align*}
	\min_{x\in \mathbb{R}^N} Q(x)>0
	\quad\text{and}\quad
	\min_{x\in \mathbb{R}^N} (V+Q)(x)>0 .
\end{align*}
Employing the standard nonsmooth critical point theory of lower semicontinuous functionals, which was developed by Szulkin \cite{Szulkin-1986}, the authors showed first the existence of positive ground state solutions by adopting the deformation lemma. Then, by using the genus theory, they proved the existence of infinitely many high-energy solutions, which are geometrically distinct under $\mathbb{Z}^N$-action. Moreover, several authors used nonsmooth variational techniques to study the logarithmic Schr\"{o}dinger equations, such as Alves-Ambrosio \cite{Alves-Ambrosio-2024}, Alves-de Morais Filho \cite{Alves-deMoraisFilho-2018}, Alves-Ji \cite{Alves-Ji-2020,Alves-Ji-2024}, d'Avenia-Montefusco-Squassina \cite{dAvenia-Montefusco-Squassina-2014}, Deng-He-Pan-Zhong \cite{Deng-He-Pan-Zhong-2023}, Ji-Szulkin \cite{Ji-Szulkin-2016}, Li-Peng-Shuai \cite{Li-Peng-Shuai-2022} and Liu-Peng-Zou \cite{Liu-Peng-Zou-2025}. In contrast to this, Tanaka-Zhang \cite{Tanaka-Zhang-2017} have also studied \eqref{eq1.3} by considering $V, Q$ as spatially $1$-periodic functions of class $C^1$. The authors showed the existence of infinitely many multi-bump solutions for \eqref{eq1.3}, which are distinct under $\mathbb{Z}^N$-action, by taking an approach using spatially $2L$-periodic problems with $L\gg1$.

During the last decade, Wang-Zhang \cite{Wang-Zhang-2019} introduced an advanced way of studying logarithmic equations, which is known as the power approximation method. First, they considered the following semiclassical scalar field equation with power-law nonlinearity
\begin{equation}\label{eq1.4}
	\begin{cases}
		-\Delta u+\lambda u= |u|^{p-2}u&\text{in }\mathbb{R}^N,\\
		~\displaystyle\lim_{|x|\to\infty}u(x)= 0,
	\end{cases}
\end{equation}
where $p\in(2,2^*)$ with $2^*=\frac{2N}{N-2}$ if $N\geq 3$ and $2^*=+\infty$ if $N\leq 2$. The authors showed that when $p\searrow 2$, then the ground state solutions of \eqref{eq1.4} either blow up or vanish, and converge to the ground state solutions of the logarithmic-scalar field equation
\begin{align*}
	\begin{cases}
		-\Delta u=\lambda u\log|u|&\text{in }\mathbb{R}^N,\\
		~\displaystyle\lim_{|x|\to\infty}u(x)= 0.
	\end{cases}
\end{align*}
In addition, they also proved that the same result holds for bound-state solutions. Later, the authors studied the concentration behavior of nodal solutions of \eqref{eq1.1} in \cite{Zhang-Wang-2020} by employing the same idea discussed above.

On the other hand, concerning the penalization method and the Lusternik-Schnirelmann category theory, which are generally used to study the multiplicity of the positive solutions of nonlinear PDEs and their concentration phenomena, we recommend the readers to study the papers of Alves-Figueiredo \cite{Alves-Figueiredo-2009}, Ambrosio-Repov\v{s} \cite{Ambrosio-Repovs-2021}, Thin \cite{Thin-2022} and Zhang-Sun-Liang-Thin \cite{Zhang-Sun-Liang-Thin-2024}, see also the references therein. The most important features and novelties of our problem are listed below:
\begin{enumerate}
	\item[(a)]
		The appearance of the $(p, N)$-Laplace operator in our problem is nonhomogeneous, and thus, the calculations are more complicated.
	\item[(b)]
		Due to the lack of compactness caused by the unboundedness of the domain, the Palais-Smale sequences do not have the compactness property.
	\item[(c)]
		The reaction combines the multiple effects generated by the logarithmic term and a term with critical growth with respect to the exponential nonlinearity, making our study more delicate and challenging.
	\item[(d)]
		The concentration phenomena create a bridge between the global maximum point of the solution and the global minimum of the potential function.
	\item[(e)]
		The proofs combine refined techniques, including variational and topological tools.
\end{enumerate}

To the best of our knowledge, this is the first time in the literature, in which two penalized functions are used simultaneously, one corresponds to the logarithmic nonlinearity and the other one corresponds to the exponential growth. Motivated by all the cited works, especially by the papers of Alves-da Silva \cite{Alves-daSilva-2023}, Alves-Ji \cite{Alves-Ji-2020} and Squassina-Szulkin \cite{Squassina-Szulkin-2015}, we study the existence, multiplicity and concentration phenomena of solutions for problem \eqref{main problem}.

Note that, by the change of variable $x\mapsto \varepsilon x$, we can see that \eqref{main problem} is equivalent to the problem
\begin{equation}\label{main problem@}\tag{$\mathcal{S}_{\varepsilon}$}
	\begin{cases}
		\widetilde{\mathcal{L}}_{p_\varepsilon}(u)+\widetilde{\mathcal{L}}_{N_\varepsilon}(u)=|u|^{N-2}u \log |u|^N+f(u)\quad\text{in } \mathbb{R}^N, \\[1ex]
		\displaystyle  \int_{\mathbb{R}^N} V(\varepsilon x)\big(|u|^p+|u|^N\big)\,\mathrm{d}x<+\infty,\quad u\in  W^{1,p}(\mathbb{R}^N)\cap W^{1,N}(\mathbb{R}^N),
	\end{cases}
\end{equation}
where
\begin{align*}
	\widetilde{\mathcal{L}}_{t_\varepsilon}(u)=-\Delta_t u+V(\varepsilon x)|u|^{t-2}u\quad\text{for }t\in\{p,N\}.
\end{align*}

\begin{definition}
	We say $u\in\mathbf{X}_\varepsilon$ (see \eqref{eq1} for its definition) is a (weak) solution of \eqref{main problem@}, if
	\begin{align*}
		|u|^{N-2}u\psi \log |u|^N\in L^1(\mathbb{R}^N)
	\end{align*}
	and
	\begin{align*}
		\big\langle{u,\psi}\big\rangle_{p,V_\varepsilon}+\big\langle{u,\psi}\big\rangle_{N,V_\varepsilon}= \int_{\mathbb{R}^N}|u|^{N-2}u\psi \log |u|^N\,\mathrm{d}x+\int_{\mathbb{R}^N}f(u)\psi\,\mathrm{d}x
	\end{align*}
	is satisfied for all $\psi\in\mathbf{X}_\varepsilon $, where  $\big\langle\cdot,\cdot\big\rangle_{t,V_\varepsilon}$ for $t\in\{p,N\}$ is defined as
	\begin{align*}
		\qquad\big\langle{u,\psi}\big\rangle_{t,V_\varepsilon}
		=\int_{\mathbb{R}^N}|\nabla u|^{t-2}\nabla u \cdot \nabla\psi\,\mathrm{d}x+\int_{\mathbb{R}^N}V(\varepsilon x)|u|^{t-2}u\psi\,\mathrm{d}x \quad\text{for all }u,\psi\in \mathbf{X}_\varepsilon .
	\end{align*}
\end{definition}

Now, we state the main results of this article.

\begin{theorem}[\textbf{Concentration phenomena}] \label{thm1.2}
	Let hypotheses \textnormal{(V1)--(V2)} and \textnormal{(f1)--(f4)} be satisfied. Then there exists $\varepsilon_0>0$ such that for any $\varepsilon\in(0,\varepsilon_0)$, problem \eqref{main problem} has a positive solution $v_\varepsilon$. Further, if $\eta_\varepsilon$ is the global maximum point of $v_\varepsilon$, then it holds
	\begin{align*}
		\lim_{\varepsilon\to 0} V(\eta_\varepsilon)=V_0.
	\end{align*}
\end{theorem}

\begin{theorem}[\textbf{Multiplicity of positive solutions}] \label{thm1.3}
	Let hypotheses \textnormal{(V1)--(V2)} and \textnormal{(f1)--(f4)} be satisfied and let $\delta>0$ be sufficiently small. Then there exists $\varepsilon_1>0$ such that for $\varepsilon\in(0,\varepsilon_1)$, the following hold:
	\begin{enumerate}
		\item[\textnormal{(a)}]
			problem \eqref{main problem} has at least $\frac{\operatorname{cat}_{M_\delta} (M)}{2}$ positive solutions, whenever $\operatorname{cat}_{M_\delta} (M)$ is an even number;
		\item[\textnormal{(b)}]
		problem \eqref{main problem} has at least $\frac{\operatorname{cat}_{M_\delta} (M)+1}{2}$ positive solutions, whenever $\operatorname{cat}_{M_\delta} (M)$ is an odd number.
	\end{enumerate}
\end{theorem}

The paper is organized as follows. In Section \ref{sec2}, we introduce the underlying function spaces, the main tools of the variational framework and some preliminary results. Section \ref{sec3} is devoted to the study of the penalized problem by using the mountain pass geometry and some topological tools. In Section \ref{sec4}, the properties of the Nehari manifold associated with the penalized problem and the concentration behavior of the positive solutions for \eqref{main problem} are studied. Finally, in Section \ref{sec5}, we prove Theorem \ref{thm1.3} by invoking the Lusternik-Schnirelmann category theory.

\section{Some Preliminary Results}\label{sec2}

This section is devoted to some basic results on Sobolev spaces, Orlicz spaces and related lemmas that will be used to establish the main results of this article. To this end, for $t\in(1,+\infty)$, $L^t(\mathbb{R}^N)$ denotes the standard Lebesgue space with the norm $\|\cdot\|_t$. Further, if $\Omega\subset \mathbb{R}^N$, then we define the norm of $L^t(\Omega)$ by $\|\cdot\|_{L^t(\Omega)}$. For nonnegative measurable functions $V\colon\mathbb{R}^N\to \mathbb{R}$, the space $L^t_{V_\varepsilon}(\mathbb{R}^N)$ consists of all real-valued measurable functions such that $V(\varepsilon x)|u|^t\in L^1(\mathbb{R}^N)$ and is equipped with the seminorm
\begin{align*}
	\|u\|_{t,V_\varepsilon}=\Bigg(\int_{\mathbb{R}^N} V(\varepsilon x)|u|^t\,\mathrm{d}x\Bigg)^{\frac{1}{t}}\quad\text{for all }u\in L^t_{V_\varepsilon}(\mathbb{R}^N),
\end{align*}
which turns into a norm due to hypothesis  (V1). The space $\big(L^t_{V_\varepsilon}(\mathbb{R}^N),\|\cdot\|_{t,V_\varepsilon}\big)$ is a  separable and uniform convex Banach space (see Pucci-Xiang-Zhang \cite{Pucci-Xiang-Zhang-2015}). Note that  under the assumption (V1), the embedding $L^t_{V_\varepsilon}(\mathbb{R}^N)\hookrightarrow L^t(\mathbb{R}^N)$ is continuous.

Next, we define
\begin{align*}
	W^{1,t}(\mathbb{R}^N)=\left\{ u\in L^{t}(\mathbb{R}^N)\colon|\nabla u|\in L^{t}(\mathbb{R}^N)\right\},
\end{align*}
endowed with the norm
\begin{align*}
	\|u\|_{W^{1,t}}=\big(\|\nabla u\|_t^t+\|u\|_{t}^t\big)^{\frac{1}{t}}.
\end{align*}
It is well-known that the space $\big(W^{1,t}(\mathbb{R}^N),\|\cdot\|_{W^{1,t}}\big)$ is a separable and uniformly convex Banach space. Note that $C_{c}^\infty(\mathbb{R}^N)$ is a dense subset of $W^{1,t}(\mathbb{R}^N)$. Moreover, the critical Sobolev exponent of $t$ is defined by $t^*=\frac{Nt}{N-t}$ if $t<N$ and $t^*=+\infty$ otherwise. Further, we set
\begin{align*}
	\mathbf{X}=W^{1,p}(\mathbb{R}^N)\cap W^{1,N}(\mathbb{R}^N)
\end{align*}
and endow it with the norm
\begin{align*}
	\|u\|_{\mathbf{X}}=\|u\|_{W^{1,p}}+\|u\|_{W^{1,N}}.
\end{align*}
Then, the space $(\mathbf{X},\|\cdot\|_{\mathbf{X}})$ is a reflexive and separable Banach space.

The weighted Sobolev space $W^{1,t}_{V_\varepsilon}(\mathbb{R}^N)$ is defined by
\begin{align*}
	W^{1,t}_{V_\varepsilon}(\mathbb{R}^N)=\left\{ u\in W^{1,t}(\mathbb{R}^N)\colon\int_{\mathbb{R}^N}V(\varepsilon x)|u|^t\,\mathrm{d}x<+\infty\right\},
\end{align*}
equipped with the norm
\begin{align*}
	\|u\|_{W^{1,t}_{V_\varepsilon}}=\left(\|\nabla u\|_t^t+\|u\|_{t,V_\varepsilon}^t\right)^{\frac{1}{t}}.
\end{align*}
The space $\big(W^{1,t}_{V_\varepsilon}(\mathbb{R}^N),\|\cdot\|_{W^{1,t}_{V_\varepsilon}}\big)$ is a separable and uniformly convex Banach space, see Proposition 2.1 in Bartolo-Candela-Salvatore \cite{Bartolo-Candela-Salvatore-2016}, thanks to (V1). Moreover, $C_c^\infty(\mathbb{R}^N)$ is a dense subset of $W^{1,t}_{V_\varepsilon}(\mathbb{R}^N)$, see Bartolo-Candela-Salvatore \cite{Bartolo-Candela-Salvatore-2016} and Chen-Chen \cite{Chen-Chen-2016}. From now on, our function space is given by
\begin{equation}\label{eq1}
	\mathbf{X}_\varepsilon=W^{1,p}_{V_\varepsilon}(\mathbb{R}^N)\cap W^{1,N}_{V_\varepsilon} (\mathbb{R}^N),
\end{equation}
which is endowed with the norm
\begin{align*}
	\|u\|_{\mathbf{X}_\varepsilon}=\|u\|_{W^{1,p}_{V_\varepsilon}}+\|u\|_{W^{1,N}_{V_\varepsilon}} \quad \text{for all }u\in \mathbf{X}_\varepsilon.
\end{align*}
Because of assumptions (V1) and Proposition 2.1 in Bartolo-Candela-Salvatore \cite{Bartolo-Candela-Salvatore-2016}, it is easy to see that  $(\mathbf{X}_\varepsilon,\|\cdot\|_{{\mathbf{X}}_\varepsilon})$ is a reflexive and separable Banach space. In the entire paper, $C, C_1, C_2, C_3,\dots$ denote some fixed positive constants possibly different at different places. Moreover, for any Banach space $(X,\|\cdot\|_X)$, we denote its continuous dual by $(X^*,\|\cdot\|_{X^*})$ and $o_n(1)$ denotes the real sequence such that $o_n(1)\to 0$ as $n\to\infty$. By $\rightharpoonup $ we mean the weak convergence and $\rightarrow$ means the strong convergence, while $u^\pm=$ max~$\{\pm u,0\}$ stand for the positive and negative part of a function $u$, respectively. Furthermore, $B_r(x_0)$ is an open ball centered at $x_0\in \mathbb{R}^N$ with radius $r>0$ and $B_r=B_r(0)$. Finally, for any set $S\subset\mathbb{R}^N$, we denote its Lebesgue measure by $|S|$ and its complement by $S^c$.

Now, we shall discuss some basic properties of Orlicz spaces.
\begin{definition}
	We say that a continuous function $\mathcal{F}\colon\mathbb{R}\to[0,+\infty)$ is a $N$-function if there hold
	\begin{enumerate}
		\item[\textnormal{(a)}]
			$\mathcal{F}$ is an even function.
		\item[\textnormal{(b)}]
			$\mathcal{F}(s)=0$ if and only if $s=0$ and $\mathcal{F}$ is convex.
		\item[\textnormal{(c)}]
			$\displaystyle\lim_{s\to 0}\frac{\mathcal{F}(s)}{s}=0$ and $\displaystyle\lim_{s\to \infty}\frac{\mathcal{F}(s)}{s}=\infty$.
	\end{enumerate}
\end{definition}

Further, we say that a $N$-function $\mathcal{F}$ satisfies the $\Delta_2$-condition, which is denoted by $\mathcal{F}\in \Delta_2 $, if there exists a constant $c>0$ such that
\begin{align*}
	\mathcal{F}(2s)\leq c \mathcal{F}(s) \quad\text{for all } s\geq 0.
\end{align*}
The conjugate of the $N$-function $\mathcal{F}$ is denoted by $\widetilde{\mathcal{F}}$ and defined as
\begin{align*}
	\widetilde{\mathcal{F}}=\max_{s\geq 0}\{ts-\mathcal{F}(s)\}\quad\text{for all }t\geq 0.
\end{align*}
Note that $\widetilde{\mathcal{F}}$ is always an $N$-function and $\widetilde{\widetilde{\mathcal{F}}}=\mathcal{F}$, i.e., $\mathcal{F}$ and $\widetilde{\mathcal{F}}$ are complementary to each other. Now, we define the Orlicz space associated with the $N$-function $\mathcal{F}$ by
\begin{align*}
	L^{\mathcal{F}}(\mathbb{R}^N)=\left\{u\in L^1_{\operatorname{loc}}(\mathbb{R}^N)\colon\int_{\mathbb{R}^N}\mathcal{F}\left(\frac{|u|}{\lambda}\right)\mathrm{d}x<+\infty \text{ for some }\lambda>0\right\},
\end{align*}
endowed with the Luxemburg norm
\begin{align*}
	\|u\|_ {\mathcal{F}}=\inf\left\{\lambda>0\colon \int_{\mathbb{R}^N}\mathcal{F}\left(\frac{|u|}{\lambda}\right)\,\mathrm{d}x\leq 1\right\}.
\end{align*}
Clearly, the space $\big(L^{\mathcal{F}}(\mathbb{R}^N),\|\cdot\|_{\mathcal{F}}\big)$ is a Banach space. Consequently, the Young's type and H\"{o}lder's type inequalities in Orlicz spaces are given by
\begin{align*}
	st\leq \mathcal{F}(s)+\widetilde{\mathcal{F}}(t)\quad \text{for all }s,t\geq 0
\end{align*}
and
\begin{align*}
	\left|\int_{\mathbb{R}^N}uv\,\mathrm{d}x\right|\leq 2 \|u\|_ {\mathcal{F}}\|v\|_ {\widetilde{\mathcal{F}}} \quad \text{for all }u\in L^{\mathcal{F}}(\mathbb{R}^N), v\in L^{\widetilde{\mathcal{F}} }(\mathbb{R}^N).
\end{align*}
Moreover, if $\mathcal{F},\widetilde{\mathcal{F}}\in\Delta_2$, then the space $L^{\mathcal{F}}(\mathbb{R}^N)$ is reflexive and separable. Again, the $\Delta_2$-condition implies at once that
\begin{align*}
	L^{\mathcal{F}}(\mathbb{R}^N)=\left\{u\in L^1_{\operatorname{loc}}(\mathbb{R}^N)\colon\int_{\mathbb{R}^N}\mathcal{F}(|u|)\,\mathrm{d}x<+\infty\right\}
\end{align*}
and
\begin{equation}\label{eq2.333}
	u_n\to u\quad\text{in }L^{\mathcal{F}}(\mathbb{R}^N)\quad\text{if and only if}\quad \int_{\mathbb{R}^N}\mathcal{F}(|u_n-u|)\,\mathrm{d}x\to 0\quad\text{as }n\to\infty.
\end{equation}
Now, we shall characterize an important property of the $N$-function. It states that if $\mathcal{F}$ is an $N$-function of class $C^1$ and $\widetilde{\mathcal{F}}$ be its conjugate such that the following condition holds
\begin{equation}\label{eq2.2}
	1<l\leq \frac{\mathcal{F}'(s)s}{\mathcal{F}(s)}\leq m<+\infty\quad \text{for all}~s\neq 0,
\end{equation}
then $\mathcal{F},\widetilde{\mathcal{F}}\in\Delta_2$.

Due to some mathematical difficulties in \eqref{main problem@}, we cannot directly apply smooth variational techniques to study the problem \eqref{main problem@}. Indeed, if we set the energy functional $I_\varepsilon$ associated with \eqref{main problem@}, which is defined on the space $\mathbf{X}_{\varepsilon}$ as
\begin{equation}\label{eqq1.2}
	I_\varepsilon(u)=\frac{1}{p} \|u\|_{W^{1,p}_{V_\varepsilon}}^p+\frac{1}{N} \|u\|_{W^{1,N}_{V_\varepsilon}}^N-\int_{\mathbb{R}^N}\mathcal{H}(u)\,\mathrm{d}x-\int_{\mathbb{R}^N}F(u)\,\mathrm{d}x\quad\text{for all }u\in \mathbf{X}_{\varepsilon},
\end{equation}
where
\begin{align*}
	\mathcal{H}(s)=\frac{1}{N}|s|^N\log |s|^N-\frac{1}{N}|s|^N,
\end{align*}
then the energy functional $I_\varepsilon$ is not well-defined on $\mathbf{X}_{\varepsilon}$, since it may happen that there exists $u\in \mathbf{X}_{\varepsilon}$ satisfying $\int_{\mathbb{R}^N}|u|^N\log|u|^N\,\mathrm{d}x=-\infty$ and hence, $I_\varepsilon(u)=+\infty$. Inspired by the works of \cite{Alves-deMoraisFilho-2018, Alves-Ji-2020, Squassina-Szulkin-2015}, we can avoid this difficulty by choosing
\begin{align*}
	\mathcal{H}_2(s)-\mathcal{H}_1(s)=\frac{1}{N} |s|^N \log|s|^N\quad\text{for all }s\in\mathbb{R},
\end{align*}
where $\mathcal{H}_1$ is a nonnegative $C^1$-function, which is also convex and $\mathcal{H}_2$ is also a nonnegative $C^1$-function satisfying some growth condition. Hence, one can easily obtain from \eqref{eqq1.2} that
\begin{align*}
	I_\varepsilon(u)=\frac{1}{p} \|u\|_{W^{1,p}_{V_\varepsilon}}^p+\frac{1}{N} \|u\|_{W^{1,N}_{V_\varepsilon}}^N+\frac{1}{N}\|u\|^N_{N}+\int_{\mathbb{R}^N}\mathcal{H}_1(u)\,\mathrm{d}x-\int_{\mathbb{R}^N}\mathcal{H}_2(u)\,\mathrm{d}x-\int_{\mathbb{R}^N}F(u)\,\mathrm{d}x.
\end{align*}
This technique guarantees that $I_\varepsilon$ may be expressed as the combination of a $C^1$-functional with a convex and lower semicontinuous functional. Therefore, the critical point theory for lower semicontinuous functionals, as established by Szulkin  \cite{Szulkin-1986}, can be used to examine solutions of \eqref{main problem}.

Another feature and novelty of \eqref{main problem@} is the fact that the corresponding energy functional defined in \eqref{eqq1.2} is not a $C^1$-functional and hence, we shall not be able to find the multiplicity of solutions of \eqref{main problem@} by using smooth variational methods together with the Lusternik-Schnirelmann category theory. To overcome this difficulty, we shall work with a newly constructed Banach space, where the functional $I_\varepsilon$ is $C^1$. Inspired by the work of Shen-Squassina \cite{Shen-Squassina-2025}, we fix $\delta>0$ sufficiently small and define $\mathcal{H}_1$ and $\mathcal{H}_2$ by
\begin{align*}
	\mathcal{H}_1(s)=
	\begin{cases}
		~0, &\text{if } s=0,\\[1ex]
		-\displaystyle\frac{1}{N}|s|^N\log|s|^N, &\text{if }0<|s|<(N-1)\delta, \\[1ex]
		-\displaystyle\frac{1}{N}|s|^N\bigg(\log\big((N-1)\delta\big)^N+N+1 \bigg)+N\delta|s|^{N-1}-\frac{\big((N-1)\delta\big)^N}{N(N-1)},&\text{if }|s|\geq (N-1)\delta,
	\end{cases}
\end{align*}
and
\begin{align*}
	\mathcal{H}_2(s)=
	\begin{cases}
		~0,&\text{if } |s|\leq(N-1)\delta, \\[1ex]
		\displaystyle\frac{1}{N}|s|^N\log\bigg(\frac{|s|^N}{\big((N-1)\delta\big)^N} \bigg)+N\delta|s|^{N-1}-\bigg(\frac{N+1}{N}\bigg)|s|^N-\frac{\big((N-1)\delta\big)^N}{N(N-1)},&\text{if }|s|\geq (N-1)\delta,
	\end{cases}
\end{align*}
so that
\begin{equation}
	\label{eq2.6}
	\mathcal{H}_2(s)-\mathcal{H}_1(s)= \frac{1}{N}|s|^N\log|s|^N \quad\text{for all }s\in\mathbb{R}.
\end{equation}
The functions $\mathcal{H}_1$ and $\mathcal{H}_2$ have the following properties, respectively.
\begin{enumerate}
	\item[\textnormal{($\mathcal{H}_1$)}]
		\begin{enumerate}
			\item[\textnormal{(a)}]
				For $\delta>0$ small enough, $\mathcal{H}_1$ is convex, even and of class $C^1(\mathbb{R},\mathbb{R})$.
			\item[\textnormal{(b)}]
				$\mathcal{H}_1(s)\geq 0$ and $\mathcal{H}_1'(s)s\geq 0$ for all $s\in \mathbb{R}$.
			\item[\textnormal{(c)}]
				For any $q>N$, there exists constants $C_1,C_2>0$ such  that
				\begin{align*}
					|\mathcal{H}_1'(s)|\leq C_1|s|^{q-1}+C_2 \quad \text{for all }s\in\mathbb{R}.
				\end{align*}
			\item[\textnormal{(d)}]
				There exists constants $C_1,C_2>0$ such that
				\begin{align*}
					|\mathcal{H}_1(s)|\leq C_1|s|^N+ C_2 \quad\text{for all }s\in\mathbb{R}.
				\end{align*}
		\end{enumerate}
\end{enumerate}
\begin{enumerate}
	\item[\textnormal{($\mathcal{H}_2$)}]
		\begin{enumerate}
			\item[\textnormal{(a)}]
				There hold $\mathcal{H}_2'(s)\geq 0$ for all $s>0$ and $\mathcal{H}_2'(s)>0$ for all $s>(N-1)\delta$.
			\item[\textnormal{(b)}]
				$\mathcal{H}_2\in C^2(\mathbb{R},\mathbb{R})$ and for any $q>N$, there exists a constant $C=C_q>0$ such that
				\begin{align*}
					|\mathcal{H}_2'(s)|\leq C|s|^{q-1}
					\quad\text{and}\quad
					|\mathcal{H}_2(s)|\leq C|s|^{q}\quad \text{for all }s\in \mathbb{R}.
				\end{align*}
			\item[\textnormal{(c)}]
				The map $s\mapsto \displaystyle\frac{\mathcal{H}_2'(s)}{s^{N-1}}$ is nondecreasing for $s>0$ and strictly increasing for $s>(N-1) \delta$.
			\item[\textnormal{(d)}]
				$\mathcal{H}_2'$ is an odd function and there holds $\displaystyle\lim_{s\to\infty}~\frac{\mathcal{H}_2'(s)}{s^{N-1}}=\infty$.
		\end{enumerate}
\end{enumerate}
The following lemmas can be directly obtained from the classical Sobolev embedding theorem.

\begin{lemma} \label{lem2.1}
	Let \textnormal{(V1)} be satisfied. If $q\in[p,p^*]$, then the following embeddings are continuous
	\begin{align*}
		W^{1,p}_{V_\varepsilon}(\mathbb{R}^N)\hookrightarrow W^{1,p}(\mathbb{R}^N) \hookrightarrow L^q (\mathbb{R}^N)
	\end{align*}
	with $\min \{1,V_0\} \|u\|^p_{W^{1,p}}\leq \|u\|^p_{W^{1,p}_{V_\varepsilon}}$ for all $u\in W^{1,p}_{V_\varepsilon}(\mathbb{R}^N)$. Furthermore, the embedding $W^{1,p}_{V_\varepsilon}(\mathbb{R}^N)\hookrightarrow L^q _{\operatorname{loc}}(\mathbb{R}^N)$ is compact for any $q\in[1,p^*)$.
\end{lemma}

\begin{lemma}\label{lem2.2}
	Let \textnormal{(V1)} be satisfied. If $s\in[N,+\infty)$, then the following embeddings are continuous
	\begin{align*}
		W^{1,N}_{V_\varepsilon}(\mathbb{R}^N)\hookrightarrow W^{1,N}(\mathbb{R}^N) \hookrightarrow L^s (\mathbb{R}^N)
	\end{align*}
	with $\min\{1,V_0\} \|u\|^N_{W^{1,N}}\leq \|u\|^N_{W^{1,N}_{V_\varepsilon}}$ for all $u\in W^{1,N}_{V_\varepsilon}(\mathbb{R}^N)$.
	Furthermore, the embedding $W^{1,N}_{V_\varepsilon}(\mathbb{R}^N)\hookrightarrow L^s _{\operatorname{loc}}(\mathbb{R}^N)$ is compact for any $s\in[1,+\infty)$.
\end{lemma}

\begin{corollary}\label{cor2.3}
	Let \textnormal{(V1)} be satisfied. Then, in view of Lemma \ref{lem2.1} and Lemma \ref{lem2.2}, the  embeddings
	\begin{align*}
		\mathbf{X}_\varepsilon\hookrightarrow W^{1,p}(\mathbb{R}^N)\cap W^{1,N}(\mathbb{R}^N) \hookrightarrow L^\vartheta (\mathbb{R}^N)
	\end{align*}
	are continuous for any $\vartheta \in [p,p^*]\cup [N,+\infty)$. Also, the embedding $\mathbf{X}_\varepsilon\hookrightarrow L^\vartheta _{\operatorname{loc}}(\mathbb{R}^N)$ is compact for any $\vartheta\in[1,+\infty)$.
\end{corollary}

\begin{remark}\label{remark2.4}
	Let \textnormal{(V1)} be satisfied. Then, the following continuous embeddings hold:
	\begin{align*}
		\mathbf{X}_\varepsilon\hookrightarrow W^{1,t}_{V_\varepsilon}(\mathbb{R}^N)\hookrightarrow W^{1,t}(\mathbb{R}^N)\quad\text{for }t\in\{p,N\}.
	\end{align*}
\end{remark}

The next two results can be found in Shen-Squassina \cite{Shen-Squassina-2025}.

\begin{lemma}\label{lem7.89}
	The function $\mathcal{H}_1$ is an $N$-function and there holds $\mathcal{H}_1,\widetilde{\mathcal{H}}_1\in\Delta_2$. In particular, the Orlicz space $L^{\mathcal{H}_1}(\mathbb{R}^N)$ is a reflexive and separable Banach space, where
	\begin{align*}
		L^{\mathcal{H}_1}(\mathbb{R}^N)=\left\{u\in L^1_{\operatorname{loc}}(\mathbb{R}^N)\colon\int_{\mathbb{R}^N}\mathcal{H}_1\left(|u|\right)\mathrm{d}x<+\infty \right\},
	\end{align*}
	equipped with the Luxemburg norm
	\begin{align*}
		\|u\|_ {\mathcal{H}_1}=\inf\left\{\lambda>0\colon \int_{\mathbb{R}^N}\mathcal{H}_1\left(\frac{|u|}{\lambda}\right)\,\mathrm{d}x\leq 1\right\}.
	\end{align*}
\end{lemma}

\begin{corollary}
	The functional $\Theta\colon L^{\mathcal{H}_1}(\mathbb{R}^N)\to \mathbb{R} $ given by $u\mapsto \displaystyle\int_{\mathbb{R}^N}\mathcal{H}_1(u)\,\mathrm{d}x $ is of class $C^1\big(L^{\mathcal{H}_1}(\mathbb{R}^N),\mathbb{R}\big)$ with
	\begin{align*}
		\langle{\Theta'(u),v}\rangle=\int_{\mathbb{R}^N}\mathcal{H}'_1(u)v\,\mathrm{d}x \quad\text{for all }u,v\in L^{\mathcal{H}_1}(\mathbb{R}^N),
	\end{align*}
	where $L^{\mathcal{H}_1}(\mathbb{R}^N)$ denotes the Orlicz space associated with $\mathcal{H}_1$ endowed with the Luxemburg norm $\|\cdot\|_{\mathcal{H}_1}$.
\end{corollary}

\begin{remark}\label{rem2.99}
	Note that the condition \eqref{eq2.2} is satisfied by the $N$-function $\mathcal{H}_1$ with $l\in(1,N)$ and $m=N$.
\end{remark}

We define now the spaces
\begin{align*}
	\mathbf{Y}=\mathbf{X}\cap L^{\mathcal{H}_1}(\mathbb{R}^N)
	\quad\text{and}\quad
	\mathbf{Y}_\varepsilon=\mathbf{X}_\varepsilon\cap L^{\mathcal{H}_1}(\mathbb{R}^N)
\end{align*}
endowed with the norms
\begin{align*}
	\|u\|_{\mathbf{Y}}=\|u\|_{\mathbf{X}}+\|u\|_{\mathcal{H}_1}\quad\text{for all }u\in \mathbf{Y}
	\quad\text{and}\quad
	\|u\|_{\mathbf{Y}_\varepsilon}=\|u\|_{\mathbf{X}_\varepsilon}+\|u\|_{\mathcal{H}_1}\quad\text{for all }u\in \mathbf{Y}_\varepsilon.
\end{align*}
Thanks to Lemma \ref{lem7.89}, the spaces $(\mathbf{Y},\|\cdot\|_{\mathbf{Y}})$ and $(\mathbf{Y}_\varepsilon,\|\cdot\|_{\mathbf{Y}_\varepsilon})$ are reflexive and separable Banach spaces. Moreover, we have the continuous embeddings
\begin{align*}
	\mathbf{Y}_\varepsilon\hookrightarrow \mathbf{X}_\varepsilon
	\quad\text{and}\quad
	\mathbf{Y}_\varepsilon\hookrightarrow L^{\mathcal{H}_1}(\mathbb{R}^N).
\end{align*}
Similarly, these continuous embeddings are also true, if we replace $\mathbf{Y}_\varepsilon$ by $\mathbf{Y}$ and $\mathbf{X}_\varepsilon$ by $\mathbf{X}$, respectively. Next, we denote by $S_{r}$ (or by $\textbf{S}_{r}$)  the best constant in the embedding from $\mathbf{Y}_\varepsilon$ (or $\mathbf{Y}$)  into some Lebesgue space $L^r(\mathbb{R}^N)$.

\begin{lemma}\label{lem8.98}
	The embedding  $\mathbf{Y}\hookrightarrow L^\theta(\mathbb{R}^N)$ is continuous for any $\theta\in[p,p^*]\cup[N,+\infty)$. Consequently, the embedding  $\mathbf{Y}\hookrightarrow L^\theta_{\operatorname{loc}}(\mathbb{R}^N)$  is compact for any $\theta\in[1,+\infty)$. Moreover, under  \textnormal{(V1)}, these embeddings also hold true if we replace $\mathbf{Y}$ by $\mathbf{Y}_\varepsilon$.
\end{lemma}

Now, we recall the following version of Lions' compactness lemma, see Alves-Figueiredo \cite[Proposition 4]{Alves-Figueiredo-2009}.

\begin{lemma}\label{lem8.99}
Let $\{u_n\}_{n\in\mathbb{N}}\subset \mathbf{Y}_\varepsilon$ be a bounded sequence in $\mathbf{Y}_\varepsilon$ and there holds
	\begin{align*}
		\liminf_{n\to\infty} \sup_{y\in \mathbb{R}^N} \int_{B_R(y)}|u_n|^N\,\mathrm{d}x=0
	\end{align*}
	for some $R>0$, then we have $u_n\to 0$ in $L^\upsilon(\mathbb{R}^N)$ as $n\to\infty$ for any $\upsilon\in (N,+\infty)$.
\end{lemma}

The next lemma can be found in Alves-da Silva \cite{Alves-daSilva-2023}.

\begin{lemma}\label{lem2.899}
	In view of Remark \ref{rem2.99}, we have for any $u\in L^{\mathcal{H}_1}(\mathbb{R}^N)$ that
	\begin{align*}
		\min\big\{\|u\|_{\mathcal{H}_1}^l,\|u\|_{\mathcal{H}_1}^N\big\}\leq \displaystyle\int_{\mathbb{R}^N}{\mathcal{H}_1}(|u|)\,\mathrm{d}x\leq \max\big\{\|u\|_{\mathcal{H}_1}^l,\|u\|_{\mathcal{H}_1}^N\big\}.
	\end{align*}
\end{lemma}

From Yang \cite{Yang-2012}, we have the following result.

\begin{corollary}\label{cor2.9}
	The function $\Phi(t)=\exp(t)-\displaystyle\sum_{j=0}^{N-2}\frac{t^j}{j!}$ is increasing and convex in $[0,+\infty)$. Moreover,
	for any $\wp\geq 1$, $t\geq 0$ real numbers and $N\geq 2$, it holds that
	\begin{align*}
		\bigg(\exp(t)-\sum_{j=0}^{N-2}\frac{t^j}{j!}\bigg)^\wp\leq \exp(\wp t)-\sum_{j=0}^{N-2}\frac{(\wp t)^j}{j!}.
	\end{align*}
\end{corollary}

The inequality in the following lemma is known as the Trudinger-Moser inequality, which was first studied by Adimurthi-Yang \cite{Adimurthi-Yang-2010}.

\begin{lemma}\label{lem2.11}
	For all $\alpha>0$, $0\leq \beta <N$ and $u\in W^{1,N}(\mathbb{R}^N)$ with $N\geq 2$, we have $\cfrac{\Phi\big(\alpha|u|^{N'}\big)}{|x|^\beta}\in L^1(\mathbb{R}^N)$. Furthermore, we have for all $\alpha\leq\big(1-\frac{\beta}{N}\big)\alpha_N$ and $\gamma>0$
	\begin{align*}
		\sup_{\|u\|_{W^{1,N}_\gamma}\leq 1} ~\int_{\mathbb{R}^N}\cfrac{\Phi\big(\alpha|u|^{N'}\big)}{|x|^\beta}\,\mathrm{d}x<+\infty, \quad\text{where}~\|u\|_{W^{1,N}_\gamma}=\big(\|\nabla u\|_N^N+\gamma\|u\|_{N}^N\big)^{\frac{1}{N}},
	\end{align*}
	 $\Phi$ is defined as in \textnormal{(f1)} and $\alpha_N=N\omega_{N-1}^\frac{1}{N-1}$, with $\omega_{N-1}$ being the volume of the unit sphere $S^{N-1}$. Also, the above inequality is sharp for $\alpha>\big(1-\frac{\beta}{N}\big)\alpha_N$, i.e., the supremum is infinity.
\end{lemma}

\begin{remark}
	To study problem \eqref{main problem}, we are going to use Lemma \ref{lem2.11} with $\beta=0$ and $\gamma=1$.
\end{remark}

\section{Existence of solution for the penalized problem}\label{sec3}

In this section, we shall establish the existence of a solution for the penalized problem \eqref{main problem@@} by using the mountain pass theorem stated in Pucci-Serrin \cite[p.\,142]{Pucci-Serrin-1985}.

Note that for the well-definedness of the functional $I_\varepsilon$ defined in Section \ref{sec2}, we shall restrict $I_\varepsilon$ to the space $\mathbf{Y}_\varepsilon$, which will be denoted by $E_\varepsilon(u)=I_\varepsilon(u)~\text{for all}~u\in \mathbf{Y}_\varepsilon $. Hence, in view of the conditions on $\mathcal{H}_1$, $V$ and Lemma \ref{lem2.11}, the functional $I_\varepsilon$ is a $C^1$-functional on the space $\mathbf{Y}_\varepsilon$ and its Gâteaux derivative is given by
\begin{align*}
	\langle{E_\varepsilon'(u),\psi}\rangle
	&=\big\langle{u,\psi}\big\rangle_{p,V_\varepsilon}+\big\langle{u,\psi}\big\rangle_{N,V_\varepsilon}+\int_{\mathbb{R}^N}|u|^{N-2}u\psi\,\mathrm{d}x+\int_{\mathbb{R}^N}\mathcal{H}_1'(u)\psi\,\mathrm{d}x\\
	&\quad -\int_{\mathbb{R}^N}\mathcal{H}_2'(u)\psi\,\mathrm{d}x-\int_{\mathbb{R}^N}f(u)\psi\,\mathrm{d}x
\end{align*}
for all $ \psi\in\mathbf{Y}_\varepsilon,$ where $\langle{\cdot,\cdot}\rangle$ is the duality pairing between $\mathbf{Y}^*_\varepsilon$ and $\mathbf{Y}_\varepsilon$.

Let $\ell,\ell'>0$ be small enough such that $V_0+1\geq 2(\ell+\ell')$ and take $a>(N-1)\delta$ such that $\displaystyle\frac{\mathcal{H}'_2(a)}{a^{N-1}}=\ell$. Then, we define
\begin{align*}
	\widehat{\mathcal{H}}'_2(s)=
	\begin{cases}
		\mathcal{H}'_2(s),&\text{if } 0\leq s\leq a, \\
		\ell s^{N-1},&\text{if } s\geq a.
	\end{cases}
\end{align*}
Further, let $t_1>(N-1)\delta$ be such that $a\in (t_1,t_2)$ and $h\in C^1\big([t_1,t_2]\big)$ satisfying the following properties:
\begin{enumerate}
	\item[(h1)]
		$h(s)\leq \widehat{\mathcal{H}}'_2(s)$ for all $s\in [t_1,t_2]$.
	\item[(h2)]
		$h(t_i)=\widehat{\mathcal{H}}'_2(t_i)$ and $h'(t_i)=\widehat{\mathcal{H}}''_2(t_i)$ for $i\in\{1,2\}$.
	\item[(h3)]
		The map $s\mapsto \displaystyle\frac{h(s)}{s^{N-1}}$ is nondecreasing on $[t_1,t_2]$.
\end{enumerate}
Define another function
\begin{align*}
	\widetilde{\mathcal{H}}'_2(s)=
	\begin{cases}
		\widehat{\mathcal{H}}'_2(s),&\text{if } s\notin [t_1,t_2] , \\
		h(s),&\text{if } s\in [t_1,t_2].
	\end{cases}
\end{align*}
If $\displaystyle\chi_\Lambda$ denotes the characteristic function corresponding to the set $\Lambda$, then we introduce the first penalized nonlinearity $G'_2\colon\mathbb{R}^N\times [0,+\infty)\to \mathbb{R}$, which is defined by
\begin{align*}
	G'_2(x,s)=\chi_\Lambda(x)\mathcal{H}'_2(s)+(1-\displaystyle\chi_\Lambda(x))\widetilde{\mathcal{H}}'_2(s)\quad \text{for all }(x,s)\in \mathbb{R}^N\times [0,+\infty).
\end{align*}
Note that $\mathcal{H}_2'$ is an odd function and hence, we can extend the definition of $G'_2$ to $\mathbb{R}^N\times \mathbb{R}$ by setting
\begin{align*}
	G'_2(x,s)=-G'_2(x,-s) \quad\text{for all }(x,s)\in \mathbb{R}^N\times (-\infty,0].
\end{align*}
The following properties can be proved by using the definition of $G'_2$:
\begin{enumerate}
	\item[\textnormal{($\mathcal{A}$)}]
		\begin{enumerate}
			\item[\textnormal{(a)}]
				There exists a constant $C=C_q>0$ and $q>N$ such that
				\begin{align*}
					|G'_2(x,s)|\leq \ell |s|^{N-1}+C|s|^{q-1}\quad\text{for all }(x,s)\in \mathbb{R}^N\times \mathbb{R}.
				\end{align*}
			\item[\textnormal{(b)}]
				$|G'_2(x,s)|\leq |\mathcal{H}'_2(s)|$ for all $(x,s)\in \mathbb{R}^N\times \mathbb{R}$. Consequently, in view of \textnormal{($\mathcal{H}_2$)(b)}, we have
				\begin{align*}
					|G_2(x,s)|\leq C|s|^{q} \quad \text{for all }(x,s)\in \mathbb{R}^N\times \mathbb{R}.
				\end{align*}
			\item[\textnormal{(c)}]
				$|G'_2(x,s)|\leq \ell |s|^{N-1}$ for all $(x,s)\in \Lambda^c\times \mathbb{R}$.
			\item[\textnormal{(d)}]
				$\displaystyle\frac{1}{N}|s|^N+\bigg[\mathcal{H}_2(s)-\frac{1}{N}\mathcal{H}'_2(s)s+\frac{1}{N}G'_2(x,s)s-G_2(x,s)\bigg]\geq 0$ for all $(x,s)\in \mathbb{R}^N\times \mathbb{R}$.
			\item[\textnormal{(e)}]
				The map $s\mapsto\displaystyle\frac{G'_2(x,s)}{s^{N-1}}$ is nondecreasing for all $(x,s)\in\mathbb{R}^N\times
				(0,+\infty)$.
		\end{enumerate}
\end{enumerate}
On the other hand, we set $\displaystyle \frac{f(a)}{a^{N-1}}=\ell'$ and define
\begin{align*}
	\widehat{f}(s)=
	\begin{cases}
		f(s), & \text{if } 0\leq s\leq a, \\
		\ell' s^{N-1},&\text{if } s\geq a.
	\end{cases}
\end{align*}
Now, we consider the function $\eta$ such that $\eta\in C^1\big([t_1,t_2]\big)$ satisfying the following properties:
\begin{enumerate}
	\item
		[($\eta$1)] $\eta(s)\leq \widehat{f}(s)$ for all $s\in [t_1,t_2]$.
	\item[($\eta$2)]
		$\eta(t_i)=\widehat{f}(t_i)$ and $\eta'(t_i)={\widehat{f}}'(t_i)$ for $i\in\{1,2\}$.
	\item[($\eta$3)]
		The map $s\mapsto \displaystyle\frac{\eta(s)}{s^{N-1}}$ is nondecreasing on $[t_1,t_2]$.
\end{enumerate}
Further, we define another function
\begin{align*}
	\widetilde{f}(s)=
	\begin{cases}
		\widehat{f}(s),&\text{if } s\notin [t_1,t_2] , \\
		\eta(s),&\text{if } s\in [t_1,t_2].
	\end{cases}
\end{align*}
Next, we introduce the second penalized nonlinearity  $g\colon\mathbb{R}^N\times [0,+\infty)\to \mathbb{R}$, which is defined by
\begin{align*}
	g(x,s)=\displaystyle\chi_\Lambda(x)f(s)+(1-\displaystyle\chi_\Lambda(x))\widetilde{f}(s) \quad\text{for all }(x,s)\in \mathbb{R}^N\times [0,+\infty).
\end{align*}
Observe that $f$ is an odd function, therefore we can extend the definition of $g$ to $\mathbb{R}^N\times \mathbb{R}$ by setting
\begin{align*}
	g(x,s)=-g(x,-s)\quad\text{for all }(x,s)\in \mathbb{R}^N\times (-\infty,0].
\end{align*}
Moreover, we also define
\begin{align*}
	\Lambda_\varepsilon&=\{x\in \mathbb{R}^N\colon\varepsilon x\in \Lambda\},\\
	G_2(x,s)&=\int_{0}^{s}G'_2(x,t)\,\mathrm{d}t\quad\text{for all }(x,s)\in \mathbb{R}^N\times \mathbb{R},\\
	\mathcal{G}(x,s)&=\int_{0}^{s}g(x,t)\,\mathrm{d}t \quad\text{for all }(x,s)\in \mathbb{R}^N\times \mathbb{R}.
\end{align*}
By using the definition of $g$, one can prove the following properties:
\begin{enumerate}
	\item[\textnormal{($\mathcal{B}$)}]
	\begin{enumerate}
		\item[\textnormal{(a)}]
			$g(x,s)\leq 0$  for all $(x,s)\in \mathbb{R}^N\times (-\infty,0]$ and $g(x,s)\geq 0$ for all $(x,s)\in \mathbb{R}^N\times [0,+\infty)$.
		\item[\textnormal{(b)}]
			$ |g(x,s)|\leq |f(s)|$ for all $(x,s)\in \mathbb{R}^N\times \mathbb{R}$.
		\item[\textnormal{(c)}]
			$\mu \mathcal{G}(x,s)\leq sg(x,s)$ for all $(x,s)\in \Lambda\times \mathbb{R}$.
		\item[\textnormal{(d)}]
			$N \mathcal{G}(x,s)\leq sg(x,s)\leq \ell'|s|^N$ for all $(x,s)\in \Lambda^c\times \mathbb{R}$.
		\item[\textnormal{(e)}]
			The map $s\mapsto\displaystyle\frac{g(x,s)}{|s|^{N-2}s}$ is nondecreasing for all $(x,s)\in\mathbb{R}^N\times \big(\mathbb{R}\setminus\{0\}\big)$.
	\end{enumerate}
\end{enumerate}
Since our goal is to study the positive solutions of \eqref{main problem}, we deal with the following penalized problem:
\begin{align}\label{main problem@@}\tag{$\widetilde{\mathcal{S}}_{\varepsilon}$}
	\begin{cases}
		\widetilde{\mathcal{L}}_{p_\varepsilon}(u)+\widetilde{\mathcal{L}}_{N_\varepsilon}(u)+|u|^{N-2}u=G'_2(\varepsilon x,u)-\mathcal{H}'_1(u)+g(\varepsilon x,u)\quad\text{in }\mathbb{R}^N, \\[1ex]
		\displaystyle ~\int_{\mathbb{R}^N} V(\varepsilon x)\big(|u|^p+|u|^N\big)\,\mathrm{d}x<+\infty,~ u\in \mathbf{Y}_\varepsilon.
	\end{cases}
\end{align}
Note that if $u_\varepsilon$ is a positive solution of \eqref{main problem@@} with $0<u_\varepsilon(x)<t_1$ for all $x\in \mathbb{R}^N\setminus \Lambda_\varepsilon $, then $G'_2(\varepsilon x,u_\varepsilon)=\mathcal{H}'_2(u_\varepsilon)$ and $g(\varepsilon x,u_\varepsilon)=f(u_\varepsilon)$. Due to this fact, we conclude that $v_\varepsilon(x)=u_\varepsilon(\frac{x}{\varepsilon})$ is a positive solution of \eqref{main problem}. Now, we define the functional $J_\varepsilon\colon\mathbf{Y}_\varepsilon\to \mathbb{R}$ associated with \eqref{main problem@@} by
\begin{equation}\label{eq3.1}
	\begin{aligned}
		J_\varepsilon(u)
		&=\frac{1}{p} \|u\|_{W^{1,p}_{V_\varepsilon}}^p+\frac{1}{N} \Big(\|u\|_{W^{1,N}_{V_\varepsilon}}^N+\|u\|^N_{N}\Big)+\int_{\mathbb{R}^N}\mathcal{H}_1(u)\,\mathrm{d}x\\
		&\quad -\int_{\mathbb{R}^N}G_2(\varepsilon x,u)\,\mathrm{d}x-\int_{\mathbb{R}^N}\mathcal{G}(\varepsilon x,u)\,\mathrm{d}x \quad\text{for all }u\in\mathbf{Y}_\varepsilon.
	\end{aligned}
\end{equation}
Clearly, $J_\varepsilon$ is well-defined, of class $C^1(\mathbf{Y}_\varepsilon,\mathbb{R})$ and the critical points of $J_\varepsilon$ are weak solutions of \eqref{main problem@@}.
Note that, from the assumptions (f1), (f2) and Corollary \ref{cor2.9}, one can easily verify that for any $\tau>0$, $\vartheta> N$, there exists a constant $\widetilde{\kappa}_\tau>0$ and $0<\alpha_0<\alpha$ such that for all $s\in \mathbb{R}$, we have
\begin{align}\label{eq3.2}
	\begin{aligned}
		|f(s)|&\leq \tau|s|^{N-1}+\widetilde{\kappa}_\tau|s|^{\vartheta-1}\Phi(\alpha |s|^{N'}), \\
		|F(s)|&\leq \tau |s|^{N}+\widetilde{\kappa}_\tau|s|^\vartheta\Phi(\alpha |s|^{N'}).
	\end{aligned}
\end{align}

The following two lemmas show that the functional $J_\varepsilon$ satisfies the mountain pass geometry.

\begin{lemma}[Mountain Pass Geometry-I] \label{lem3.1}
	There exist $\rho\in(0,1]$ very small enough and $\jmath>0$ such that $J_\varepsilon(u)\geq \jmath$ for all $u\in \mathbf{Y}_\varepsilon$ with $\|u\|_{\mathbf{Y}_\varepsilon}=\rho$ .
\end{lemma}

\begin{proof}
	From ($\mathcal{B}$)(b) and \eqref{eq3.2}, we have
	\begin{equation} \label{eq3.3}
		|\mathcal{G}(x,s)|\leq \tau |s|^{N}+\widetilde{\kappa}_\tau|s|^\vartheta\Phi(\alpha |s|^{N'})\quad \text{for all } (x,s)\in \mathbb{R}^N\times \mathbb{R}.
	\end{equation}
	Let $\sigma\in(0,1]$ be sufficiently small such that $0<\|u\|_{\mathbf{Y}_\varepsilon}\leq\sigma$. Further, we choose $r, r'>1$ satisfying $\frac{1}{r}+\frac{1}{r'}=1$, then, by using H\"older's inequality, Corollary \ref{cor2.9} and Lemma \ref{lem8.98}, we obtain from \eqref{eq3.3} that
	\begin{align}\label{eq3.4}
		\int_{\mathbb{R}^N}|\mathcal{G}(x,u)|\,\mathrm{d}x
		\leq \tau S^{-N}_N\|u\|^N_{\mathbf{Y}_\varepsilon}+\widetilde{\kappa}_\tau S^{-\vartheta}_{\vartheta r} \|u\|^\vartheta_{\mathbf{Y}_\varepsilon}\Bigg(\int_{\mathbb{R}^N}\Phi\big(r'\alpha \|u\|^{N'}_{\mathbf{Y}_\varepsilon}|\widetilde{u}|^{N'}\big)\,\mathrm{d}x\Bigg)^{\frac{1}{r'}},
	\end{align}
	where $\widetilde{u}=u/\|u\|_{\mathbf{Y}_\varepsilon}$. Since  $\|u\|_{\mathbf{Y}_\varepsilon}$ is very small, we can choose $r'>1$ close to 1 and $\alpha>\alpha_0$ close to $\alpha_0$ such that $r'\alpha\|u\|_{\mathbf{Y}_\varepsilon}^{N'}\leq\alpha_N$ holds. Take $\gamma>0$ with $\gamma\leq V_0$, then we have $\|\widetilde{u}\|_{W^{1,N}_\gamma}\leq \|\widetilde{u}\|_{{W^{1,N}_{V_\varepsilon}}}\leq\|\widetilde{u}\|_{\mathbf{Y}_\varepsilon}=1$. Consequently, by using \eqref{eq3.4} and Lemma \ref{lem2.11}, there exists a constant $\widetilde{C}>0$ such that
	\begin{align}\label{eq3.5}
		\int_{\mathbb{R}^N}|\mathcal{G}(x,u)|\,\mathrm{d}x & \leq \tau S^{-N}_N\|u\|^N_{\mathbf{Y}_\varepsilon}+\widetilde{\kappa}_\tau \widetilde{C} S^{-\vartheta}_{\vartheta r} \|u\|^\vartheta_{\mathbf{Y}_\varepsilon}.
	\end{align}
	Note that $\|u\|_{\mathbf{Y}_\varepsilon}\leq 1$, therefore it follows from ($\mathcal{A}$)(b) and Lemma \ref{lem8.98} that
	\begin{align} \label{eq3.6}
		\int_{\mathbb{R}^N}|G_2(x,u)|\,\mathrm{d}x & \leq CS^{-q}_q \|u\|^q_{\mathbf{Y}_\varepsilon}\leq CS^{-q}_q \|u\|^N_{\mathbf{Y}_\varepsilon}.
	\end{align}
	Hence, in view of Lemma \ref{lem2.899}, we obtain from \eqref{eq3.1}, \eqref{eq3.5} and \eqref{eq3.6} that
	\begin{align*}
		J_{\varepsilon}(u)
		& \geq \frac{1}{N}\big(\|u\|_{W^{1,p}_{V_\varepsilon}}^N+\|u\|_{W^{1,N}_{V_\varepsilon}}^{N}+\|u\|^N_{\mathcal{H}_1}\big)-CS^{-q}_q \|u\|^N_{\mathbf{Y}_\varepsilon}-\tau S^{-N}_N\|u\|^N_{\mathbf{Y}_\varepsilon}-\widetilde{\kappa}_\tau \widetilde{C} S^{-\vartheta}_{\vartheta r} \|u\|^\vartheta_{\mathbf{Y}_\varepsilon} \\
		& \geq \bigg(\frac{3^{1-N}}{N}-\big( CS^{-q}_q+\tau S^{-N}_N \big)\bigg)\|u\|^N_{\mathbf{Y}_\varepsilon}-C_\vartheta \|u\|^\vartheta_{\mathbf{Y}_\varepsilon},
	\end{align*}
	for all $u\in \mathbf{Y}_\varepsilon $ satisfying $\|u\|_{\mathbf{Y}_\varepsilon}\leq \sigma$, where $C_\vartheta=\widetilde{\kappa}_\tau \widetilde{C} S^{-\vartheta}_{\vartheta r}$. Choose $ CS^{-q}_q+\tau S^{-N}_N=\displaystyle\frac{2  }{3^{N}N}$ and define the function
	\begin{align*}
		\psi(t)=\frac{1}{3^{N+1}N}t^N-C_\vartheta t^\vartheta\quad\text{for all }t\in[0,\sigma].
	\end{align*}
	By using elementary calculus, we infer that $\psi$ admits a positive maximum $\jmath$ in $[0,\sigma]$ at a point $\rho\in(0,\sigma]$. Moreover, for all $u\in \mathbf{Y}_\varepsilon $ satisfying $\|u\|_{\mathbf{Y}_\varepsilon}=\rho$, we obtain
	\begin{align*}
		J_{\varepsilon}(u)\geq \frac{1}{3^{N}N}\rho^N-C_\vartheta \rho^\vartheta\geq \psi(\rho)=\jmath>0.
	\end{align*}
	This completes the proof.
\end{proof}

\begin{lemma}[Mountain Pass Geometry-II] \label{lem3.2}
	Let $\rho\in(0,1]$ be as in Lemma \ref{lem3.1}, then there exists a nonnegative function $e\in\mathbf{Y}_\varepsilon$ with $\|e\|_{\mathbf{Y}_\varepsilon}>\rho$ satisfying  $J_{\varepsilon}(e)<0$.
\end{lemma}

\begin{proof}
	Define
	\begin{align*}
		\mathcal{O}_\varepsilon=\ \left\{u\in\mathbf{Y}_\varepsilon\colon|\operatorname{supp}(|u|)\cap \Lambda_\varepsilon|>0\right\}
	\end{align*}
	and let $u\in \mathcal{O}_\varepsilon\setminus\{0\} $ be such that $u(x)\geq 0$ a.e.\,in $\mathbb{R}^N$. Define
	\begin{align*}
		\Psi(t)=t^{-\mu} \mathcal{G}(x,tu
		)-\mathcal{G}(x,u)\quad\text{for all }x\in \Lambda~\text{and}~t\geq1.
	\end{align*}
	From ($\mathcal{B}$)(c), we have
	\begin{align*}
		\Psi'(t)=t^{-\mu-1}\big( g(x,tu)tu-\mu G(x,tu)\big)\geq 0\quad\text{for all } x\in \Lambda \text{ and for all } t\geq1.
	\end{align*}
	It follows that $\Psi$ is an increasing function on $[1,+\infty)$. Hence, we obtain $G(x,tu)\geq t^\mu G(x,u)$ for all $x\in \Lambda$ and $t\geq1$. Further, since $\mu>N$ and $u(x)\geq 0$ a.e.\,in $\mathbb{R}^N$, we infer that $ G(x,tu)\geq t^N G(x,u)$ for all $x\in \Lambda$ and $t\geq1$. Similarly, by using ($\mathcal{B}$)(d), we can prove that $ G(x,tu)\geq t^N G(x,u)$ for all $x\in \Lambda^c$ and $t\geq 1$. Consequently, we obtain
	\begin{align} \label{eq3.10}
		G(x,tu)\geq t^N G(x,u)\quad\text{for all } x\in \mathbb{R}^N \text{ and  for all } t\geq1.
	\end{align}
	Notice that for each $x\in \mathbb{R}^N $ and $t>0$, we can write
	\begin{align*}
		\mathcal{H}_1(tu)=\chi_{\Lambda_\varepsilon}(x) \mathcal{H}_1(tu)+(1-\chi_{\Lambda_\varepsilon}(x))\mathcal{H}_1(tu).
	\end{align*}
	Now, by using the definition of $G_2$ and \eqref{eq2.6}, it follows that
	\begin{equation}\label{eq3.11}
		\begin{aligned}
			&\int_{\mathbb{R}^N}\big(\mathcal{H}_1(tu)-G_2(\varepsilon x,tu)\big)\,\mathrm{d}x \\
			& = -\frac{1}{N} \int_{\mathbb{R}^N}\chi_{\Lambda_\varepsilon}|tu|^N\log|tu|^N\,\mathrm{d}x-\frac{1}{N} \int_{\{tu\leq t_1\}}(1-\chi_{\Lambda_\varepsilon})|tu|^N\log|tu|^N\,\mathrm{d}x  \\
			&\quad + \int_{\{tu> t_1\}}(1-\chi_{\Lambda_\varepsilon})\big(\mathcal{H}_1(tu)-\widetilde{\mathcal{H}}_2(tu)\big)\,\mathrm{d}x.
		\end{aligned}
	\end{equation}
	By the definition of $\widetilde{\mathcal{H}}_2$ and ($\mathcal{H}_1$)(d), we have
	\begin{align*}
		\widetilde{\mathcal{H}}_2(s)\geq 0
		\quad \text{and}\quad
		0\leq \mathcal{H}_1(s)\leq C_1 s^N+C_2\quad \text{for all }s\geq 0.
	\end{align*}
	In addition, it follows from $u\in\mathbf{Y}_\varepsilon$ that
	\begin{align*}
		\int_{\{tu> t_1\}} |tu|^N\,\mathrm{d}x\leq  Kt^N
		\quad\text{and}
		\quad
		|\{tu> t_1\}|\leq K_1 t^N,
	\end{align*}
	where $K=\|u\|^N_N$ and $K_1=K t_1^{-N}$.  Consequently, from the above information, we can deduce that
	\begin{equation}
		\label{eq3.12}
		\int_{\{tu> t_1\}}(1-\chi_{\Lambda_\varepsilon})\big(\mathcal{H}_1(tu)-\widetilde{\mathcal{H}}_2(tu)\big)\,\mathrm{d}x\leq Ct^N,
	\end{equation}
	where $C=C_1K+C_2K_1>0$. Let $t\geq1$ be large enough, then by using \eqref{eq3.10}, \eqref{eq3.11} and \eqref{eq3.12}, we obtain from \eqref{eq3.1} that
	\begin{equation}\label{eq3.13}
		\begin{aligned}
			J_\varepsilon(tu)
			& \leq t^N\Bigg[\frac{1}{p}\bigg(\|u\|_{W^{1,p}_{V_\varepsilon}}^N+\|u\|_{W^{1,N}_{V_\varepsilon}}^{N}+\|u\|^N_N-p\int_{\mathbb{R}^N}\mathcal{G}(\varepsilon x,u)\,\mathrm{d}x\bigg)-\frac{1}{N}\int_{\mathbb{R}^N} \chi_{\Lambda_\varepsilon}|u|^N \log|u|^N\,\mathrm{d}x\\
			&\qquad\qquad  -\log t\bigg(\int_{\mathbb{R}^N} \chi_{\Lambda_\varepsilon}|u|^N\,\mathrm{d}x +\int_{\{tu\leq t_1\}}(1-\chi_{\Lambda_\varepsilon})|u|^N\,\mathrm{d}x\bigg)\\
			&\qquad\qquad -\frac{1}{N} \int_{\{tu\leq t_1\}}(1-\chi_{\Lambda_\varepsilon})|u|^N\log|u|^N\,\mathrm{d}x+C\Bigg].
		\end{aligned}
	\end{equation}
	Due to the application of Lebesgue's dominated convergence theorem, we obtain
	\begin{equation}\label{eq3.14}
		\int_{\{tu\leq t_1\}}(1-\chi_{\Lambda_\varepsilon})|u|^N\,\mathrm{d}x\to 0\quad\text{as }t\to\infty.
	\end{equation}
	Now, by using \eqref{eq2.6}, ($\mathcal{H}_2$)(b), Lemma \ref{lem8.98} and the fact that $u\in L^{\mathcal{H}_1}(\mathbb{R}^N)$, we have
	\begin{align}\label{eq3.16}
		\Bigg|-\frac{1}{N} \int_{\{tu\leq t_1\}}(1-\chi_{\Lambda_\varepsilon})|u|^N\log|u|^N\,\mathrm{d}x\Bigg| & \leq \int_{\mathbb{R}^N} |\mathcal{H}_1(u)|\,\mathrm{d}x+\int_{\mathbb{R}^N} |\mathcal{H}_2(u)|\,\mathrm{d}x< +\infty.
	\end{align}
	Sending $t\to\infty$ in \eqref{eq3.13} and using \eqref{eq3.14} as well as \eqref{eq3.16}, we deduce that $J_\varepsilon(tu)\to-\infty$ as $t\to\infty$. Taking $e=tu$ with $t$ sufficiently large shows that assertion of the lemma.
\end{proof}

To use the mountain pass theorem, it is necessary to verify the Palais-Smale compactness condition at a suitable level $c$. We say that a sequence $\{u_n\}_{n\in\mathbb{N}}\subset\mathbf{Y}_\varepsilon$ is a \textnormal{(PS)$_c$} sequence for $J_{\varepsilon}$ at any suitable level $c\in\mathbb{R}$ if
\begin{equation}
	\label{eq3.17}
	J_{\varepsilon}(u_n)\to c
	\quad\text{and}\quad
	\sup_{\|\varphi\|_{\mathbf{Y}_\varepsilon}=1}|\langle{J'_{\varepsilon}(u_n),\varphi}\rangle|\to 0\quad\text{as}~n\to\infty.
\end{equation}
If this sequence has a convergent subsequence in $\mathbf{Y}_\varepsilon$, we say that $J_{\varepsilon}$ satisfies \textnormal{(PS)$_c$} condition at any suitable level $c\in\mathbb{R}$.

Note that the following logarithmic inequality is useful to prove the boundedness of \textnormal{(PS)$_c$} sequences for $J_{\varepsilon}$.

\begin{lemma}\label{lem3.3}
	Let $N\geq 2$ with $N<s$ and $u\in L^N(\Lambda_\varepsilon)\cap L^s(\Lambda_\varepsilon)$. Then, we have
	\begin{align*}
		\int_{\Lambda_\varepsilon} |u|^N \log\bigg(\displaystyle\frac{|u|}{\|u\|_{L^N(\Lambda_\varepsilon)}}\bigg)\,\mathrm{d}x\leq \frac{s}{s-N}\|u\|_{L^N(\Lambda_\varepsilon)}\log\bigg(\displaystyle\frac{\|u\|_{L^s(\Lambda_\varepsilon)}}{\|u\|_{L^N(\Lambda_\varepsilon)}} \bigg).
	\end{align*}
\end{lemma}
\begin{proof}
	By using the logarithmic interpolation inequality in Del Pino-Dolbeault \cite[p.\,153]{DelPino-Dolbeault-2003} and applying a similar procedure as in Alves-Ambrosio \cite[Lemma 3.2]{Alves-Ambrosio-2024}, the lemma can be proved.
\end{proof}

\begin{remark}
    To study quasilinear elliptic problems involving $N$-Laplace operator, generally one requires the following condition
    \begin{equation}\label{eq3.999}
      \limsup_{n\to\infty} \|u_n\|^{N'}_{W^{1,N}}<\frac{\alpha_N}{\alpha_0}
    \end{equation}
    to handle terms containing critical exponential growth. Such type of inequality originates by default from the well-known Palais-Smale condition or Cerami condition at a suitable level $c$.

Whereas, in this article, due to the presence of logarithmic nonlinearity, by using the \textnormal{(PS)$_c$} condition it is difficult to directly prove the boundedness of the corresponding Palais-Smale sequence as well as we can not obtain inequality of type \eqref{eq3.999}. To show the boundedness of such type of sequence, we need the inequality \eqref{eq3.999} along with some extra computational work. Moreover, one can see that the inequality \eqref{eq3.999} can be assumed because of Lemma \ref{lem2.2} due to the fact that the bounedness of $\|u_n\|_{W^{1,N}}$ does not imply the boundedness of $\|u_n\|_{W^{1,N}_{V_\varepsilon}}$ in general.
\end{remark}

\begin{lemma}\label{lem3.4}
	Let $\{u_n\}_{n\in\mathbb{N}}\subset \mathbf{Y}_\varepsilon$ be a \textnormal{(PS)$_c$} sequence for $J_{\varepsilon}$ satisfying $\displaystyle\limsup_{n\to\infty} \|u_n\|^{N'}_{W^{1,N}}<\frac{\alpha_N}{\alpha_0}$. Then, the sequence $\{u_n\}_{n\in\mathbb{N}}$ is bounded in $ \mathbf{Y}_\varepsilon$.
\end{lemma}

\begin{proof}
	Let $\{u_n\}_{n\in\mathbb{N}}\subset \mathbf{Y}_\varepsilon$ be a \textnormal{(PS)$_c$} sequence for $J_{\varepsilon}$ and let $\displaystyle\limsup_{n\to\infty} \|u_n\|^{N'}_{W^{1,N}}<\frac{\alpha_N}{\alpha_0}$ be satisfied. Now, by using \eqref{eq3.17} and the properties of ($\mathcal{B}$), there exists $d>0$ such that, as $n\to\infty$, we have
	\begin{equation} \label{eq3.18}
		\begin{aligned}
			&c+d\|u_n\|_{\mathbf{Y}_\varepsilon}+o_n(1)\\
			& \geq J_\varepsilon(u_n)-\frac{1}{N}\langle{J'_\varepsilon(u_n),u_n}\rangle\\
			& \geq  \bigg(\frac{1}{p}-\frac{1}{N}\bigg)\|u_n\|^p_{W^{1,p}_{V_\varepsilon}}+\int_{\mathbb{R}^N} \Big(\mathcal{H}_1(u_n)-\frac{1}{N}\mathcal{H}'_1(u_n)u_n+\frac{1}{N}G_2'(\varepsilon x,u_n)u_n-G_2(\varepsilon x,u_n)\Big)\,\mathrm{d}x  \\
			& \geq \frac{1}{N} \int_{\mathbb{R}^N}|u_n|^N\,\mathrm{d}x+\int_{\mathbb{R}^N} \Big(\mathcal{H}_2(u_n)-\frac{1}{N}\mathcal{H}'_2(u_n)u_n+\frac{1}{N}G_2'(\varepsilon x,u_n)u_n-G_2(\varepsilon x,u_n)\Big)\,\mathrm{d}x                      \\
			& =:\mathcal{T}(u_n),
		\end{aligned}
	\end{equation}
	where we have used
	\begin{align*}
		\int_{\mathbb{R}^N} \Big(\mathcal{H}_1(u_n)-\frac{1}{N}\mathcal{H}'_1(u_n)u_n+\frac{1}{N}\mathcal{H}'_2(u_n)u_n-\mathcal{H}_2(u_n)\Big)\,\mathrm{d}x=\frac{1}{N} \int_{\mathbb{R}^N}|u_n|^N\,\mathrm{d}x.
	\end{align*}
	Observe that
	\begin{align*}
		\mathbb{R}^N=\big(\Lambda_\varepsilon\cup \{|u_n|\leq t_1\} \big) \cup \big(\Lambda_\varepsilon^c\cap\{|u_n|>t_1\}\big).
	\end{align*}
	Consequently, by using ($\mathcal{A}$)(d), one can easily deduce that
	\begin{equation}\label{eq3.19}
		\begin{aligned}
			\mathcal{T}(u_n) & \geq \frac{1}{N} \int_{\Lambda_\varepsilon}|u_n|^N\,\mathrm{d}x+\int_{\Lambda_\varepsilon^c\cap\{|u_n|>t_1\}} \Big(\frac{1}{N}|u_n|^N+\mathcal{H}_2(u_n)-\frac{1}{N}\mathcal{H}'_2(u_n)u_n\\
			&\hspace{7cm}+\frac{1}{N}G_2'(\varepsilon x,u_n)u_n-G_2(\varepsilon x,u_n)\Big)\,\mathrm{d}x  \\
			& \geq \frac{1}{N} \int_{\Lambda_\varepsilon}|u_n|^N\,\mathrm{d}x.
		\end{aligned}
	\end{equation}
	Combining \eqref{eq3.18} and \eqref{eq3.19}, we obtain
	\begin{equation}\label{eq3.20}
		\frac{1}{N} \int_{\Lambda_\varepsilon}|u_n|^N\,\mathrm{d}x\leq c+d\|u_n\|_{\mathbf{Y}_\varepsilon}+o_n(1)\quad \text{as }n\to\infty.
	\end{equation}
	By using ($\mathcal{H}_1$)(d) and \eqref{eq3.20}, we can find constants $\tilde{c},\tilde{d}>0$ such that
	\begin{equation}
		\label{eq3.21}
		\int_{\Lambda_\varepsilon} \mathcal{H}_1(u_n)\,\mathrm{d}x\leq \tilde{c}+\tilde{d} \|u_n\|_{\mathbf{Y}_\varepsilon}+o_n(1)\quad\text{as }n\to\infty.
	\end{equation}
	On the other hand, by using Lemma \ref{lem3.3} with $D=\displaystyle\frac{s}{s-N}$ and Lemma \ref{lem8.98}, we have
	\begin{equation}	\label{eq3.22}
		\begin{aligned}
			&\int_{\Lambda_\varepsilon} |u_n|^N \log |u_n|^N \,\mathrm{d}x\\
			& \leq \big(N\|u_n\|^N_{L^N(\Lambda_\varepsilon)}-DN\|u_n\|_{L^N(\Lambda_\varepsilon)}\big) \log \big(\|u_n\|_{L^N(\Lambda_\varepsilon)}\big)+DN\|u_n\|_{L^N(\Lambda_\varepsilon)}\log \big(\|u_n\|_{L^s(\Lambda_\varepsilon)}\big)                                    \\
			& \leq \big(N\|u_n\|^N_{L^N(\Lambda_\varepsilon)}+DN\|u_n\|_{L^N(\Lambda_\varepsilon)}\big) \big|\log \big(\|u_n\|_{L^N(\Lambda_\varepsilon)}\big)\big|+\widetilde{D}\|u_n\|_{\mathbf{Y}_\varepsilon}\big|\log \big(\widetilde{D}\|u_n\|_{\mathbf{Y}_\varepsilon}\big)\big|,
		\end{aligned}
	\end{equation}
	where the constant $\widetilde{D}>0$ is independent of $\varepsilon$. Recall that in view of Lemma 4.2 of Alves-da Silva \cite{Alves-daSilva-2023}, there exists $\xi\in(0,1)$ and a constant $A>0$ such that
	\begin{equation} \label{eq3.23}
		|t\log t|\leq A(1+t^{1+\xi})\quad\text{for all }t\geq 0.
	\end{equation}
	Due to the above inequality and \eqref{eq3.20}, we obtain
	\begin{equation} \label{eq3.24}
		\|u_n\|^N_{L^N(\Lambda_\varepsilon)} \big|\log \big(\|u_n\|_{L^N(\Lambda_\varepsilon)}\big)\big|\leq \frac{A}{N}\Big(1+\big(\|u_n\|^N_{L^N(\Lambda_\varepsilon)}\big)^{1+\xi}\Big)\leq \widetilde{A}_1\big(1+\|u_n\|^{1+\xi}_{\mathbf{Y}_\varepsilon}\big)+o_n(1)
	\end{equation}
	as $n\to\infty$. Likewise, by using \eqref{eq3.23} and Lemma \ref{lem8.98}, we have
	\begin{equation}\label{eq3.25}
		\begin{aligned}
			\|u_n\|_{L^N(\Lambda_\varepsilon)} \big|\log \big(\|u_n\|_{L^N(\Lambda_\varepsilon)}\big)\big| &\leq A\Big(1+\big(\|u_n\|_{L^N(\Lambda_\varepsilon)}\big)^{1+\xi}\Big)\leq \widetilde{A}_2\big(1+\|u_n\|^{1+\xi}_{\mathbf{Y}_\varepsilon}\big), \\
			\widetilde{D}\|u_n\|_{\mathbf{Y}_\varepsilon}\big|\log \big(\widetilde{D}\|u_n\|_{\mathbf{Y}_\varepsilon}\big)\big|&\leq A \Big(1+\big(\widetilde{D}\|u_n\|_{\mathbf{Y}_\varepsilon}\big)^{1+\xi}\Big))\leq \widetilde{A}_3\big(1+\|u_n\|^{1+\xi}_{\mathbf{Y}_\varepsilon}\big),
		\end{aligned}
	\end{equation}
	where in the above estimates $\widetilde{A}_i$ with $i=1,2,3$ are positive constants. Now, by using \eqref{eq3.24} and \eqref{eq3.25} in \eqref{eq3.22}, we deduce that there exists a constant $\widetilde{A}>0$ such that
	\begin{equation}
		\label{eq3.26}
		\int_{\Lambda_\varepsilon} |u_n|^N \log |u_n|^N \,\mathrm{d}x \leq {\widetilde{A}} \big(1+\|u_n\|^{1+\xi}_{\mathbf{Y}_\varepsilon}\big)+o_n(1)\quad\text{as }n\to\infty.
	\end{equation}
	In virtue of ($\mathcal{B}$)(b), (f1) and Corollary \ref{cor2.9}, we can deduce that
	\begin{equation}\label{eq3.27}
		|\mathcal{G}(x,s)|\leq \tau |s|^{N}+\kappa_\tau|s|\Phi(\alpha |s|^{N'})\quad\text{for all }(x,s)\in \mathbb{R}^N\times \mathbb{R}.
	\end{equation}
	From the assumption, it follows that there exists $m_0>0$ and $n_0\in\mathbb{N}$ sufficiently large such that
	\begin{align*}
		\|u_n\|^{N'}_{W^{1,N}}<m_0<\frac{\alpha_N}{\alpha_0}\quad\text{for all }n\geq n_0.
	\end{align*}
	Take $r\geq N$ with $r'=\frac{r}{r-1}>1$ and satisfying $\frac{1}{r}+\frac{1}{r'}=1$ . Let $r'$ close to 1 and $\alpha>\alpha_0$ close to $\alpha_0$ such that we still have $r'\alpha\|u_n\|^{N'}_{W^{1,N}}<m_0<\alpha_N$ for all $n\geq n_0$. Consequently, by applying H\"older's inequality, Corollary \ref{cor2.9} and Lemma \ref{lem8.98}, we obtain from \eqref{eq3.27} that
	\begin{equation}\label{eq3.28}
		\begin{aligned}
			\int_{\mathbb{R}^N}|\mathcal{G}(\varepsilon x,u_n)|\,\mathrm{d}x & \leq \tau \|u_n\|^N_N+\kappa_\tau  \|u_n\|_r\Bigg(\int_{\mathbb{R}^N}\Phi\big(r'\alpha \|u_n\|^{N'}_{W^{1,N}}|\widetilde{u}_n|^{N'}\big)\,\mathrm{d}x\Bigg)^{\frac{1}{r'}}  \\
			& \leq \tau \|u_n\|^N_N +\kappa_\tau C S^{-1}_r \|u_n\|_{\mathbf{Y}_\varepsilon}
		\end{aligned}
	\end{equation}
	for $n$ large enough, where $\widetilde{u}_n = \frac{u_n}{\|u_n\|_{W^{1,N}}}$. Moreover, due to Lemma \ref{lem2.11}, we have
	\begin{align*}
		C=\Bigg(\sup_{n\geq n_0}\int_{\mathbb{R}^N}\Phi\big(r'\alpha\|u_n\|^{N'}_{W^{1,N}}|\widetilde{u}_n|^{N'}\big)\,\mathrm{d}x\Bigg)^{\frac{1}{r'}}<+\infty.
	\end{align*}
	By using \eqref{eq2.6}, \eqref{eq3.21}, \eqref{eq3.26} and the growth $|G_2(\varepsilon x,s)|\leq \ell |s|^{N}$ for all $( x,s)\in \Lambda_\varepsilon^c\times \mathbb{R}$, we have, as $n\to\infty$, that
	\begin{equation}\label{eq3.29}
		\begin{aligned}
			&\int_{\mathbb{R}^N} \big(\mathcal{H}_1(u_n)-G_2(\varepsilon x,u_n)\big)\,\mathrm{d}x\\
			& \geq  \int_{\Lambda^c_\varepsilon} \mathcal{H}_1(u_n)\,\mathrm{d}x-\ell \|u_n\|^N_N-\frac{1}{N} \int_{\Lambda_\varepsilon} |u_n|^N \log |u_n|^N \,\mathrm{d}x  \\
			& \geq \int_{\mathbb{R}^N} \mathcal{H}_1(u_n)\,\mathrm{d}x-(\tilde{c}+\tilde{d} \|u_n\|_{\mathbf{Y}_\varepsilon})- \frac{\widetilde{A}}{N}\big(1+\|u_n\|^{1+\xi}_{\mathbf{Y}_\varepsilon}\big)-\ell \|u_n\|^N_N+o_n(1).
		\end{aligned}
	\end{equation}
	In view of \eqref{eq3.28}, \eqref{eq3.29} and the fact that $\{u_n\}_{n\in\mathbb{N}}$ is a \textnormal{(PS)$_c$} sequence for $J_{\varepsilon}$, we obtain, as $n\to\infty$, that
	\begin{equation}\label{eq3.30}
		\begin{aligned}
			c+o_n(1) =J_\varepsilon(u_n)
			& \geq \frac{1}{p} \|u_n\|_{W^{1,p}_{V_\varepsilon}}^p+\frac{1}{N} \|u_n\|_{W^{1,N}_{V_\varepsilon}}^N+\Big(\frac{1}{N}-\tau-\ell\Big)\|u_n\|^N_{N}-\kappa_\tau C S^{-1}_r \|u_n\|_{\mathbf{Y}_\varepsilon}  \\
			& +\int_{\mathbb{R}^N} \mathcal{H}_1(u_n)\,\mathrm{d}x-(\tilde{c}+\tilde{d} \|u_n\|_{\mathbf{Y}_\varepsilon})- \frac{\widetilde{A}}{N}\big(1+\|u_n\|^{1+\xi}_{\mathbf{Y}_\varepsilon}\big).
		\end{aligned}
	\end{equation}
	Since $\tau>0$ is arbitrary and $\ell>0$ is small enough, we can choose $\tau$ small enough such that $\tau+2\ell\leq \frac{1}{N}$. Consequently, by using \eqref{eq3.30} and Lemma \ref{lem2.899}, there exists constants $C_1,C_2,C_3>0$ such that, as $n\to\infty$, we have
	\begin{equation}\label{eq3.31}
		C_1+C_2\|u_n\|_{\mathbf{Y}_\varepsilon} +C_3\|u_n\|^{1+\xi}_{\mathbf{Y}_\varepsilon}+o_n(1)\geq \frac{1}{N} \Big(\|u_n\|_{W^{1,p}_{V_\varepsilon}}^p+ \|u_n\|_{W^{1,N}_{V_\varepsilon}}^N\Big)+\min\big\{\|u_n\|_{\mathcal{H}_1}^l,\|u_n\|_{\mathcal{H}_1}^N\big\},
	\end{equation}
	where $1<l<N$. For the rest of the proof, we fix $\xi\in(0,1)$ such that $1+\xi<p<l$. If possible, let $\|u_n\|_{\mathcal{H}_1}\leq 1$, then we obtain from \eqref{eq3.31} that
	\begin{equation}
		\label{eq3.32}
		C_1+C_2\|u_n\|_{\mathbf{Y}_\varepsilon} +C_3\|u_n\|^{1+\xi}_{\mathbf{Y}_\varepsilon}+o_n(1)\geq \frac{1}{N} \Big(\|u_n\|_{W^{1,p}_{V_\varepsilon}}^p+ \|u_n\|_{W^{1,N}_{V_\varepsilon}}^N\Big)+\|u_n\|_{\mathcal{H}_1}^N\quad\text{as }n\to\infty.
	\end{equation}
	In this case, we have three possibilities as follows:\\
	\textit{Case-1:}
	Let $\|u_n\|_{W^{1,p}_{V_\varepsilon}}\to\infty$ and $\|u_n\|_{W^{1,N}_{V_\varepsilon}}\to\infty$ as $n\to\infty$. It follows that $\|u_n\|^{N}_{W^{1,N}_{V_\varepsilon}}\geq \|u_n\|^{p}_{W^{1,N}_{V_\varepsilon}}>1$ for $n$ large enough. Consequently, we obtain from \eqref{eq3.32} that
	\begin{align*}
		C_1+C_2\|u_n\|_{\mathbf{Y}_\varepsilon} +C_3\|u_n\|^{1+\xi}_{\mathbf{Y}_\varepsilon}+o_n(1)\geq \frac{2^{1-p}}{N}\|u_n\|^p_{\mathbf{X}_\varepsilon}+\|u_n\|_{\mathcal{H}_1}^N\quad\text{as }n\to\infty.
	\end{align*}
	Dividing $\|u_n\|^p_{\mathbf{X}_\varepsilon}$ on both sides and letting $n\to\infty$, we get $0\geq \displaystyle\frac{2^{1-p}}{N}>0 $, which is a contradiction.\\
	\textit{Case-2:}
	Let $\|u_n\|_{W^{1,p}_{V_\varepsilon}}\to\infty$ as $n\to\infty$ and $\|u_n\|_{W^{1,N}_{V_\varepsilon}}$ is bounded. Dividing $\|u_n\|_{W^{1,p}_{V_\varepsilon}}$ on both the sides of \eqref{eq3.32}  and letting $n\to\infty$, we get $0\geq \displaystyle\frac{1}{N}>0 $, which is again a contradiction.\\
	\textit{Case-3:} Let $\|u_n\|_{W^{1,N}_{V_\varepsilon}}\to\infty$ as $n\to\infty$ and $\|u_n\|_{W^{1,p}_{V_\varepsilon}}
	$ is bounded. Similar to \textit{Case-2}, we get a contradiction.

	Suppose that  $\|u_n\|_{\mathcal{H}_1}\to \infty$ as $n\to\infty$, then we can assume that $\|u_n\|_{\mathcal{H}_1}>1$ for $n$ large enough and hence, we obtain from \eqref{eq3.31} that
	\begin{equation}
		\label{eq3.33}
		C_1+C_2\|u_n\|_{\mathbf{Y}_\varepsilon} +C_3\|u_n\|^{1+\xi}_{\mathbf{Y}_\varepsilon}+o_n(1)\geq \frac{1}{N} \Big(\|u_n\|_{W^{1,p}_{V_\varepsilon}}^p+ \|u_n\|_{W^{1,N}_{V_\varepsilon}}^N\Big)+\|u_n\|_{\mathcal{H}_1}^l\quad\text{as }n\to\infty.
	\end{equation}
	In this case, we have four possibilities as follows:\\
	\textit{Case-1:} If $\|u_n\|_{W^{1,p}_{V_\varepsilon}}
	$ and $\|u_n\|_{W^{1,p}_{V_\varepsilon}} $ are bounded, then dividing $\|u_n\|_{\mathcal{H}_1}^l$ on both sides of \eqref{eq3.33} and letting $n\to\infty$, we get $0\geq 1>0$, which is a contradiction.\\
	\textit{Case-2:} Let $\|u_n\|_{W^{1,p}_{V_\varepsilon}}\to\infty$ and $\|u_n\|_{W^{1,N}_{V_\varepsilon}}\to\infty$ as $n\to\infty$. It follows that $\|u_n\|^{N}_{W^{1,N}_{V_\varepsilon}}\geq \|u_n\|^{p}_{W^{1,N}_{V_\varepsilon}} >1$ and $\|u_n\|_{\mathcal{H}_1}^l\geq \|u_n\|_{\mathcal{H}_1}^p>1$  for $n$ large enough. Hence, we obtain from \eqref{eq3.33} that
	\begin{align*}
		C_1+C_2\|u_n\|_{\mathbf{Y}_\varepsilon} +C_3\|u_n\|^{1+\xi}_{\mathbf{Y}_\varepsilon}+o_n(1)\geq \frac{3^{1-p}}{N}\|u_n\|^p_{\mathbf{Y}_\varepsilon}~\quad\text{as }n\to\infty.
	\end{align*}
	Dividing $\|u_n\|^p_{\mathbf{Y}_\varepsilon}$ on both sides and letting $n\to\infty$, we get $0\geq \displaystyle\frac{3^{1-p}}{N}>0 $, which is again contradiction.\\
	\textit{Case-3:}
	Let $\|u_n\|_{W^{1,p}_{V_\varepsilon}}\to\infty$ as $n\to\infty$ and $\|u_n\|_{W^{1,N}_{V_\varepsilon}}$ is bounded. It follows that $\|u_n\|_{\mathcal{H}_1}^l\geq \|u_n\|_{\mathcal{H}_1}^p>1$  for $n$ large enough. Consequently, from \eqref{eq3.33}, we have
	\begin{align*}
		C_1+C_2\|u_n\|_{\mathbf{Y}_\varepsilon} +C_3\|u_n\|^{1+\xi}_{\mathbf{Y}_\varepsilon}+o_n(1)\geq \frac{2^{1-p}}{N} \Big(\|u_n\|_{W^{1,p}_{V_\varepsilon}}+\|u_n\|_{\mathcal{H}_1} \Big)^p~\quad\text{as }n\to\infty.
	\end{align*}
	Dividing $\Big(\|u_n\|_{W^{1,p}_{V_\varepsilon}}+\|u_n\|_{\mathcal{H}_1} \Big)^p$ on both sides and letting $n\to\infty$, we get $0\geq \displaystyle\frac{2^{1-p}}{N}>0 $, which is a contradiction.\\
	\textit{Case-4:} Let $\|u_n\|_{W^{1,N}_{V_\varepsilon}}\to\infty$ as $n\to\infty$ and $\|u_n\|_{W^{1,p}_{V_\varepsilon}}$ is bounded. Similar to \textit{Case-3}, we get a contradiction.\\
	Hence, we conclude from the above seven cases that $\{u_n\}_{n\in\mathbb{N}}\subset\mathbf{Y}_\varepsilon$ must be bounded. This finishes the proof.
\end{proof}

The following lemma is devoted to the tightness of the Palais-Smale sequences for $J_\varepsilon$.

\begin{lemma}\label{lem3.5}
	Let $\{u_n\}_{n\in\mathbb{N}}\subset \mathbf{Y}_\varepsilon$ be a \textnormal{(PS)$_c$} sequence for $J_{\varepsilon}$ as stated in Lemma \ref{lem3.4}. Then, for all $\xi>0$, there exists $R=R(\xi)>0$ such that
	\begin{align*}
		\limsup_{n\to\infty} \int_{B^c_R} \big[\big(|\nabla u_n|^{p} +V(\varepsilon x)|u_n|^{p}\big)+\big(|\nabla u_n|^{N} +(V(\varepsilon x)+1)|u_n|^{N}\big)\big]\,\mathrm{d}x<\xi.
	\end{align*}
\end{lemma}

\begin{proof}
	For $R>0$, let $\psi_R\in C^\infty(\mathbb{R}^N)$ be a such that $0\leq \psi_R\leq 1~\text{in}~\mathbb{R}^N$, $\psi_R\equiv 0~\text{in}~B_{\frac{R}{2}}$, $\psi_R\equiv 1~\text{in}~B^c_R$ and $|\nabla \psi_R|\leq \frac{C}{R}$, where $C>0$ is a constant independent of $R$. Further, we choose $R>0$ in such a way that $\Lambda_\varepsilon\subset B_{\frac{R}{2}}$. Due to Lemma \ref{lem3.4}, we infer that the sequence $\{u_n\psi_R\}_{n\in\mathbb{N}}$ is bounded in $\mathbf{Y}_\varepsilon$ and there holds
	$\langle{J'_\varepsilon(u_n),u_n\psi_R }\rangle\to 0$  as $n\to\infty$. Consequently, by using ($\mathcal{H}_1$)(b), ($\mathcal{A}$)(c) and ($\mathcal{B}$)(d), we  obtain
	\begin{align}
		\label{eq3.34}
		\int_{\mathbb{R}^N}\big[\big(|\nabla u_n|^{p} +V(\varepsilon x)|u_n|^{p}\big)+\big(|\nabla u_n|^{N} +(V(\varepsilon x)+1)|u_n|^{N}\big)\big]\psi_R\,\mathrm{d}x\leq I_1+I_2+o_n(1)~\quad\text{as }n\to\infty,
	\end{align}
	where
	\begin{align*}
		I_1=(\ell+\ell')\int_{\mathbb{R}^N} |u_n|^N\psi_R\,\mathrm{d}x
		\quad\text{and}\quad
		I_2=\bigg|\int_{\mathbb{R}^N} u_n\big(|\nabla u_n|^{p-2}+|\nabla u_n|^{N-2}\big) \nabla u_n\cdot \nabla \psi_R\,\mathrm{d}x\bigg|.
	\end{align*}
	In view of (V1) and the fact that $V_0+1\geq 2(\ell+\ell')$, we get
	\begin{equation}\label{eq3.35}
		\begin{aligned}
			I_1 & \leq \frac{V_0+1}{2} \int_{\mathbb{R}^N} |u_n|^N\psi_R\,\mathrm{d}x\leq \frac{1}{2} \int_{\mathbb{R}^N}(V(\varepsilon x)+1) |u_n|^N\psi_R\,\mathrm{d}x         \\
			& \leq \frac{1}{2} \int_{\mathbb{R}^N} \big[\big(|\nabla u_n|^{p} +V(\varepsilon x)|u_n|^{p}\big)+\big(|\nabla u_n|^{N} +(V(\varepsilon x)+1)|u_n|^{N}\big)\big]\psi_R\,\mathrm{d}x.
		\end{aligned}
	\end{equation}
	Moreover, due to the boundedness of $\{u_n\}_{n\in\mathbb{N}}$ in $\mathbf{Y}_\varepsilon$ and H\"older's inequality, we have
	\begin{align}\label{eq3.36}
		I_2 & \leq \frac{C}{R} \big(\|u_n\|_p\|\nabla u_n\|^{p-1}_p+\|u_n\|_N\|\nabla u_n\|^{N-1}_N\big)  \leq \frac{C}{R} \big(S^{-1}_p\|u_n\|^p_{\mathbf{Y}_\varepsilon}+S^{-1}_N\|u_n\|^N_{\mathbf{Y}_\varepsilon}\big)\leq \frac{\widetilde{C}}{R},
	\end{align}
	where $\widetilde{C}>0$ is a constant. Now, it follows from \eqref{eq3.34}, \eqref{eq3.35} and \eqref{eq3.36} that
	\begin{align*}
		\int_{\mathbb{R}^N}\big[\big(|\nabla u_n|^{p} +V(\varepsilon x)|u_n|^{p}\big)+\big(|\nabla u_n|^{N} +(V(\varepsilon x)+1)|u_n|^{N}\big)\big]\psi_R\,\mathrm{d}x\leq \frac{2\widetilde{C}}{R}+o_n(1)\quad\text{as }n\to\infty.
	\end{align*}
	Fix $\xi>0$ and take $R>0$ sufficiently large such that $\frac{2\widetilde{C}}{R}<\xi$. Passing to $n\to\infty$ in the above inequality, we obtain
	\begin{align*}
		\limsup_{n\to\infty} \int_{B^c_R} \big[\big(|\nabla u_n|^{p} +V(\varepsilon x)|u_n|^{p}\big)+\big(|\nabla u_n|^{N} +(V(\varepsilon x)+1)|u_n|^{N}\big)\big]\,\mathrm{d}x\leq \frac{2\widetilde{C}}{R}<\xi.
	\end{align*}
	This finishes the proof.
\end{proof}

\begin{lemma}\label{lem3.6}
	Let $u\in \mathbf{Y}_\varepsilon $ and $\{u_n\}_{n\in\mathbb{N}}$ be a \textnormal{(PS)$_c$} sequence for $J_{\varepsilon}$ satisfying $\displaystyle\limsup_{n\to\infty} \|u_n\|^{N'}_{W^{1,N}}<\frac{\alpha_N}{\alpha_0}$. If $u_n\rightharpoonup u$ in $\mathbf{Y}_\varepsilon$ as $n\to\infty$, then we have $\nabla u_n\to \nabla u$ a.e.\,in $\mathbb{R}^N$ as $n\to\infty$. Consequently, we deduce that $u$ is a critical point for $J_{\varepsilon}$, that is, $J'_{\varepsilon}(u)=0$.
\end{lemma}

\begin{proof}
	Note that by Lemma \ref{lem3.4}, we conclude that the sequence $\{u_n\}_{n\in\mathbb{N}}$ is bounded in $ \mathbf{Y}_\varepsilon$. By the hypothesis, we have  $u_n\rightharpoonup u$ in $\mathbf{Y}_\varepsilon$ as $n\to\infty$. Due to Remark \ref{remark2.4} and Lemma \ref{lem8.98}, we infer that
	\begin{equation}\label{eq3.37}
		u_n\rightharpoonup u\quad\text{in }W^{1,t}_{V_\varepsilon}(\mathbb{R}^N)\text{ for }t\in\{p,N\},
		\quad u_n\to u \quad\text{in }L^\theta (B_R), \quad   u_n\to u\quad\text{a.e.\,in }\mathbb{R}^N\quad\text{as }n\to\infty
	\end{equation}
	for any $R>0$ and $\theta\in[1,+\infty)$. Consequently, there exists $g_R\in L^\wp(B_R)$ with $\wp\geq N$ such that $|u_n|\leq g_R$ a.e.\,in $B_R$. Fix $R>0$ and $\psi\in C^\infty_c(\mathbb{R}^N)$ such that $0\leq \psi\leq 1$ in $\mathbb{R}^N$, $\psi\equiv 1$ in $B_R$, $\psi\equiv 0$ in $B^c_{2R}$ and $\|\nabla \psi\|_\infty\leq C$ for some constant $C>0$ independent of $R$. Note that $J_\varepsilon\in C^1(\mathbf{Y}_\varepsilon,\mathbb{R})$, $u_n\rightharpoonup u$ in $\mathbf{Y}_\varepsilon$, and $J_\varepsilon'(u_n)\to 0$ in $\mathbf{Y}_\varepsilon^*$ as $n\to\infty$, thus we obtain
	\begin{equation}\label{eq3.38}
		\langle{J_\varepsilon'(u_n)-J_\varepsilon'(u),(u_n-u)\psi}\rangle=o_n(1)\quad\text{as }n\to\infty.
	\end{equation}
	For any $n\in\mathbb{N}$ and $t\in\{p,N\}$, we define
	\begin{align*}
		\boldsymbol{D}^t_n=
		\big(|\nabla u_n|^{t-2} \nabla u_n-|\nabla u|^{t-2} \nabla u\big)\cdot \big(\nabla u_n-\nabla u\big) +V(\varepsilon x)\big(| u_n|^{t-2} u_n-|u|^{t-2}u\big) (u_n-u).
	\end{align*}
	Moreover, by convexity and (V1), we can see that
	\begin{align*}
		&\big(|\nabla u_n|^{t-2} \nabla u_n-|\nabla u|^{t-2} \nabla u\big)\cdot \big(\nabla u_n-\nabla u\big)\geq 0\quad \text{a.e.\,in } \mathbb{R}^N,\\
		&V(\varepsilon x)\big(| u_n|^{t-2} u_n-|u|^{t-2}u\big) (u_n-u)\geq 0\quad \text{a.e.\,in } \mathbb{R}^N
	\end{align*}
	for any $n\in\mathbb{N}$ and $t\in\{p,N\}$. Due to Simon's inequality  (see \cite{Simon-1977}) with $N\geq 2$ and \eqref{eq3.38}, there exists $c_N>0$ such that as $n\to\infty$, we have
	\begin{equation} \label{eq3.39}
		\begin{aligned}
			&c^{-1}_N \int_{B_R} |\nabla u_n-\nabla u|^N\,\mathrm{d}x\\
			& \leq  c^{-1}_N \bigg[\int_{B_R} |\nabla u_n-\nabla u|^N\,\mathrm{d}x+\int_{B_R} V(\varepsilon x)| u_n-u|^N\,\mathrm{d}x\bigg]      \\
			& \leq  \int_{B_R} \boldsymbol{D}^N_n \,\mathrm{d}x\leq \displaystyle\sum_{t\in\{p,N\}} \bigg(\int_{B_R}\boldsymbol{D}^t_n \,\mathrm{d}x\bigg)\leq \displaystyle\sum_{t\in\{p,N\}} \bigg(\int_{\mathbb{R}^N}\boldsymbol{D}^t_n~ \psi\,\mathrm{d}x\bigg)  \\
			& =o_n(1) -\sum_{t\in\{p,N\}} \Bigg[\int_{\mathbb{R}^N}\big(|\nabla u_n|^{t-2}\nabla u_n-|\nabla u|^{t-2}\nabla u\big)\cdot\nabla\psi( u_n-u)\,\mathrm{d}x\Bigg]    \\
			& \quad -\int_{\mathbb{R}^N} \big(| u_n|^{N-2} u_n-|u|^{N-2}u\big) (u_n-u)\psi \,\mathrm{d}x-\int_{\mathbb{R}^N} \big(\mathcal{H}'_1(u_n)-\mathcal{H}'_1(u)\big) (u_n-u)\psi \,\mathrm{d}x   \\
			& \quad +\int_{\mathbb{R}^N} \big(G'_2(\varepsilon x,u_n)-G'_2(\varepsilon x,u)\big) (u_n-u)\psi\,\mathrm{d}x +\int_{\mathbb{R}^N} \big(g(\varepsilon x,u_n)-g(\varepsilon x,u)\big) (u_n-u)\psi\,\mathrm{d}x.
		\end{aligned}
	\end{equation}
	Using H\"older's inequality and \eqref{eq3.37}, for $t\in\{p,N\}$, we get
	\begin{align*}
		&\bigg|\int_{\mathbb{R}^N}\big(|\nabla u_n|^{t-2}\nabla u_n-|\nabla u|^{t-2}\nabla u\big)\cdot\nabla\psi( u_n-u)\,\mathrm{d}x\bigg|\\
		& \leq \|\nabla\psi\|_\infty\big(\|\nabla u_n\|^{t-1}_t+\|\nabla u\|^{t-1}_t\big) \bigg(\int_{B_{2R}}|u_n-u|^t\,\mathrm{d}x\bigg)^{\frac{1}{t}}
		\to 0\quad\text{as }n\to\infty.
	\end{align*}
	It follows that for $t\in\{p,N\}$, we have
	\begin{equation}
		\label{eq3.40}
		\lim_{n\to\infty} \int_{\mathbb{R}^N}\big(|\nabla u_n|^{t-2}\nabla u_n-|\nabla u|^{t-2}\nabla u\big)\cdot\nabla\psi( u_n-u)\,\mathrm{d}x=0.
	\end{equation}
	Likewise, we can also prove that
	\begin{equation}
		\label{eq3.41}
		\lim_{n\to\infty} \int_{\mathbb{R}^N}\big(| u_n|^{N-2} u_n-| u|^{N-2}u\big)( u_n-u)\psi\,\mathrm{d}x=0.
	\end{equation}
	In view of H\"older's inequality, ($\mathcal{H}_1$)(c) and \eqref{eq3.37}, we obtain
	\begin{align*}
		&\bigg|\int_{\mathbb{R}^N} \big(\mathcal{H}'_1(u_n)-\mathcal{H}'_1(u)\big) (u_n-u)\psi~ \mathrm{d}x\bigg|\\
		& \leq C_1\big(\|u_n\|^{q-1}_q+\|u\|^{q-1}_q\big)\bigg(\int_{B_{2R}}|u_n-u|^q\,\mathrm{d}x\bigg)^{\frac{1}{q}}+2C_2 \int_{B_{2R}}|u_n-u|\,\mathrm{d}x
		\to 0\quad\text{as }n\to\infty.
	\end{align*}
	This yields
	\begin{equation}\label{eq3.42}
		\lim_{n\to\infty} \int_{\mathbb{R}^N} \big(\mathcal{H}'_1(u_n)-\mathcal{H}'_1(u)\big) (u_n-u)\psi~ \mathrm{d}x=0.
	\end{equation}
	Similarly, by using H\"older's inequality, ($\mathcal{A}$)(a) and \eqref{eq3.37}, we have
	\begin{align*}
		&\bigg|\int_{\mathbb{R}^N} \big(G'_2(\varepsilon x,u_n)-G'_2(\varepsilon x,u)\big) (u_n-u)\psi\,\mathrm{d}x\bigg|\\
		& \leq  \ell\big(\|u_n\|^{N-1}_N+\|u\|^{N-1}_N\big)\bigg(\int_{B_{2R}}|u_n-u|^N\,\mathrm{d}x\bigg)^{\frac{1}{N}} \\
		& \quad+C\big(\|u_n\|^{q-1}_q+\|u\|^{q-1}_q\big)\bigg(\int_{B_{2R}}|u_n-u|^q\,\mathrm{d}x\bigg)^{\frac{1}{q}} \to 0\quad\text{as }n\to\infty.
	\end{align*}
	It follows that
	\begin{equation}\label{eq3.43}
		\lim_{n\to\infty} \int_{\mathbb{R}^N} \big(G'_2(\varepsilon x,u_n)-G'_2(\varepsilon x,u)\big) (u_n-u)\psi\,\mathrm{d}x=0.
	\end{equation}
	Due to H\"older's inequality and \eqref{eq3.37}, we have
	\begin{align*}
		J_1=\tau \int_{B_{2R}} \big(|u_n|^{N-1}+|u|^{N-1}\big)|u_n-u|\,\mathrm{d}x\leq \tau \big(\|u_n\|^{N-1}_N+\|u\|^{N-1}_N\big)\bigg(\int_{B_{2R}}|u_n-u|^N\,\mathrm{d}x\bigg)^{\frac{1}{N}} \to 0
	\end{align*}
	as $n\to \infty$. From the hypothesis, it follows that there exists $m_1>0$ and $n_0\in\mathbb{N}$ sufficiently large such that $\|u_n\|^{N'}_{W^{1,N}}<m_1<\frac{\alpha_N}{\alpha_0}$ for all $n\geq n_0$. Take $r>1$ with $r'=\frac{r}{r-1}>1$ and satisfying $\frac{1}{r}+\frac{1}{r'}=1$ . Let $r'$ close to 1 and $\alpha>\alpha_0$ close to $\alpha_0$ such that we still have $r'\alpha\|u_n\|^{N'}_{W^{1,N}}<m_1<\alpha_N$ for all $n\geq n_0$ and $\widetilde{u}_n = \frac{u_n}{\|u_n\|_{W^{1,N}}}$.  By using H\"older's inequality, Corollary \ref{cor2.9} and \eqref{eq3.37}, we obtain
	\begin{align*}
		J_2 & =\kappa_\tau \int_{B_{2R}} \big(\Phi(\alpha_0 |u_n|^{N'})+\Phi(\alpha_0 |u|^{N'})\big)|u_n-u|\,\mathrm{d}x\\
		& \leq \kappa_\tau\Big(\big\|\Phi(\alpha_0 |u_n|^{N'})\big\|_{r'}+\big\|\Phi(\alpha_0 |u|^{N'})\big\|_{r'}\Big)\bigg(\int_{B_{2R}}|u_n-u|^r\,\mathrm{d}x\bigg)^{\frac{1}{r}}\\
		& \leq \kappa_\tau \Bigg(\Bigg[\int_{\mathbb{R}^N}\Phi\big(r'\alpha\|u_n\|^{N'}_{W^{1,N}}|\widetilde{u}_n|^{N'}\big)\,\mathrm{d}x\Bigg]^{\frac{1}{r'}}+\Bigg[\int_{\mathbb{R}^N}\Phi\big(r'\alpha_0|u|^{N'}\big)\,\mathrm{d}x\Bigg]^{\frac{1}{r'}}\Bigg) \bigg(\int_{B_{2R}}|u_n-u|^r\,\mathrm{d}x\bigg)^{\frac{1}{r}} \\
		& \leq \widehat{C} \bigg(\int_{B_{2R}}|u_n-u|^r\,\mathrm{d}x\bigg)^{\frac{1}{r}}\to 0\quad\text{as }n\to\infty,
	\end{align*}
	where
	\begin{align*}
		\widehat{C}= \kappa_\tau \Bigg(\Bigg[\sup_{n\geq n_0}\int_{\mathbb{R}^N}\Phi\big(r'\alpha\|u_n\|^{N'}_{W^{1,N}}|\widetilde{u}_n|^{N'}\big)\,\mathrm{d}x\Bigg]^{\frac{1}{r'}}+\Bigg[\int_{\mathbb{R}^N}\Phi\big(r'\alpha_0|u|^{N'}\big)\,\mathrm{d}x\Bigg]^{\frac{1}{r'}}\Bigg)<+\infty,
	\end{align*}
	thanks Lemma \ref{lem2.11}. Consequently, by using ($\mathcal{B}$)(b) and (f1), we obtain
	\begin{align*}
		\bigg|\int_{\mathbb{R}^N} \big(g(\varepsilon x,u_n)-g(\varepsilon x,u)\big) (u_n-u)\psi\,\mathrm{d}x\bigg| & \leq \int_{B_{2R}} \big(|f(u_n)|+|f(u)|\big)|u_n-u|\,\mathrm{d}x \leq J_1+J_2\to 0
	\end{align*}
	as $n\to\infty$. This shows that
	\begin{equation}\label{eq3.44}
		\lim_{n\to\infty} \int_{\mathbb{R}^N} \big(g(\varepsilon x,u_n)-g(\varepsilon x,u)\big) (u_n-u)\psi\,\mathrm{d}x=0.
	\end{equation}
	Passing $n\to\infty$ in \eqref{eq3.39} and using \eqref{eq3.40}, \eqref{eq3.41}, \eqref{eq3.42}, \eqref{eq3.43} and \eqref{eq3.44}, we get $\nabla u_n\to \nabla u$ in $[L^{N}(B_{R})]^N$ as $n\to\infty$ for all $R>0$. Hence, up to a subsequence, still denoted by itself, $\nabla u_n\to \nabla u$ a.e.\,in $\mathbb{R}^N$ as $n\to\infty$.

Now, fix $v\in C^\infty_c(\mathbb{R}^N)$ and let $R>0$ be large enough such that $\operatorname{supp}v\subset B_R$. Using the boundedness of $\{u_n\}_{n\in\mathbb{N}}$ in $ \mathbf{Y}_\varepsilon$, we deduce that $\{|\nabla u_n|^{t-2}\nabla u_n\}_{n\in\mathbb{N}}$ and $\{V(\varepsilon x)^{\frac{t-1}{t}}| u_n|^{t-2} u_n\}_{n\in\mathbb{N}}$ are bounded in $[L^{\frac{t}{t-1}}(\mathbb{R}^N)]^N$ and $L^{\frac{t}{t-1}}(\mathbb{R}^N)$, respectively, for $t\in\{p,N\}$. Note that for $t\in\{p,N\}$, we have
	\begin{align*}
		|\nabla u_n|^{t-2}\nabla u_n\to |\nabla u|^{t-2}\nabla u
		\quad\text{and}\quad
		V(\varepsilon x)^{\frac{t-1}{t}}| u_n|^{t-2} u_n\to V(\varepsilon x)^{\frac{t-1}{t}}| u|^{t-2} u\quad\text{a.e.\,in }\mathbb{R}^N\quad\text{as }n\to\infty.
	\end{align*}
	Consequently, for $t\in\{p,N\}$, we have
	\begin{align*}
		|\nabla u_n|^{t-2}\nabla u_n&\rightharpoonup |\nabla u|^{t-2}\nabla u\quad\text{in }[L^{\frac{t}{t-1}}(\mathbb{R}^N)]^N,\\
		V(\varepsilon x)^{\frac{t-1}{t}}| u_n|^{t-2} u_n&\rightharpoonup V(\varepsilon x)^{\frac{t-1}{t}}| u|^{t-2} u\quad\text{in }L^{\frac{t}{t-1}}(\mathbb{R}^N)
	\end{align*}
	as $n\to\infty$. Exploiting the density of $C^\infty_c(\mathbb{R}^N)$  in $W^{1,t}_{V_\varepsilon}(\mathbb{R}^N)$ for $t\in\{p,N\}$, we  obtain for $t\in\{p,N\}$ that
	\begin{equation}\label{eq3.45}
		\int_{\mathbb{R}^N} |\nabla u_n|^{t-2}\nabla u_n\cdot \nabla v\,\mathrm{d}x\to  \int_{\mathbb{R}^N} |\nabla u|^{t-2}\nabla u\cdot \nabla v\,\mathrm{d}x\quad\text{as }n\to\infty
	\end{equation}
	and
	\begin{equation}\label{eq3.46}
		\int_{\mathbb{R}^N} V(\varepsilon x)| u_n|^{t-2} u_n v\,\mathrm{d}x\to  \int_{\mathbb{R}^N} V(\varepsilon x)| u|^{t-2} u v\,\mathrm{d}x\quad\text{as }n\to\infty.
	\end{equation}
	Similarly, we can also prove that
	\begin{equation}\label{eq3.47}
		\int_{\mathbb{R}^N} | u_n|^{N-2} u_n v\,\mathrm{d}x\to  \int_{\mathbb{R}^N} | u|^{N-2} u v\,\mathrm{d}x\quad\text{as }n\to\infty.
	\end{equation}
	On the other hand, by using H\"older's inequality, ($\mathcal{H}_1$)(c), ($\mathcal{A}$)(a) and \eqref{eq3.37}, we get
	\begin{align*}
		|\mathcal{H}'_1(u_n)v|\leq \big(C_1|u_n|^{q-1}+C_2\big)|v|\leq \big( C_1 g_R^{q-1}+C_2)|v|\in L^1(B_R)
	\end{align*}
	and
	\begin{align*}
		|G'_2(\varepsilon x,u_n)v|\leq \big(\ell~|u_n|^{N-1}+C|u_n|^{q-1}\big)|v|\leq \big( \ell ~g_R^{N-1}+Cg_R^{q-1})|v|\in L^1(B_R).
	\end{align*}
	Due to Lebesgue's dominated convergence theorem, we conclude that
	\begin{equation}\label{eq3.48}
		\int_{\mathbb{R}^N}  \mathcal{H}'_1(u_n)v\,\mathrm{d}x\to  \int_{\mathbb{R}^N}  \mathcal{H}'_1(u)v\,\mathrm{d}x\quad\text{as }n\to\infty
	\end{equation}
	and
	\begin{equation}
		\label{eq3.49}
		\int_{\mathbb{R}^N}  G'_2(\varepsilon x,u_n)v\,\mathrm{d}x\to  \int_{\mathbb{R}^N}  G'_2(\varepsilon x,u)v\,\mathrm{d}x\quad\text{as }n\to\infty.
	\end{equation}
	Taking into account the notations used on the previous page to handle the exponential nonlinearity, we obtain by using ($\mathcal{B}$)(b), (f1), H\"older's inequality, Corollary \ref{cor2.9} and Lemma \ref{lem8.98} that for all $n\geq n_0$
	\begin{equation}\label{eq3.50}
		\int_{\mathbb{R}^N} g(\varepsilon x,u_n)v\,\mathrm{d}x \leq C\Bigg[\bigg(\int_{\mathbb{R}^N}|v|^N\,\mathrm{d}x\bigg)^{\frac{1}{N}}+\bigg(\int_{\mathbb{R}^N}|v|^r\,\mathrm{d}x\bigg)^{\frac{1}{r}}\Bigg] <+\infty,
	\end{equation}
	where
	\begin{align*}
		C=\tau S^{-(N-1)}_N\sup_{n\in\mathbb{N}}\|u_n\|^{N-1}_{\mathbf{Y}_\varepsilon}+\kappa_\tau \Bigg[\sup_{n\geq n_0}\int_{\mathbb{R}^N}\Phi\big(r'\alpha\|u_n\|^{N'}_{W^{1,N}}|\widetilde{u}_n|^{N'}\big)\,\mathrm{d}x\Bigg]^{\frac{1}{r'}},
	\end{align*}
	which is finite because of Lemma \ref{lem2.11} and the fact that $\{u_n\}_{n\in\mathbb{N}}$ is uniformly bounded in $\mathbf{Y}_\varepsilon$. It follows from \eqref{eq3.50} that $\{g(\varepsilon x,u_n)v\}_{n\geq n_0}$ is bounded $L^1(\mathbb{R}^N)$. Consequently, it is not difficult to verify that $\{g(\varepsilon x,u_n)v\}_{n\geq n_0}$ is uniformly absolutely integrable and tight over $\mathbb{R}^N$. Since $u_n\to u$ a.e.\,in $\mathbb{R}^N$ as $n\to\infty$, therefore we have $g(\varepsilon x,u_n)v\to g(\varepsilon x,u)v$ a.e.\,in $\mathbb{R}^N$ as $n\to\infty$. Now, by applying Vitali's convergence theorem, we obtain
	\begin{equation}
		\label{eq3.51}
		\int_{\mathbb{R}^N} g(\varepsilon x,u_n)v\,\mathrm{d}x \to   \int_{\mathbb{R}^N} g(\varepsilon x,u)v\,\mathrm{d}x \quad\text{as }n\to\infty.
	\end{equation}
	Observe that $\langle{J'_\varepsilon(u_n),v}\rangle=o_n(1)$ as $n\to\infty$, therefore using \eqref{eq3.45}, \eqref{eq3.46}, \eqref{eq3.47}, \eqref{eq3.48}, \eqref{eq3.49}, \eqref{eq3.50} and \eqref{eq3.51}, we infer that $\langle{J'_\varepsilon(u),v}\rangle=0$ for all $v\in C^\infty_c(\mathbb{R}^N)$. Exploiting the density of $C^\infty_c(\mathbb{R}^N)$ in $\mathbf{Y}_{\varepsilon}$, we deduce that $u$ is a critical point for $J_{\varepsilon}$, that is, $J'_{\varepsilon}(u)=0$. This completes the proof.
\end{proof}

\begin{lemma}\label{lem3.7}
	Under the assumptions of Lemma \ref{lem3.4}, Lemma  \ref{lem3.5} and Lemma \ref{lem3.6}, the functional $J_\varepsilon$ satisfies the \textnormal{(PS)$_c$} condition at any level $c\in\mathbb{R}$.
\end{lemma}

\begin{proof}
	Let $c\in\mathbb{R}$ and $\{u_n\}_{n\in\mathbb{N}}\subset \mathbf{Y}_\varepsilon$ be a \textnormal{(PS)$_c$} sequence for $J_{\varepsilon}$ at level $c\in\mathbb{R}$. By Lemma \ref{lem3.4}, the sequence $\{u_n\}_{n\in\mathbb{N}}$ is bounded in $\mathbf{Y}_\varepsilon$. Then, without loss of generality, up to a subsequence, we have $u_n\rightharpoonup u$ in $\mathbf{Y}_\varepsilon$ as $n\to\infty$ for some $u\in\mathbf{Y}_\varepsilon$. By using Lemma \ref{lem3.6}, we have $\langle{J'_{\varepsilon}(u),v}\rangle=0$ for all $v\in \mathbf{Y}_\varepsilon$. In particular, we have $\langle{J'_{\varepsilon}(u),u}\rangle=0$, that is,
	\begin{equation}\label{eq3.52}
		\|u\|_{W^{1,p}_{V_\varepsilon}}^p+ \|u\|_{W^{1,N}_{V_\varepsilon}}^N+\int_{\mathbb{R}^N}\mathcal{H}'_1(u)u\,\mathrm{d}x=\int_{\mathbb{R}^N}G'_2(\varepsilon x,u)u\,\mathrm{d}x+\int_{\mathbb{R}^N}g(\varepsilon x,u)u\,\mathrm{d}x-\int_{\mathbb{R}^N}|u|^N\,\mathrm{d}x.
	\end{equation}
	Note that \eqref{eq3.37} holds. In view of (V1) and Lemma \ref{lem3.5}, we obtain
	\begin{align*}
		\limsup_{n\to\infty}\displaystyle\int_{B^c_R}|u_n|^p\,\mathrm{d}x<\frac{\xi}{V_0}
		\quad\text{and}\quad
		\displaystyle\limsup_{n\to\infty}\displaystyle\int_{B^c_R}|u_n|^N\,\mathrm{d}x<\frac{\xi}{V_0+1}.
	\end{align*}
	It follows from $u\in L^t(\mathbb{R}^N)$ for $t\in\{p,N\}$ that there exists $R>0$ large enough such that
	$\int_{B^c_R}|u|^t\,\mathrm{d}x<\xi.$
	Now, using all these information and \eqref{eq3.37}, we get
	\begin{align*}
		\limsup_{n\to\infty}\|u_n-u\|^p_p
		& =\limsup_{n\to\infty}\big(\|u_n-u\|^p_{L^p(B
		_R)}+\|u_n-u\|^p_{L^p(B^c
		_R)} \big)\\
		& =\lim_{n\to\infty}\|u_n-u\|^p_{L^p(B
		_R)}+\limsup_{n\to\infty}\|u_n-u\|^p_{L^p(B^c
		_R)}=\limsup_{n\to\infty}\|u_n-u\|^p_{L^p(B^c
		_R)}                                                                                   \\
		& \leq 2^{p-1}\bigg(\limsup_{n\to\infty}\|u_n\|^p_{L^p(B^c
		_R)}+\|u\|^p_{L^p(B^c
		_R)}\bigg)<2^{p-1}\Big(\frac{\xi}{V_0}+\xi\Big)=\widetilde{C}_1\xi\quad\text{with some }\widetilde{C}_1>0.
	\end{align*}
	Similarly, we can also prove that
	\begin{align*}
		\limsup_{n\to\infty}\|u_n-u\|^N_N
		<2^{N-1}\Big(\frac{\xi}{V_0+1}+\xi\Big)=\widetilde{C}_2\xi \quad\text{with }\widetilde{C}_2>0.
	\end{align*}
	Due to the arbitrariness of $\xi>0$, it follows that  $u_n\to u$ in $ L^t(\mathbb{R}^N)$ for $t\in\{p,N\}$ as $n\to\infty$. Applying the boundedness of $\{u_n\}_{n\in\mathbb{N}}$ in $\mathbf{Y}_\varepsilon$, we obtain by using Lemma \ref{lem8.98} and the interpolation inequality that
	\begin{equation}\label{eq3.53}
		u_n\to u\quad\text{in } L^\theta(\mathbb{R}^N) \quad\text{for } \theta\in[p,p^*)\cup[N,+\infty)\quad\text{as }n\to\infty.
	\end{equation}
	By using \eqref{eq3.37}, \eqref{eq3.53} and Lebesgue's dominated convergence theorem, it follows that
	\begin{align*}
		\lim_{R\to\infty} \limsup_{n\to\infty} \|u_n\|^\theta_{L^\theta(B^c
			_R)}=0\quad \text{for any }\theta\in[p,p^*)\cup[N,+\infty).
	\end{align*}
	This implies that for all $\xi>0$, there exists $R=R(\xi)$ large enough such that
	\begin{equation}\label{eq3.54}
		\limsup_{n\to\infty} \int_{B^c
			_R}|u_n|^\theta\,\mathrm{d}x<\xi\quad\text{for any }\theta\in[p,p^*)\cup[N,+\infty).
	\end{equation}
	In view of \eqref{eq3.54} and ($\mathcal{A}$)(a), we have
	\begin{align*}
		\limsup_{n\to\infty} \bigg|\int_{B^c
			_R} G'_2(\varepsilon x,u_n)u_n\,\mathrm{d}x\bigg|\leq \limsup_{n\to\infty} \int_{B^c
			_R}\big(\ell |u_n|^N+C|u_n|^q\big)\,\mathrm{d}x<(\ell+C)\xi=\widetilde{C}_3\xi
	\end{align*}
	with $\widetilde{C}_3>0$.
	Note that $G'_2(\varepsilon x,u)u\in L^1(\mathbb{R}^N) $, therefore choosing $R>0$ large enough, we may assume that
	$\int_{B^c_R} G'_2(\varepsilon x,u)u\,\mathrm{d}x<\xi. $
	Gathering all these information, we have
	\begin{align*}
		\limsup_{n\to\infty} \bigg|\int_{B^c
		_R} G'_2(\varepsilon x,u_n)u_n\,\mathrm{d}x-\int_{B^c
		_R} G'_2(\varepsilon x,u)u\,\mathrm{d}x\bigg|<\widetilde{C}\xi,
	\end{align*}
	for all $\xi>0$ and some suitable constant $\widetilde{C}>0$. By the arbitrariness of $\xi>0$, we conclude that
	\begin{equation}\label{eq3.55}
		\int_{B^c
			_R} G'_2(\varepsilon x,u_n)u_n\,\mathrm{d}x\to \int_{B^c
			_R} G'_2(\varepsilon x,u)u\,\mathrm{d}x\quad\text{as }n\to\infty.
	\end{equation}
	On the other hand, from ($\mathcal{A}$)(a) and \eqref{eq3.37}, we get
	\begin{align*}
		|G'_2(\varepsilon x,u_n)u_n|\leq ~\ell~|u_n|^{N}+C|u_n|^{q}\leq  \ell ~g_R^{N}+Cg_R^{q}\in L^1(B_R).
	\end{align*}
	Consequently, by \eqref{eq3.37} and Lebesgue's dominated convergent theorem, it follows that
	\begin{equation}\label{eq3.56}
		\int_{B
			_R} G'_2(\varepsilon x,u_n)u_n\,\mathrm{d}x\to \int_{B
			_R} G'_2(\varepsilon x,u)u\,\mathrm{d}x\quad\text{as }n\to\infty.
	\end{equation}
	Combining \eqref{eq3.55} and \eqref{eq3.56} together, we obtain
	\begin{equation}\label{eq3.57}
		\int_{\mathbb{R}^N} G'_2(\varepsilon x,u_n)u_n\,\mathrm{d}x\to \int_{\mathbb{R}^N} G'_2(\varepsilon x,u)u\,\mathrm{d}x\quad\text{as }n\to\infty.
	\end{equation}
	From the hypothesis, there exists $m_2>0$ and $n_0\in\mathbb{N}$ sufficiently large such that $\|u_n\|^{N'}_{W^{1,N}}<m_2<\frac{\alpha_N}{\alpha_0}$ for all $n\geq n_0$. Choose $r>1$ with $r'=\frac{r}{r-1}>1$ and satisfying $\frac{1}{r}+\frac{1}{r'}=1$ . Let $r'$ close to 1 and $\alpha>\alpha_0$ close to $\alpha_0$ such that we still have $r'\alpha\|u_n\|^{N'}_{W^{1,N}}<m_2<\alpha_N$ for all $n\geq n_0$ and $\widetilde{u}_n = \frac{u_n}{\|u_n\|_{W^{1,N}}}$. It follows from ($\mathcal{B}$)(b), \eqref{eq3.2}, H\"older's inequality, Corollary \ref{cor2.9} and \eqref{eq3.37} that for all $n\geq n_0$, we have
	\begin{equation}\label{eq3.58}
		\int_{B
			_R} g(\varepsilon x,u_n)u_n\,\mathrm{d}x \leq C\Bigg[\int_{B_R}g_R^N\,\mathrm{d}x+\bigg(\int_{B_R}g_R^{r\vartheta}\,\mathrm{d}x\bigg)^{\frac{1}{r}}\Bigg] <+\infty,
	\end{equation}
	where
	\begin{align*}
		C=\max\Bigg\{\tau ,\widetilde{\kappa}_\tau \bigg[\sup_{n\geq n_0}\int_{\mathbb{R}^N}\Phi\big(r'\alpha\|u_n\|^{N'}_{W^{1,N}}|\widetilde{u}_n|^{N'}\big)\,\mathrm{d}x\bigg]^{\frac{1}{r'}}\Bigg\}<+\infty,
	\end{align*}
	due to Lemma \ref{lem2.11}. Then, from \eqref{eq3.58}, we get that $\{g(\varepsilon x,u_n)u_n\}_{n\geq n_0}$ is bounded $L^1(B
		_R)$. Consequently, it is not difficult to verify that $\{g(\varepsilon x,u_n)u_n\}_{n\geq n_0}$ is uniformly absolutely integrable and tight over $B_R$. In virtue of \eqref{eq3.37}, we have $u_n\to u$ a.e.\,in $B_R$ as $n\to\infty$ and hence, $g(\varepsilon x,u_n)u_n\to g(\varepsilon x,u)u$ a.e.\,in $B
		_R$ as $n\to\infty$. Now, by applying Vitali's convergence theorem, we obtain
	\begin{equation}\label{eq3.59}
		\int_{B
			_R} g(\varepsilon x,u_n)u_n\,\mathrm{d}x \to   \int_{B
			_R} g(\varepsilon x,u)u\,\mathrm{d}x \quad\text{as }n\to\infty.
	\end{equation}
	Similarly, by using ($\mathcal{B}$)(b), \eqref{eq3.2}, H\"older's inequality, Corollary \ref{cor2.9} and \eqref{eq3.54}, we have for all $n\geq n_0$
	\begin{align*}
		\bigg|\int_{B^c
		_R} g(\varepsilon x,u_n)u_n\,\mathrm{d}x\bigg| & \leq C\Bigg[\int_{B^c_R}|u_n|^N\,\mathrm{d}x+\bigg(\int_{B^c_R}|u_n|^{r\vartheta}\,\mathrm{d}x\bigg)^{\frac{1}{r}}\Bigg]
		 <C \big(\xi+\xi^{\frac{1}{r}} \big),
	\end{align*}
	where $C$ is defined below \eqref{eq3.58}. It follows that
	\begin{align*}
		\limsup_{n\to\infty} \bigg|\int_{B^c
		_R} g(\varepsilon x,u_n)u_n\,\mathrm{d}x\bigg|\leq C \big(\xi+\xi^{\frac{1}{r}} \big).
	\end{align*}
	One can observe that $g(\varepsilon x,u)u\in L^1(\mathbb{R}^N)$. So, there exists $R>0$ large enough such that $\int_{B^c_R}g(\varepsilon x,u)u \,\mathrm{d}x<\xi. $ Consequently, we deduce that
	\begin{align*}
		\limsup_{n\to\infty} \bigg|\int_{B^c
		_R} g(\varepsilon x,u_n)u_n\,\mathrm{d}x-\displaystyle\int_{B^c
		_R}g(\varepsilon x,u)u \,\mathrm{d}x\bigg|< \widehat{C} \big(\xi+\xi^{\frac{1}{r}} \big),
	\end{align*}
	for some suitable constant $\widehat{C}>0$. Now, letting $\xi\to 0^+$ in the above inequality, we obtain
	\begin{equation}
		\label{eq3.60}
		\int_{B^c_R} g(\varepsilon x,u_n)u_n\,\mathrm{d}x \to   \int_{B^c_R} g(\varepsilon x,u)u\,\mathrm{d}x \quad\text{as }n\to\infty.
	\end{equation}
	Combining \eqref{eq3.59} and \eqref{eq3.60} together, we get
	\begin{equation}\label{eq3.61}
		\int_{\mathbb{R}^N} g(\varepsilon x,u_n)u_n\,\mathrm{d}x \to   \int_{\mathbb{R}^N} g(\varepsilon x,u)u\,\mathrm{d}x \quad\text{as }n\to\infty.
	\end{equation}
	Using the fact that $\langle{J'_{\varepsilon}(u_n),u_n}\rangle=o_n(1)$ as $n\to\infty$,  we have as $n\to\infty$
	\begin{equation}\label{eq3.62}
		\begin{aligned}
			&\|u_n\|_{W^{1,p}_{V_\varepsilon}}^p+ \|u_n\|_{W^{1,N}_{V_\varepsilon}}^N+\int_{\mathbb{R}^N}\mathcal{H}'_1(u_n)u_n\,\mathrm{d}x\\
			&=\int_{\mathbb{R}^N}G'_2(\varepsilon x,u_n)u_n\,\mathrm{d}x+\int_{\mathbb{R}^N}g(\varepsilon x,u_n)u_n\,\mathrm{d}x-\int_{\mathbb{R}^N}|u_n|^N\,\mathrm{d}x+o_n(1).
		\end{aligned}
	\end{equation}
	Therefore, by using \eqref{eq3.53}, \eqref{eq3.57} and \eqref{eq3.61}, we obtain from \eqref{eq3.52} and \eqref{eq3.62} that
	\begin{align*}
		\|u_n\|_{W^{1,p}_{V_\varepsilon}}^p+ \|u_n\|_{W^{1,N}_{V_\varepsilon}}^N+\int_{\mathbb{R}^N}\mathcal{H}'_1(u_n)u_n\,\mathrm{d}x=\|u\|_{W^{1,p}_{V_\varepsilon}}^p+ \|u\|_{W^{1,N}_{V_\varepsilon}}^N+\int_{\mathbb{R}^N}\mathcal{H}'_1(u)u\,\mathrm{d}x+o_n(1)
	\end{align*}
	as $n\to\infty$.
	From this, one has
	\begin{equation}\label{eq3.63}
		\|u_n\|_{W^{1,t}_{V_\varepsilon}}^t\to \|u\|_{W^{1,t}_{V_\varepsilon}}^t\quad\text{for } t\in\{p,N\}
		\quad\text{and}\quad
		\int_{\mathbb{R}^N}\mathcal{H}'_1(u_n)u_n\,\mathrm{d}x\to\int_{\mathbb{R}^N}\mathcal{H}'_1(u)u\,\mathrm{d}x \quad\text{as }~n\to\infty.
	\end{equation}
	Consequently, by using \eqref{eq3.37} and Corollary A.2 of Autuori-Pucci  \cite{Autuori-Pucci-2013}, we obtain from \eqref{eq3.63} that $u_n\to u$ in $W^{1,t}_{V_\varepsilon}(\mathbb{R}^N)$ as $n\to\infty$ for $t\in\{p,N\}$. By Remark \ref{rem2.99} and \eqref{eq2.2}, one can notice that $0\leq \mathcal{H}_1(s)\leq \mathcal{H}'_1(s)s $ for all $s\in\mathbb{R}$. Using this fact together with the generalized dominated convergence theorem of Lebesgue, see Royden \cite[Theorem 19]{Royden-1988} and \eqref{eq3.63}, we conclude that
	\begin{align*}
		\int_{\mathbb{R}^N}\mathcal{H}_1(u_n)\,\mathrm{d}x\to\int_{\mathbb{R}^N}\mathcal{H}_1(u)\,\mathrm{d}x \quad\text{as }n\to\infty.
	\end{align*}
	Recall that $\mathcal{H}_1$ is a $N$-function, which satisfies the $\Delta_2$-condition. Hence, by using a Br\'ezis-Lieb type result found in Alves-da Silva \cite[Proposition 2.2]{Alves-daSilva-2024}, we can prove that
	\begin{align*}
		\int_{\mathbb{R}^N}\mathcal{H}_1(u_n-u)\,\mathrm{d}x\to 0 \quad\text{as }n\to\infty.
	\end{align*}
	This shows that $u_n\to u$ in $L^{\mathcal{H}_1}(\mathbb{R}^N)$ as $n\to\infty$ and hence, $u_n\to u$ in $\mathbf{Y}_\varepsilon$ as $n\to\infty$. This finishes the proof.
\end{proof}

Now, we can give the main result in this section.

\begin{theorem}\label{thm3.8}
	For each $\varepsilon>0$, the functional $J_\varepsilon$ has a nontrivial critical point $u_\varepsilon\in \mathbf{Y}_\varepsilon $ such that $J_\varepsilon(u_\varepsilon)=c_\varepsilon$, where $c_\varepsilon$ denotes the mountain pass level associated with $J_\varepsilon$. Consequently, we deduce that $u_\varepsilon$ is a solution of \eqref{main problem@@}.
\end{theorem}

\begin{proof}
	Due to Lemmas \ref{lem3.1}, \ref{lem3.2} and \ref{lem3.7}, the functional $J_\varepsilon$ fulfills the geometry of the mountain pass theorem, see \cite[p.\,142]{Pucci-Serrin-1985}. Consequently, for each $\varepsilon>0$, there exists a nontrivial critical point $u_\varepsilon\in \mathbf{Y}_\varepsilon $ of $J_\varepsilon$ such that $J_\varepsilon(u_\varepsilon)=c_\varepsilon$, where the mountain pass level $c_\varepsilon$ is characterized as
	\begin{align*}
		c_\varepsilon=\inf_{\gamma\in\Gamma_\varepsilon}\max_{t\in[0,1]}J_\varepsilon(\gamma(t))
	\end{align*}
	with
	\begin{align*}
		\Gamma_\varepsilon=\big\{\gamma\in C\big([0,1],\mathbf{Y}_\varepsilon\big)\colon \gamma(0)=0~~\text{and}~~J_\varepsilon(\gamma(1))<0\big\}.
	\end{align*}
	The proof is now complete.
\end{proof}

\section{Existence of positive solution to the main problem via Nehari manifold method}\label{sec4}

This section is devoted to the study of the existence of a positive solution of \eqref{main problem} by using the Nehari manifold technique and the characterization of the mountain pass levels $c_\varepsilon$ as $\varepsilon\to 0^+$. To this end, we define the Nehari manifold associated with the functional $J_\varepsilon$, which is defined by
\begin{align*}
	\mathcal{N}_\varepsilon=\left\{u\in\mathbf{Y}_\varepsilon\setminus\{0\}\colon ~\langle{J'_\varepsilon(u),u}\rangle=0\right\}.
\end{align*}
A very nice introduction to this method has been done by Szulkin-Weth \cite{Szulkin-Weth-2010}. It is obvious that $\mathcal{N}_\varepsilon$ contains all nontrivial critical points of $J_\varepsilon$. Define
\begin{align*}
	c_\varepsilon=\inf_{u\in\mathcal{N}_\varepsilon} J_\varepsilon(u).
\end{align*}
Note that, if $c_\varepsilon$ is achieved by some $v_\varepsilon\in\mathcal{N}_\varepsilon$, then we say that $v_\varepsilon$ is a critical point of $J_\varepsilon$. In addition, since $c_\varepsilon$ is the lowest level for $J_\varepsilon$, therefore, $v_\varepsilon$ is said to be a ground state solution of \eqref{main problem@@}. Furthermore, we define
\begin{equation}\label{eq4.99}
	\begin{aligned}
		\Psi_\varepsilon(u) =\frac{1}{N}\langle{J'_\varepsilon(u),u}\rangle
		& =J_\varepsilon(u) - \bigg(\frac{1}{p}-\frac{1}{N}\bigg)\|u\|^p_{W^{1,p}_{V_\varepsilon}}-\frac{1}{N} \|u\|_N^N-\int_{\mathbb{R}^N} \Big(\frac{1}{N}g(\varepsilon x,u)u-\mathcal{G}(\varepsilon x,u)\Big)\,\mathrm{d}x  \\
		& \qquad-\bigg[\int_{\mathbb{R}^N} \Big[\mathcal{H}_2(u)-\frac{1}{N}\mathcal{H}'_2(u)u+\frac{1}{N}G_2'(\varepsilon x,u)u-G_2(\varepsilon x,u)\Big]\,\mathrm{d}x\bigg].
	\end{aligned}
\end{equation}
It follows that
\begin{align*}
	\mathcal{N}_\varepsilon=\Psi_\varepsilon^{-1} (\{0\}).
\end{align*}

Now, we prove some properties of $\mathcal{N}_\varepsilon$, which will be used in the sequel of this paper.

\begin{proposition}
	\label{prop4.1}
	The set $\mathcal{N}_\varepsilon$ is bounded away from the origin, that is, there exists a constant $\beta>0$ such that $\|u\|_{\mathbf{Y}_\varepsilon}\geq \beta>0$ for all $u\in \mathcal{N}_\varepsilon$ and for all $\varepsilon>0$.
\end{proposition}

\begin{proof}
	By Lemma \ref{lem2.2}, one can easily see that
	\begin{equation}
		\label{eq4.1}
		\|u\|_{W^{1,N}}\leq \big(\min\{1,V_0\}\big)^{-\frac{1}{N}} \|u\|_{W^{1,N}_{V_\varepsilon}}\leq \big(\min\{1,V_0\}\big)^{-\frac{1}{N}} \|u\|_{\mathbf{Y}_\varepsilon}.
	\end{equation}
	If $\|u\|_{\mathbf{Y}_\varepsilon}\geq \big(\min\{1,V_0\}\big)^{\frac{1}{N}}$, the conclusion is obvious. If $\|u\|_{\mathbf{Y}_\varepsilon}< \big(\min\{1,V_0\}\big)^{\frac{1}{N}}$, then one has $\|u\|_{W^{1,N}}<1$, thanks to \eqref{eq4.1}. Choose $r,r'>1$ satisfying $\frac{1}{r}+\frac{1}{r'}=1$. Further, suppose $\alpha>\alpha_0$ close to $\alpha_0$ and $r'>1$ close to $1$ such that $r' \alpha\leq \alpha_N$ holds. Now, by using ($\mathcal{B}$)(b), \eqref{eq3.2}, H\"older's inequality, Corollary \ref{cor2.9} and Lemma \ref{lem8.98}, we can deduce that
	\begin{align}\label{eq4.2}
		\int_{\mathbb{R}^N} |g(\varepsilon x,u)u|\,\mathrm{d}x \leq \tau S^{-N}_N\|u\|^N_{\mathbf{Y}_\varepsilon}+\widetilde{\kappa}_\tau S^{-\vartheta}_{\vartheta r} D \|u\|^\vartheta_{\mathbf{Y}_\varepsilon},
	\end{align}
	where, due to Lemma \ref{lem2.11}, we have
	\begin{align*}
		D=\bigg(\sup_{u\in \mathbf{Y}_\varepsilon}\bigg\{\int_{\mathbb{R}^N} \Phi(r' \alpha |u|^{N'})\,\mathrm{d}x\colon\|u\|_{\mathbf{Y}_\varepsilon}< \big(\min\{1,V_0\}\big)^{\frac{1}{N}}\bigg\}\bigg)^{\frac{1}{r'}}<+\infty,
	\end{align*}
	for $\vartheta>N$ and for all $u\in \mathcal{N}_\varepsilon $ with $\|u\|_{\mathbf{Y}_\varepsilon}$ small enough. Suppose by contradiction that $\{u_n\}_{n\in\mathbb{N}}\subset \mathcal{N}_\varepsilon $ is a sequence such that $\|u_n\|_{\mathbf{Y}_\varepsilon}\to 0$ as $n\to\infty$. Consequently, we can see that \eqref{eq4.2} holds whenever $u$ is replaced by $u_n$ for $n$ large enough. From the definition of $\mathcal{N}_\varepsilon $, we have $\langle{J'_\varepsilon(u_n),u_n}\rangle=0$ for each $n\in\mathbb{N}$. It follows at once that
	\begin{equation}\label{eq4.3}
		\|u_n\|_{W^{1,p}_{V_\varepsilon}}^p+ \|u_n\|_{W^{1,N}_{V_\varepsilon}}^N+\|u_n\|^N_N+\int_{\mathbb{R}^N}\mathcal{H}'_1(u_n)u_n\,\mathrm{d}x=\int_{\mathbb{R}^N}G'_2(\varepsilon x,u_n)u_n\,\mathrm{d}x+\int_{\mathbb{R}^N}g(\varepsilon x,u_n)u_n\,\mathrm{d}x.
	\end{equation}
	Note that $\|u_n\|_{W^{1,p}_{V_\varepsilon}}$ and $\|u_n\|_{\mathcal{H}_1}$ are small enough for sufficiently large $n$. Consequently, by using the fact that $0\leq \mathcal{H}_1(s)\leq \mathcal{H}'_1(s)s $ for all $s\in\mathbb{R}$ and Lemma \ref{lem2.899}, we obtain for $n$ large enough that
	\begin{equation}
		\label{eq4.4}
		\|u_n\|_{W^{1,p}_{V_\varepsilon}}^p\geq \|u_n\|_{W^{1,p}_{V_\varepsilon}}^N
		\quad\text{and}\quad
		\int_{\mathbb{R}^N}\mathcal{H}'_1(u_n)u_n\,\mathrm{d}x\geq \|u_n\|^N_{\mathcal{H}_1}.
	\end{equation}
	It follows from ($\mathcal{A}$)(a) and Lemma \ref{lem8.98} that
	\begin{equation}\label{eq4.5}
		\int_{\mathbb{R}^N}|G'_2(\varepsilon x,u_n)u_n|\,\mathrm{d}x\leq \ell S^{-N}_N \|u_n\|^N_{\mathbf{Y}_\varepsilon}+CS^{-q}_q \|u_n\|^q_{\mathbf{Y}_\varepsilon}.
	\end{equation}
	In view of \eqref{eq4.1},  \eqref{eq4.2}, \eqref{eq4.3}, \eqref{eq4.4} and \eqref{eq4.5}, we get for $n$ large enough that
	\begin{align*}
		(\tau+\ell)S^{-N}_N \|u_n\|^N_{\mathbf{Y}_\varepsilon}+CS^{-q}_q \|u_n\|^q_{\mathbf{Y}_\varepsilon}+\widetilde{\kappa}_\tau S^{-\vartheta}_{\vartheta r} D \|u_n\|^\vartheta_{\mathbf{Y}_\varepsilon} \geq 3^{1-N}\|u_n\|^N_{\mathbf{Y}_\varepsilon}.
	\end{align*}
	Due to the arbitrariness of $\tau$ and the fact that $\ell$ is very small, we deduce from the above inequality that there exist constants $C_1, C_2, C_3>0$ such that for $n$ large enough there holds
	\begin{align*}
		C_1\|u_n\|^q_{\mathbf{Y}_\varepsilon}+C_2 \|u_n\|^\vartheta_{\mathbf{Y}_\varepsilon} \geq C_3\|u_n\|^N_{\mathbf{Y}_\varepsilon}.
	\end{align*}
	Dividing $\|u_n\|^N_{\mathbf{Y}_\varepsilon}$ on both sides of the above inequality and letting $n\to\infty$, we get $0\geq C_3>0$, which is a contradiction. It follows that there exists $\beta>0$ such that $\|u\|_{\mathbf{Y}_\varepsilon}\geq \beta>0$ for all $u\in \mathcal{N}_\varepsilon$ and for all $\varepsilon>0$.
\end{proof}

\begin{lemma}\label{lem4.2}
	For each $u\in\mathcal{O}_\varepsilon=\ \{u\in\mathbf{Y}_\varepsilon\colon|\operatorname{supp}(|u|)\cap \Lambda_\varepsilon|>0\}\setminus\{0\}$, there exists a unique $t_u>0$ such that $t_u u\in \mathcal{N}_\varepsilon $. In particular, for $u\in\mathcal{N}_\varepsilon\cap \mathcal{O}_\varepsilon $, we have $J_\varepsilon(u)=\max_{t\geq 0}J_\varepsilon(tu)$. Consequently, if $u\in\mathcal{N}_\varepsilon$, then  $u\in\mathcal{O}_\varepsilon$.
\end{lemma}

\begin{proof}
	We define the function $h(t)=J_\varepsilon(tu)$ for $t>0$ and $u\in\mathcal{O}_\varepsilon\setminus\{0\}$. From Lemmas \ref{lem3.1} and \ref{lem3.2}, we have $h(0)=0$, $h(t)>0$ for $t>0$ sufficiently small and $h(t)<0$ for $t>0$ sufficiently large. It follows that $\max_{t>0} h(t) $ is achieved at some point $t=t_u$ such that $h'(t_u)=0 $ and $t_u u\in \mathcal{N}_\varepsilon $. Next, we claim that $t_u$ is the unique critical point of $h$. Assume by contradiction that there exist $t_1$ and $t_2$ with $0<t_1<t_2$ such that $h'(t_1)=h'(t_2)=0 $. Consequently, we obtain for $i=1,2$ that
	\begin{align*}
		\frac{1}{t_i^{N-p}}\|u\|^p_{W^{1,p}_{V_\varepsilon}}+\|u\|^N_{W^{1,N}_{V_\varepsilon}}+\|u\|^N_N +\int_{\mathbb{R}^N}\frac{\mathcal{H}'_1(t_iu)u}{t_i^{N-1}}\,\mathrm{d}x-\int_{\mathbb{R}^N}\frac{G'_2(\varepsilon x,t_iu)u}{t_i^{N-1}}\,\mathrm{d}x-\int_{\mathbb{R}^N}\frac{g(\varepsilon x,t_iu)u}{t_i^{N-1}}\,\mathrm{d}x=0.
	\end{align*}
	Subtracting the equalities above gives
	\begin{align*}
		 & \bigg(\frac{1}{t_1^{N-p}}-\frac{1}{t_2^{N-p}}\bigg) \|u\|^p_{W^{1,p}_{V_\varepsilon}}+\int_{\mathbb{R}^N}\bigg(\frac{G'_2(\varepsilon x,t_2u)u}{t_2^{N-1}}-\frac{G'_2(\varepsilon x,t_1u)u}{t_1^{N-1}}\bigg)\,\mathrm{d}x+\int_{\mathbb{R}^N}\bigg(\frac{g(\varepsilon x,t_2u)u}{t_2^{N-1}}-\frac{g(\varepsilon x,t_1u)u}{t_1^{N-1}}\bigg)\,\mathrm{d}x  \\
		 & =\int_{\mathbb{R}^N}\bigg(\frac{\mathcal{H}'_1(t_2u)u}{t_2^{N-1}}-\frac{\mathcal{H}'_1(t_1u)u}{t_1^{N-1}}\bigg)\,\mathrm{d}x=:J.
	\end{align*}
	Now, it follows from the definitions of $G'_2$ and $g$ together with ($\mathcal{A}$)(e) as well as ($\mathcal{B}$)(e) that the right-hand side of the above equality is positive. On the other hand, we have
	\begin{align*}
		J
		&=\int_{\big\{x\in\mathbb{R}^N\colon|u|<\frac{(N-1)\delta}{t_2}\big\}}\bigg(\frac{\mathcal{H}'_1(t_2u)u}{t_2^{N-1}}-\frac{\mathcal{H}'_1(t_1u)u}{t_1^{N-1}}\bigg)\,\mathrm{d}x\\
		&\quad +\int_{\big\{x\in\mathbb{R}^N\colon\frac{(N-1)\delta}{t_2}<|u|<\frac{(N-1)\delta}{t_1}\big\}}\bigg(\frac{\mathcal{H}'_1(t_2u)u}{t_2^{N-1}}-\frac{\mathcal{H}'_1(t_1u)u}{t_1^{N-1}}\bigg)\,\mathrm{d}x  \\
		& \quad+ \int_{\big\{x\in\mathbb{R}^N\colon|u|>\frac{(N-1)\delta}{t_1}\big\}}\bigg(\frac{\mathcal{H}'_1(t_2u)u}{t_2^{N-1}}-\frac{\mathcal{H}'_1(t_1u)u}{t_1^{N-1}}\bigg)\,\mathrm{d}x \\
		& =\int_{\big\{x\in\mathbb{R}^N\colon|u|<\frac{(N-1) \delta}{t_2}\big\}}|u|^N\log\bigg(\frac{t_1}{t_2}\bigg)^N\,\mathrm{d}x+\int_{\big\{x\in\mathbb{R}^N\colon|u|>\frac{(N-1)\delta}{t_1}\big\}} N(N-1)\delta|u|^{N-1}\bigg(\frac{1} {t_2}-\frac{1}{t_1}\bigg)\,\mathrm{d}x \\
		& \quad+\int_{\big\{x\in\mathbb{R}^N\colon\frac{(N-1) \delta}{t_2}<|u|<\frac{(N-1)\delta}{t_1}\big\}}\bigg[|u|^N\log\bigg(\frac{|t_1u|}{(N-1)\delta}\bigg)^N+N|u|^{N-1}\bigg(\frac{(N-1)\delta}{t_2}-|u|\bigg)\bigg]\,\mathrm{d}x<0,
	\end{align*}
	which is a contradiction and thus, $t_u$ must be unique. Next, we have to prove that if $u\in\mathcal{N}_\varepsilon$, then  $u\in\mathcal{O}_\varepsilon$. Indeed, if not, then for $|\operatorname{supp}(|u|)\cap \Lambda_\varepsilon|=0$ and  $u\in\mathcal{N}_\varepsilon$, we obtain by using (V1), $V_0+1\geq 2(\ell+\ell')$, $\mathcal{H}_1(u)\leq\mathcal{H}'_1(u)u$, ($\mathcal{A}$)(c), ($\mathcal{B}$)(d), and Lemma \ref{lem2.899} that
	\begin{align*}
		\|u\|_{W^{1,p}_{V_\varepsilon}}^p+ \|u\|_{W^{1,N}_{V_\varepsilon}}^N+\|u\|^N_N+\min\big\{\|u\|_{\mathcal{H}_1}^l,\|u\|_{\mathcal{H}_1}^N\big\} & \leq (\ell+\ell')\int_{\mathbb{R}^N}|u|^N\,\mathrm{d}x\leq \frac{1}{2}\int_{\mathbb{R}^N}\big(V(\varepsilon x)+1\big)|u|^N\,\mathrm{d}x  \\
		& \leq \frac{1}{2}\big( \|u\|_{W^{1,p}_{V_\varepsilon}}^p+ \|u\|_{W^{1,N}_{V_\varepsilon}}^N+\|u\|^N_N\big).
	\end{align*}
	This shows that
	\begin{align*}
		0<\frac{1}{2}\big( \|u\|_{W^{1,p}_{V_\varepsilon}}^p+ \|u\|_{W^{1,N}_{V_\varepsilon}}^N+\|u\|^N_N\big)+\min\big\{\|u\|_{\mathcal{H}_1}^l,\|u\|_{\mathcal{H}_1}^N\big\}\leq 0.
	\end{align*}
	The above inequality implies $u=0$ in $\mathbf{Y}_\varepsilon$, which is a contradiction because $u\in\mathcal{N}_\varepsilon$. This finishes the proof.
\end{proof}

\begin{proposition}\label{prop4.3}
	The set $\mathcal{N}_\varepsilon$ is a $C^1$-manifold for each $\varepsilon>0$ . Moreover, the critical points of $J_\varepsilon\big|_{\mathcal{N}_\varepsilon}$ are critical points of $J_\varepsilon$ in $\mathbf{Y}_\varepsilon$.
\end{proposition}

\begin{proof}
	To prove that $\mathcal{N}_\varepsilon$ is a $C^1$-manifold, it is sufficient to show that $\langle{\Psi'_\varepsilon(u),u}\rangle\neq 0$ for all $u\in \mathcal{N}_\varepsilon$. Now, arguing by contradiction, we assume that there exists some  $u\in \mathcal{N}_\varepsilon$ such that $\langle{\Psi'_\varepsilon(u),u}\rangle= 0$. Consequently, by using $\langle{J'_\varepsilon(u),u}\rangle=0$, we obtain from \eqref{eq4.99} that
	\begin{align*}
		0 & = -p\bigg(\frac{1}{p}-\frac{1}{N}\bigg)\|u\|^p_{W^{1,p}_{V_\varepsilon}}-\int_{\mathbb{R}^N}|u|^N\,\mathrm{d}x-\int_{\mathbb{R}^N}\bigg(\frac{1}{N}g'(\varepsilon x,u)u^2-\bigg(\frac{N-1}{N}\bigg)g(\varepsilon x,u)u\bigg)\,\mathrm{d}x                                                            \\
		  & \quad - \int_{\mathbb{R}^N}\bigg(\bigg(\frac{N-1}{N}\bigg)\mathcal{H}_2'(u)u-\frac{1}{N} \mathcal{H}_2''(u)u^2\bigg)\,\mathrm{d}x- \int_{\mathbb{R}^N}\bigg(\frac{1}{N} G_2''(\varepsilon x,u)u^2-\bigg(\frac{N-1}{N}\bigg) G_2'(\varepsilon x,u)u\bigg)\,\mathrm{d}x.
	\end{align*}
	Observe that $g(\varepsilon x,u)=f(u)$ and $G_2'(\varepsilon x,u)=\mathcal{H}_2'(u)$ for all $x\in\Lambda_\varepsilon$. Now, simplifying the above equality, we obtain
	\begin{equation}\label{eq4.7}
		\begin{aligned}
			\int_{\Lambda_\varepsilon} |u|^N\mathrm{d}x & < -  \int_{\Lambda^c_\varepsilon}\bigg(|u|^N+\bigg(\frac{N-1}{N}\bigg)\mathcal{H}_2'(u)u-\frac{1}{N} \mathcal{H}_2''(u)u^2\bigg)\mathrm{d}x\\
			&\qquad- \int_{\Lambda^c_\varepsilon}\bigg(\frac{1}{N} G_2''(\varepsilon x,u)u^2-\bigg(\frac{N-1}{N}\bigg) G_2'(\varepsilon x,u)u\bigg)\mathrm{d}x  \\
			& \qquad -\int_{\Lambda_\varepsilon}\bigg(\frac{1}{N}f'(u)u^2-\bigg(\frac{N-1}{N}\bigg)f(u)u\bigg)\,\mathrm{d}x\\
			&\qquad -\int_{\Lambda^c_\varepsilon}\bigg(\frac{1}{N}g'(\varepsilon x,u)u^2-\bigg(\frac{N-1}{N}\bigg)g(\varepsilon x,u)u\bigg)\,\mathrm{d}x.
		\end{aligned}
	\end{equation}
	By the definition of $\mathcal{H}_2$ and using the assumption (f3), one can easily obtain
	\begin{align*}
		|u|^N+\bigg(\frac{N-1}{N}\bigg)\mathcal{H}_2'(u)u-\frac{1}{N} \mathcal{H}_2''(u)u^2\geq 0\quad \text{for a.e.\,}x\in \Lambda^c_\varepsilon
	\end{align*}
	and
	\begin{align*}
		\frac{1}{N}f'(u)u^2-\bigg(\frac{N-1}{N}\bigg)f(u)u\geq 0\quad \text{for a.e.\,}x\in \Lambda^c_\varepsilon.
	\end{align*}
	Note that
	\begin{align*}
		\Lambda^c_\varepsilon \cap \{u>0\}=[\Lambda^c_\varepsilon \cap \{0<u<t_1\}]\cup[\Lambda^c_\varepsilon \cap \{t_1\leq u\leq t_2\}]\cup [\Lambda^c_\varepsilon \cap \{u>t_2\}].
	\end{align*}
	Due to ($\mathcal{H}_2$)(c) and (f3), one has
	\begin{align*}
		\frac{1}{N} G_2''(\varepsilon x,u)u^2 -\bigg(\frac{N-1}{N}\bigg) G_2'(\varepsilon x,u)u=\frac{1}{N}\big(\mathcal{H}_2''(u)u^2-(N-1)\mathcal{H}_2'(u)u\big)\geq 0
	\end{align*}
	for a.a.\,$x\in \Lambda^c_\varepsilon \cap \{0<u<t_1\}$ and
	\begin{align*}
		\frac{1}{N} g'(\varepsilon x,u)u^2 -\bigg(\frac{N-1}{N}\bigg) g(\varepsilon x,u)u=\frac{1}{N}\big(f'(u)u^2-(N-1)f(u)u\big)\geq 0
	\end{align*}
	for a.a.\,$x\in \Lambda^c_\varepsilon \cap \{0<u<t_1\}$. Moreover, because of (h3) and ($\eta$3), we infer that
	\begin{align*}
		\frac{1}{N} G_2''(\varepsilon x,u)u^2 -\bigg(\frac{N-1}{N}\bigg) G_2'(\varepsilon x,u)u=\frac{1}{N}\big(h'(u)u^2-(N-1)h(u)u\big)\geq 0
	\end{align*}
	for a.a.\,$x\in \Lambda^c_\varepsilon \cap \{t_1\leq u\leq t_2\}$ and
	\begin{align*}
		\frac{1}{N} g'(\varepsilon x,u)u^2 -\bigg(\frac{N-1}{N}\bigg) g(\varepsilon x,u)u=\frac{1}{N}\big(\eta'(u)u^2-(N-1)\eta(u)u\big)\geq 0
	\end{align*}
	for a.a.\,$x\in \Lambda^c_\varepsilon \cap \{t_1\leq u\leq t_2\}$.

	By using the definitions of $\widehat{\mathcal{H}}'_2$ and $\widehat{f}$, we have
	\begin{align*}
		\frac{1}{N} G_2''(\varepsilon x,u)u^2 -\bigg(\frac{N-1}{N}\bigg) G_2'(\varepsilon x,u)u=\frac{1}{N} g'(\varepsilon x,u)u^2 -\bigg(\frac{N-1}{N}\bigg) g(\varepsilon x,u)u= 0
	\end{align*}
	for a.a.\,$x\in \Lambda^c_\varepsilon \cap \{ u> t_2\}$. Consequently, we deduce that
	\begin{equation}\label{eq4.8}
		\int_{\Lambda^c_\varepsilon \cap \{u>0\}}\bigg(\frac{1}{N} G_2''(\varepsilon x,u)u^2-\bigg(\frac{N-1}{N}\bigg) G_2'(\varepsilon x,u)u\bigg)\,\mathrm{d}x\geq 0
	\end{equation}
	and
	\begin{equation}\label{eq4.9}
		\int_{\Lambda^c_\varepsilon \cap \{u>0\}}\bigg(\frac{1}{N}g'(\varepsilon x,u)u^2-\bigg(\frac{N-1}{N}\bigg)g(\varepsilon x,u)u\bigg)\,\mathrm{d}x\geq 0.
	\end{equation}
	From the definitions of $G_2'$ and \eqref{eq4.8}, we have
	\begin{align*}
		&\int_{\Lambda^c_\varepsilon \cap \{u<0\}}\bigg(\frac{1}{N} G_2''(\varepsilon x,u)u^2-\bigg(\frac{N-1}{N}\bigg) G_2'(\varepsilon x,u)u\bigg)\mathrm{d}x\\
		& = \int_{\Lambda^c_\varepsilon \cap \{-u>0\}}\bigg(\frac{1}{N} G_2''(\varepsilon x,-u)(-u)^2-\bigg(\frac{N-1}{N}\bigg) G_2'(\varepsilon x,-u)(-u)\bigg)\mathrm{d}x  \\  & \geq 0.
	\end{align*}
	Similarly, from the definition of $g$ and \eqref{eq4.9}, we have
	\begin{align*}
		&\int_{\Lambda^c_\varepsilon \cap \{u<0\}}\bigg(\frac{1}{N} g'(\varepsilon x,u)u^2-\bigg(\frac{N-1}{N}\bigg) g(\varepsilon x,u)u\bigg)\mathrm{d}x\\
		& = \int_{\Lambda^c_\varepsilon \cap \{-u>0\}}\bigg(\frac{1}{N} g'(\varepsilon x,-u)(-u)^2-\bigg(\frac{N-1}{N}\bigg) g(\varepsilon x,-u)(-u)\bigg)\mathrm{d}x  \\  & \geq 0.
	\end{align*}
	Combining all these facts, we obtain from \eqref{eq4.7} that
	$ 0< \int_{\Lambda_\varepsilon} |u|^N\mathrm{d}x<0.$ This yields that $u=0$ a.e.\,in $\Lambda_\varepsilon$. Moreover, since $\langle{J'_\varepsilon(u),u}\rangle=0$, arguing similarly as in Lemma \ref{lem4.2}, we can obtain
	\begin{align*}
		0<\frac{1}{2}\big( \|u\|_{W^{1,p}_{V_\varepsilon}}^p+ \|u\|_{W^{1,N}_{V_\varepsilon}}^N+\|u\|^N_N\big)+\min\big\{\|u\|_{\mathcal{H}_1}^l,\|u\|_{\mathcal{H}_1}^N\big\}\leq 0.
	\end{align*}
	This ensures that $u=0$ in $\mathbf{Y}_\varepsilon$, which is a contradiction because $u\in\mathcal{N}_\varepsilon$. It follows that $\langle{\Psi'_\varepsilon(u),u}\rangle\neq 0$ for all $u\in \mathcal{N}_\varepsilon$.

	Now, let $u\in \mathcal{N}_\varepsilon $ be a critical point of $J_\varepsilon$ constrained to $\mathcal{N}_\varepsilon$. Due to the application of Lagrange's multiplier rule, we infer that
	\begin{align*}
		J_\varepsilon'(u)=\lambda \Psi_\varepsilon'(u)\quad\text{in }\mathbf{Y}_\varepsilon^* \text{ for some } \lambda\in\mathbb{R}.
	\end{align*}
	Because of $u\in \mathcal{N}_\varepsilon$, we obtain from the above relation that
	\begin{align*}
		0=\langle{J_\varepsilon'(u),u}\rangle=\lambda\langle{\Psi_\varepsilon'(u),u}\rangle.
	\end{align*}
	It follows that $\lambda=0$ and $J_\varepsilon'(u)=0~\text{in}~\mathbf{Y}_\varepsilon^*$, that is, $u$ is a critical point of $J_\varepsilon$ on $\mathbf{Y}_\varepsilon$. This completes the proof.
\end{proof}

The last proposition implies at once that a critical point of $J_\varepsilon\big|_{\mathcal{N}_\varepsilon}$ is a point $u\in \mathbf{Y}_\varepsilon$ such that
\begin{equation}\label{eq4.10}
	\|J_\varepsilon'(u)\|_{\mathbf{Y}_\varepsilon^*}=\min_{\lambda\in\mathbb{R}} \|J_\varepsilon'(u)-\lambda \Psi_\varepsilon'(u)\|=0,
\end{equation}
thanks to Proposition 5.12 of Willem \cite{Willem-1996}. Now, we recall that a \textnormal{(PS)$_c$} sequence for $J_\varepsilon\big|_{\mathcal{N}_\varepsilon}$ is a sequence $\{u_n\}_{n\in\mathbb{N}}$ in $\mathcal{N}_\varepsilon$ such that
\begin{equation}\label{eq4.11}
	J_\varepsilon(u_n)\to c\quad\text{and}\quad\|J_\varepsilon'(u_n)\|_{\mathbf{Y}_\varepsilon^*}\to 0\quad\text{as }n\to\infty.
\end{equation}
We say that $J_\varepsilon\big|_{\mathcal{N}_\varepsilon}$ satisfies the \textnormal{(PS)} condition when each \textnormal{(PS)$_c$} sequence for $J_\varepsilon\big|_{\mathcal{N}_\varepsilon}$ has a convergent subsequence for any $c\in\mathbb{R}$.

\begin{proposition}\label{prop4.4}
	If $\{u_n\}_{n\in\mathbb{N}}\subset \mathcal{N}_\varepsilon$ is a \textnormal{(PS)$_c$} sequence for $J_\varepsilon$ satisfying $\displaystyle\limsup_{n\to\infty} \|u_n\|^{N'}_{W^{1,N}}<\frac{\alpha_N}{\alpha_0}$, then $J_\varepsilon\big|_{\mathcal{N}_\varepsilon}$ satisfies the \textnormal{(PS)} condition.
\end{proposition}

\begin{proof}
	Let $\{u_n\}_{n\in\mathbb{N}}\subset \mathcal{N}_\varepsilon$ be a \textnormal{(PS)$_c$} sequence for $J_\varepsilon$ satisfying $\displaystyle\limsup_{n\to\infty} \|u_n\|^{N'}_{W^{1,N}}<\frac{\alpha_N}{\alpha_0}$. Therefore, we infer that \eqref{eq4.11} is satisfied. Further, in view of  \eqref{eq4.10}, there exists a real sequence $\{\lambda_n\}_{n\in\mathbb{N}}\subset\mathbb{R}$ such that
	\begin{align*}
		J_\varepsilon'(u_n)=\lambda_n \Psi_\varepsilon'(u_n)+o_n(1)\quad\text{as }n\to\infty.
	\end{align*}
	Since $u_n\in \mathcal{N}_\varepsilon$ for each $n\in\mathbb{N}$, we have $ J_\varepsilon(u_n)\to c$ as $n\to\infty$ and $\langle{J_\varepsilon'(u_n),u_n}\rangle=0$. Repeating the same arguments as in Lemma \ref{lem3.4}, one can see that $\{u_n\}_{n\in\mathbb{N}}$ is a bounded sequence. In addition to this, due to Lemma \ref{lem3.7}, we only have proved that $\{u_n\}_{n\in\mathbb{N}}$ is a \textnormal{(PS)$_c$} sequence for $J_\varepsilon$. To finish the proof, it is sufficient to show that $\lambda_n\to 0$ as $n\to\infty$.
	It is obvious that the sequence $\{u_n\}_{n\in\mathbb{N}}$ satisfies the following relation:
	\begin{equation}\label{eq4.12}
		0=\langle{J_\varepsilon'(u_n),u_n}\rangle=\lambda_n\langle{\Psi_\varepsilon'(u_n),u_n}\rangle+o_n(1)\quad\text{as }n\to\infty.
	\end{equation}
	Due to Proposition \ref{prop4.3}, if $\langle{\Psi_\varepsilon'(u_n),u_n}\rangle=o_n(1)$ as $n\to\infty$, then one has $u_n\to 0$ in $L^N(\Lambda_\varepsilon)$ as $n\to\infty$. From the boundedness of $\{u_n\}_{n\in\mathbb{N}}$ in $\mathbf{Y}_\varepsilon$, we deduce from Lemma \ref{lem8.98} and the interpolation inequality that
	\begin{equation}\label{eq4.13}
		u_n\to 0
		\quad\text{in } L^\theta(\Lambda_\varepsilon) \text{ for } \theta\in [N,+\infty) \quad\text{as }n\to\infty.
	\end{equation}
	From the hypothesis, it follows that there exist $m>0$ and $n_0\in\mathbb{N}$ large enough such that $\|u_n\|^{N'}_{W^{1,N}}<m<\frac{\alpha_N}{\alpha_0}$ for all $n\geq n_0$. Take $r>1$ with $r'=\frac{r}{r-1}>1$ and satisfying $\frac{1}{r}+\frac{1}{r'}=1$ . Let $r'$ be close to 1 and $\alpha>\alpha_0$ close to $\alpha_0$ such that we still have $r'\alpha\|u_n\|^{N'}_{W^{1,N}}<m<\alpha_N$ for all $n\geq n_0$ and $\widetilde{u}_n = \frac{u_n}{\|u_n\|_{W^{1,N}}}$. It follows from ($\mathcal{B}$)(b), \eqref{eq3.2}, H\"older's inequality, Corollary \ref{cor2.9} and \eqref{eq4.13} that
	\begin{align*}
		\Bigg|\int_{\Lambda_\varepsilon} g(\varepsilon x,u_n)u_n\,\mathrm{d}x\Bigg|\leq \tau \int_{\Lambda_\varepsilon} |u_n|^N\,\mathrm{d}x+C\bigg(\int_{\Lambda_\varepsilon} |u_n|^{\vartheta r}\,\mathrm{d}x\bigg)^{\frac{1}{r}}\to 0\quad\text{as }n\to\infty,
	\end{align*}
	where
	\begin{align*}
		C=\widetilde{\kappa}_\tau \bigg[\sup_{n\geq n_0}\int_{\mathbb{R}^N}\Phi\big(r'\alpha\|u_n\|^{N'}_{W^{1,N}}|\widetilde{u}_n|^{N'}\big)\,\mathrm{d}x\bigg]^{\frac{1}{r'}}<+\infty,
	\end{align*}
	due to Lemma \ref{lem2.11}. Consequently, we get
	\begin{equation}\label{eq4.14}
		\lim_{n\to\infty} \int_{\Lambda_\varepsilon} g(\varepsilon x,u_n)u_n\,\mathrm{d}x=0.
	\end{equation}
	Following similar arguments as in Lemma \ref{lem4.2} and using the facts in ($\mathcal{A}$)(a), (c) and ($\mathcal{B}$)(d) as well as \eqref{eq4.12}, \eqref{eq4.13} and \eqref{eq4.14}, one can easily obtain
	\begin{align*}
		0<\frac{1}{2}\big( \|u_n\|_{W^{1,p}_{V_\varepsilon}}^p+ \|u_n\|_{W^{1,N}_{V_\varepsilon}}^N+\|u_n\|^N_N\big)+\min\big\{\|u_n\|_{\mathcal{H}_1}^l,\|u_n\|_{\mathcal{H}_1}^N\big\}\leq o_n(1)\quad\text{as }n\to\infty.
	\end{align*}
	It follows that $u_n\to 0$ in $\mathbf{Y}_\varepsilon$ as $n\to\infty$, which contradicts Proposition \ref{prop4.3} (or Proposition \ref{prop4.1}) . Hence, we must have  $\lambda_n\to 0$ as $n\to\infty.$ This finishes the proof.
\end{proof}

\subsection{Existence and concentration phenomena of positive ground state solutions}

In this part, we first introduce the autonomous problem related to \eqref{main problem@} as follows:
\begin{align}\label{main problem@@@}\tag{$\mathcal{P}_{0}$}
	\begin{cases}
		\mathcal{L}_{p}(u)+\mathcal{L}_{N}(u)=|u|^{N-2}u \log |u|^N+f(u)~~\text{in}~~\mathbb{R}^N, \\[1ex]
		u\in \mathbf{Y}= W^{1,p}(\mathbb{R}^N)\cap W^{1,N}(\mathbb{R}^N)\cap L^{\mathcal{H}_1}(\mathbb{R}^N),
	\end{cases}
\end{align}
where
\begin{align*}
	\mathcal{L}_{t}(u)=-\Delta_t u+V_0|u|^{t-2}u\quad \text{for }t\in\{p,N\}.
\end{align*}
The corresponding energy functional associated to \eqref{main problem@@@} will be denoted by $J_0\colon\mathbf{Y}\to \mathbb{R}$ and is defined by
\begin{align*}
	J_0(u)=\frac{1}{p} \|u\|_{W^{1,p}_{V_0}}^p+\frac{1}{N} \Big(\|u\|_{W^{1,N}_{V_0}}^N+\|u\|^N_{N}\Big)+\int_{\mathbb{R}^N}\mathcal{H}_1(u)\,\mathrm{d}x-\int_{\mathbb{R}^N}\mathcal{H}_2(u)\,\mathrm{d}x-\int_{\mathbb{R}^N} F(u)\,\mathrm{d}x
\end{align*}
for all $u\in\mathbf{Y}$. By adopting a similar strategy as in \cite{Alves-Ambrosio-2024, Alves-deMoraisFilho-2018, Alves-Ji-2020, Ji-Szulkin-2016, Squassina-Szulkin-2015, Squassina-Szulkin-2017}, one sees that \eqref{main problem@@@} has a ground state solution $u_0$ which fulfills
\begin{align*}
	J_0(u_0)=c_0=\inf_{u\in\mathcal{N}_0} J_0(u)=\inf_{u\in\mathbf{Y}\setminus\{0\}}\max_{t\geq 0} J_0(tu),
\end{align*}
where $\mathcal{N}_0$ is the Nehari set associated with $J_0$ and defined by
\begin{align*}
	\mathcal{N}_0=\left\{u\in\mathbf{Y}\setminus\{0\}\colon \langle{J'_0(u),u}\rangle=0\right\}.
\end{align*}
Recall that $\|\cdot\|_{W^{1,t}_{V_0}}$ and $\|\cdot\|_{W^{1,t}}$ are two equivalent norms on the Banach space $W^{1,t}(\mathbb{R}^N)$ for any $t\in\{p,N\}$.

In the next lemma, we demonstrate that the mountain pass level $c_\varepsilon$ is the ground state level of $J_\varepsilon$, and we establish an important relation between the two levels $c_\varepsilon$ and $c_0$.

\begin{lemma}\label{lem4.5}
	The following properties hold:
	\begin{enumerate}
		\item[\normalfont{(a)}]
			$c_\varepsilon\geq\beta>0$ for all $\varepsilon>0$,
		\item[\normalfont{(b)}]
			$ c_\varepsilon=\displaystyle\inf_{u\in\mathcal{N}_\varepsilon} J_\varepsilon(u) $ for all $\varepsilon>0,$
		\item[\normalfont{(c)}]
			$\displaystyle\limsup_{\varepsilon\to0}c_\varepsilon\leq c_0$.
	\end{enumerate}
\end{lemma}

\begin{proof}
	The assertion in (a) follows directly from the definition of $c_\varepsilon$ (see Theorem \ref{thm3.8}) and Lemma \ref{lem3.1}. Now, we prove (b). For this, observe that by Lemma \ref{lem4.2}, we have $u\in\mathcal{O}_\varepsilon$ for each $u\in \mathcal{N}_\varepsilon$. Further, by Lemma \ref{lem3.2}, there exists $t_0>0$ such that $J_\varepsilon(t_0 u)<0$. Define the map $\gamma_\varepsilon\colon[0,1]\to\mathbf{Y}_\varepsilon$ such that $\gamma_\varepsilon(t)=t(t_0u)$ for all $t\in [0,1]$. It is easy to see that $\gamma_\varepsilon\in\Gamma_\varepsilon$. Consequently, from the definition of $c_\varepsilon$, we have
	\begin{align*}
		c_\varepsilon\leq \max_{t\in[0,1]}J_\varepsilon(\gamma_\varepsilon(t))\leq \max_{t\geq 0}J_\varepsilon (tu)=J_\varepsilon (u).
	\end{align*}
	Due to the arbitrariness of $u\in \mathcal{N}_\varepsilon$ and from the definition of the infimum, we obtain from the above inequality that
	\begin{align}\label{eq4.15}
		c_\varepsilon\leq\displaystyle\inf_{u\in\mathcal{N}_\varepsilon} J_\varepsilon(u).
	\end{align}
	By Theorem \ref{thm3.8}, we have $u_\varepsilon\in \mathbf{Y}_\varepsilon\setminus\{0\}$, $J_\varepsilon(u_\varepsilon)=c_\varepsilon$ and $J'_\varepsilon(u_\varepsilon)=0$. It follows that $u_\varepsilon\in\mathcal{N}_\varepsilon$ and hence, we have
	\begin{equation}\label{eq4.16}
		\inf_{u\in\mathcal{N}_\varepsilon} J_\varepsilon(u)\leq J_\varepsilon(u_\varepsilon)=c_\varepsilon .
	\end{equation}
	Combining \eqref{eq4.15} and \eqref{eq4.16}, we get $ c_\varepsilon=\inf_{u\in\mathcal{N}_\varepsilon} J_\varepsilon(u)= J_\varepsilon(u_\varepsilon).$ This completes the proof of (b).

	Next, we aim to prove (c). To prove this, we assume that $u_0\in \mathcal{N}_0$ is a positive ground state solution of \eqref{main problem@@@}, that is, there hold
	\begin{align*}
		c_0=J_0(u_0)=\inf_{u\in\mathcal{N}_0} J_0(u)~~\text{and} ~~J'_0(u_0)=0.
	\end{align*}
	Choose $\phi\in C^\infty_{c}(\mathbb{R}^N)$ be such that $0\leq \phi\leq 1$, $\phi\equiv 1$ in $B_1$, and $\phi\equiv 0$ in $B^c_{2}$. For each $R>0$, let $B_{2R}\subset \Lambda_\varepsilon$ and define $\phi_R(\cdot)=\phi(\frac{\cdot}{R})$. Next, we set $u_R(\cdot)=u(\frac{\cdot}{R})$ and $u_R(x)=\phi_R(x)u_0(x)$. It follows that $0\leq u_R\leq u_0 $, $u_R\equiv u_0$ in $B_R$, and $u_R\equiv 0$ in $B^c_{2R}$. Note that $\operatorname{supp}(u_R)\subset B_{2R}\subset \Lambda_\varepsilon$ and hence, $u_R\in \mathcal{O}_\varepsilon $. By applying Lebesgue's dominated convergence theorem, we can easily deduce that
	\begin{equation}
		\label{eq4.17}
		u_R\to u_0\quad\text{in }W^{1,t}(\mathbb{R}^N)\quad\text{as } R\to\infty \text{ for } t\in\{p,N\}.
	\end{equation}
	It follows that $ u_R\to u_0$ a.e.\,in $\mathbb{R}^N$ as $R\to\infty$. Observe that $\mathcal{H}_1$ is nondecreasing for all $t\geq 0$. This implies at once that $\mathcal{H}_1(u_R)\leq \mathcal{H}_1(u_0)\in L^1(\mathbb{R}^N)$. Further, due to the continuity of $\mathcal{H}_1$, we infer that $\mathcal{H}_1(u_R)\to \mathcal{H}_1(u_0)$ a.e.\, $\mathbb{R}^N$ as $R\to\infty$. Consequently, by Lebesgue's dominated convergence theorem, we have
	\begin{align*}
		\int_{\mathbb{R}^N}\mathcal{H}_1(u_R)\,\mathrm{d}x\to\int_{\mathbb{R}^N}\mathcal{H}_1(u_0)\,\mathrm{d}x \quad\text{as }R\to\infty.
	\end{align*}
	Employing similar ideas as in Lemma \ref{lem3.7}, we can deduce that $u_R\to u_0$ in $L^{\mathcal{H}_1}(\mathbb{R}^N)$ as $R\to\infty$. Consequently, we obtain $u_R\to u$ in $\mathbf{Y}$ as $R\to\infty$.

	Using the fact that $u_R\in \mathcal{O}_\varepsilon $, we obtain from Lemma \ref{lem4.2} that there exists $t_\varepsilon>0$ such that $t_\varepsilon u_R\in \mathcal{N}_\varepsilon $. Now, arguing similarly as in (b), we have
	\begin{align*}
		c_\varepsilon\leq \max_{t\geq 0}J_\varepsilon (tu)=J_\varepsilon (t_\varepsilon u_R).
	\end{align*}
	Next, we claim that for some $\varepsilon_0>0$, the family $\{t_\varepsilon\}_{0<\varepsilon<\varepsilon_0}$ is bounded. Indeed, if not, assume that $t_\varepsilon\to\infty$ as $\varepsilon\to 0$. Now, using the fact that $t_\varepsilon u_R\in \mathcal{N}_\varepsilon$, we have
	\begin{equation}\label{eq4.18}
		\begin{aligned}
			&\frac{1}{t_\varepsilon^{N-p}}\|u_R\|^p_{W^{1,p}_{V_\varepsilon}}+\|u_R\|^N_{W^{1,N}_{V_\varepsilon}}+\|u_R\|^N_N\\
			& =\int_{\Lambda_\varepsilon}\frac{\mathcal{H}'_2(t_\varepsilon u_R)u_R}{t_\varepsilon^{N-1}}\,\mathrm{d}x+\int_{\Lambda^c_\varepsilon}\frac{\widetilde{\mathcal{H}}'_2(t_\varepsilon u_R)u_R}{t_\varepsilon^{N-1}}\,\mathrm{d}x-\int_{\mathbb{R}^N}\frac{\mathcal{H}'_1(t_\varepsilon u_R)u_R}{t_\varepsilon^{N-1}}\,\mathrm{d}x  \\
			& \quad +\int_{\Lambda_\varepsilon}\frac{f(t_\varepsilon u_R)u_R}{t_\varepsilon^{N-1}}\,\mathrm{d}x+\int_{\Lambda^c_\varepsilon}\frac{\widetilde{f}(t_\varepsilon u_R)u_R}{t_\varepsilon^{N-1}}\,\mathrm{d}x.
		\end{aligned}
	\end{equation}
	Without loss of generality, we can choose $V(0)=V_0$. Note that $u_R$ has compact support and $V(\varepsilon x)\to V_0 $ as $\varepsilon\to 0$. Also observe that  $\Lambda_\varepsilon\to\mathbb{R}^N$ as $\varepsilon\to 0$, see Alves-Ji \cite[Lemma 3.7]{Alves-Ji-2020}. This yields that $|\Lambda^c_\varepsilon|\to0$ as $\varepsilon\to 0$. Consequently, due to Lebesgue's dominated convergence theorem, we have for each $R>0$
	\begin{align*}
		\displaystyle\int_{\mathbb{R}^N} \big(|\nabla u_R|^{t}+V(\varepsilon x)| u_R|^{t}\big)\,\mathrm{d}x\to  \displaystyle\int_{\mathbb{R}^N} \big(|\nabla u_R|^{t}+V_0| u_R|^{t}\big)\,\mathrm{d}x \quad\text{as }\varepsilon\to 0 \text{ and }t\in\{p,N\}.
	\end{align*}
	Using again Lebesgue's dominated convergence theorem leads to
	\begin{align*}
		\displaystyle\int_{\Lambda^c_\varepsilon}\frac{\widetilde{\mathcal{H}}'_2(t_\varepsilon u_R)u_R}{t_\varepsilon^{N-1}}\,\mathrm{d}x\to 0
		\quad\text{and}\quad \displaystyle\int_{\Lambda^c_\varepsilon}\frac{\widetilde{f}(t_\varepsilon u_R)u_R}{t_\varepsilon^{N-1}}\,\mathrm{d}x\to 0\quad\text{as }\varepsilon\to 0.
	\end{align*}
	Define
	\begin{align*}
		\mathcal{T}_\varepsilon=\int_{\Lambda_\varepsilon}\frac{f(t_\varepsilon u_R)u_R}{t_\varepsilon^{N-1}}\,\mathrm{d}x+\int_{\Lambda_\varepsilon}\frac{\mathcal{H}'_2(t_\varepsilon u_R)u_R}{t_\varepsilon^{N-1}}\,\mathrm{d}x-\int_{\mathbb{R}^N}\frac{\mathcal{H}'_1(t_\varepsilon u_R)u_R}{t_\varepsilon^{N-1}}\,\mathrm{d}x .
	\end{align*}
	In view of (f4) and \eqref{eq2.6}, we have
	\begin{align*}
		\mathcal{T}_\varepsilon & \geq \gamma t_\varepsilon^{\mu-N}\int_{\Lambda_\varepsilon} |u_R|^\mu\,\mathrm{d}x+\log(t_\varepsilon)^N \int_{\mathbb{R}^N} |u_R|^N\,\mathrm{d}x-\int_{\Lambda^c_\varepsilon}\frac{\mathcal{H}'_2(t_\varepsilon u_R)u_R}{t_\varepsilon^{N-1}}\,\mathrm{d}x+A,
	\end{align*}
	where $A=\int_{\mathbb{R}^N} |u_R|^N(1+\log |u_R|^N)\,\mathrm{d}x$. From the definition of $\mathcal{H}_2$, we have the estimate
	\begin{align*}
		\frac{\mathcal{H}'_2(t_\varepsilon u_R)u_R}{t_\varepsilon^{N-1}}=|u_R|^N\log |t_\varepsilon u_R|^N-\big(\log((N-1)\delta)^N+N\big)|u_R|^N +\frac{N(N-1)\delta}{t_\varepsilon}|u_R|^{N-1}.
	\end{align*}
	This yields
	\begin{align*}
		\int_{\Lambda^c_\varepsilon}\frac{\mathcal{H}'_2(t_\varepsilon u_R)u_R}{t_\varepsilon^{N-1}}\,\mathrm{d}x\leq \int_{\Lambda^c_\varepsilon}|u_R|^N\log |t_\varepsilon u_R|^N\,\mathrm{d}x+\frac{N(N-1)\delta}{t_\varepsilon}\int_{\mathbb{R}^N}|u_R|^{N-1}\,\mathrm{d}x+B,
	\end{align*}
	where $B=-\big(\log((N-1)\delta)^N+N\big)\int_{\mathbb{R}^N}|u_R|^N\,\mathrm{d}x$. Gathering all these information and the fact that $t_\varepsilon\to\infty$ as $\varepsilon\to 0$, we obtain by setting $C=A-B$ that
	\begin{align*}
		\mathcal{T}_\varepsilon \geq \gamma t_\varepsilon^{\mu-N}\int_{\Lambda_\varepsilon} |u_R|^\mu\,\mathrm{d}x+\log(t_\varepsilon)^N \int_{\Lambda_\varepsilon} |u_R|^N\,\mathrm{d}x-\int_{\Lambda^c_\varepsilon}|u_R|^N\log |u_R|^N\,\mathrm{d}x+C+o_\varepsilon(1)\to\infty \quad\text{as }\varepsilon\to 0.
	\end{align*}
	It follows that $\mathcal{T}_\varepsilon \to\infty$ as $\varepsilon\to 0$. Letting $\varepsilon\to 0$ in \eqref{eq4.18} and using the above estimates, we get a contradiction. Hence, we conclude that $\{t_\varepsilon\}_{0<\varepsilon<\varepsilon_0}$ is bounded. Choose $t_R>0$ such that $ J_0(t_R u_R)=\max_{t\geq 0} J_0(tu_R)$. It follows that $t_R u_R\in\mathcal{N}_0$.

	Now, we claim that $t_R$ is bounded for $R>0$ large enough. Indeed, if not, let $t_R\to\infty$ as $R\to\infty$. In virtue of $\langle{J'_0(t_R u_R), t_R u_R}\rangle=0$, we have
	\begin{equation}\label{eq4.19}
		\begin{aligned}
			&\frac{1}{t_R^{N-p}}\|u_R\|^p_{W^{1,p}_{V_0}}+\|u_R\|^N_{W^{1,N}_{V_0}}+\|u_R\|^N_N \\
			& =\int_{\mathbb{R}^N}\frac{f(t_R u_R)u_R}{t_R^{N-1}}\mathrm{d}x+\int_{\mathbb{R}^N}\frac{\mathcal{H}'_2(t_R u_R)u_R}{t_R^{N-1}}\mathrm{d}x-\int_{\mathbb{R}^N}\frac{\mathcal{H}'_1(t_R u_R)u_R}{t_R^{N-1}}\mathrm{d}x.
		\end{aligned}
	\end{equation}
	Due to \eqref{eq4.17}, one sees that
	\begin{align*}
		t_R^{p-N} \|u_R\|^p_{W^{1,p}_{V_0}}+\|u_R\|^N_{W^{1,N}_{V_0}}+\|u_R\|^N_N\to \|u_0\|^N_{W^{1,N}_{V_0}}+\|u_0\|^N_N \quad\text{as }R\to\infty.
	\end{align*}
	Moreover, using (f4), \eqref{eq2.6} and \eqref{eq4.17}, one has
	\begin{align*}
		&\int_{\mathbb{R}^N}\frac{f(t_R u_R)u_R}{t_R^{N-1}}\mathrm{d}x+\int_{\mathbb{R}^N}\frac{\mathcal{H}'_2(t_R u_R)u_R}{t_R^{N-1}}\mathrm{d}x-\int_{\mathbb{R}^N}\frac{\mathcal{H}'_1(t_R u_R)u_R}{t_R^{N-1}}\mathrm{d}x\\
		& \geq \gamma t_R^{\mu-N}\int_{\mathbb{R}^N} |u_R|^\mu\mathrm{d}x+\big(\log(t_R)^N+1\big)\int_{\mathbb{R}^N} |u_R|^N\mathrm{d}x  \\
		& \qquad+\int_{\mathbb{R}^N}|u_R|^N\log |u_R|^N\,\mathrm{d}x \to\infty \quad\text{as }R\to\infty.
	\end{align*}
	It follows that
	\begin{align*}
		\int_{\mathbb{R}^N}\frac{f(t_R u_R)u_R}{t_R^{N-1}}\,\mathrm{d}x+\int_{\mathbb{R}^N}\frac{\mathcal{H}'_2(t_R u_R)u_R}{t_R^{N-1}}\,\mathrm{d}x-\int_{\mathbb{R}^N}\frac{\mathcal{H}'_1(t_R u_R)u_R}{t_R^{N-1}}\,\mathrm{d}x \to\infty \quad\text{as }R\to\infty.
	\end{align*}
	We get a contradiction by using all of the above information and letting $R\to\infty$ in \eqref{eq4.19}. This shows the claim. Hence, up to a subsequence still denoted by the same symbol, $t_R\to t_1~(\geq 0)$ as $R\to\infty$. Employing the same ideas as in Proposition \ref{prop4.1}, one can prove that $\mathcal{N}_0$ is bounded away from the origin. This together with $t_Ru_R\in \mathcal{N}_0 $ implies that there exists some $\beta>0$ such that $\|t_Ru_R\|_{\mathbf{Y}}\geq\beta>0$. Now, sending $R\to\infty$ and using $u_R\to u_0$ in $\mathbf{Y}$ as $R\to\infty$, we can see that $t_1\|u_0\|_{\mathbf{Y}}\geq\beta>0$. From this, we conclude that $t_1\neq 0$. A direct computation implies that
	\begin{align*}
		J_\varepsilon(t_\varepsilon u_R)- J_0(t_\varepsilon u_R) & = \frac{t^p_\varepsilon}{p}\int_{\mathbb{R}^N}\big(V(\varepsilon x)-V_0\big)|u_R|^p\,\mathrm{d}x+\frac{t^N_\varepsilon}{N}\int_{\mathbb{R}^N}\big(V(\varepsilon x)-V_0\big)|u_R|^N\,\mathrm{d}x  \\
		& \quad +\int_{\Lambda^c_\varepsilon}\big(\mathcal{H}_2(t_\varepsilon u_R)-\widetilde{\mathcal{H}}_2(t_\varepsilon u_R)\big)\,\mathrm{d}x+\int_{\Lambda^c_\varepsilon}\big(F(t_\varepsilon u_R)-\widetilde{F}(t_\varepsilon u_R)\big)\,\mathrm{d}x.
	\end{align*}
	Using the fact that $\{t_\varepsilon\}_{0<\varepsilon<\varepsilon_0}$ is bounded, $u_R$ has compact support, $V(\varepsilon x)\to V_0 $ as $\varepsilon\to 0$ and $0\leq u_R\leq u_0$, we obtain from Lebesgue's dominated convergence theorem that
	\begin{align*}
		\frac{t^{\wp}_\varepsilon}{\wp}\int_{\mathbb{R}^N}\big(V(\varepsilon x)-V_0\big)|u_R|^{\wp}\,\mathrm{d}x\to 0 \quad\text{as }\varepsilon\to 0~\text{ for }~\wp\in\{p,N\}.
	\end{align*}
	In virtue of $|\Lambda^c_\varepsilon|\to0$ as $\varepsilon\to 0$, by using ($\mathcal{H}_2$)(b), \eqref{eq3.2} and Lemma \ref{lem2.11}, we obtain from Lebesgue's dominated convergence theorem that
	\begin{align*}
		\int_{\Lambda^c_\varepsilon}\big(\mathcal{H}_2(t_\varepsilon u_R)-\widetilde{\mathcal{H}}_2(t_\varepsilon u_R)\big)\,\mathrm{d}x\to 0 \quad\text{and}\quad\int_{\Lambda^c_\varepsilon}\big(F(t_\varepsilon u_R)-\widetilde{F}(t_\varepsilon u_R)\big)\,\mathrm{d}x\to0\quad\text{as }\varepsilon\to 0 .
	\end{align*}
	Gathering all of the above information, we deduce that
	\begin{align*}
		J_\varepsilon(t_\varepsilon u_R)- J_0(t_\varepsilon u_R)=o_\varepsilon(1)\quad\text{as }\varepsilon\to 0.
	\end{align*}
	Consequently, we have
	\begin{equation}
		\label{eq4.20}
		\limsup_{\varepsilon\to0}c_\varepsilon\leq \limsup_{\varepsilon\to0} J_\varepsilon(t_\varepsilon u_R)\leq J_0(t_R u_R).
	\end{equation}
	Notice that $\mathcal{H}_1(t_Ru_R)\to \mathcal{H}_1(t_1u_0)~\text{and}~\mathcal{H}'_1(t_Ru_R)t_Ru_R\to \mathcal{H}'_1(t_1u_0)t_1u_0$ a.e.\,in $\mathbb{R}^N$ as $R\to\infty $. Now, using the fact that $\{t_R\}_{R>0}$ is bounded for large $R$, $0\leq u_R\leq u_0$ and $\mathcal{H}_1$ is nondecreasing for $t\geq 0$, we infer that there exists $K>0$ such that $t_R\leq K$ and satisfying $\mathcal{H}_1(t_Ru_R)\leq \mathcal{H}_1(Ku_0)\in L^1(\mathbb{R}^N)$, thanks to $u_0\in\mathcal{N}_0$. Further, in light of Remark \ref{rem2.99} and $u_0\in\mathcal{N}_0$, we also have $\mathcal{H}'_1(t_Ru_R)t_Ru_R\leq N \mathcal{H}_1(t_Ru_R)\leq N \mathcal{H}_1(Ku_0)\in L^1(\mathbb{R}^N)$. Consequently, due to Lebesgue's dominated convergence theorem, we obtain
	\begin{align*}
		\int_{\mathbb{R}^N} \mathcal{H}_1(t_Ru_R)\,\mathrm{d}x\to \int_{\mathbb{R}^N}\mathcal{H}_1(t_1u_0)\,\mathrm{d}x
		\quad\text{and}\quad
		\int_{\mathbb{R}^N}\mathcal{H}'_1(t_Ru_R)t_Ru_R\,\mathrm{d}x\to \int_{\mathbb{R}^N}\mathcal{H}'_1(t_1u_0)t_1u_0\,\mathrm{d}x\quad\text{as }R\to \infty.
	\end{align*}
	By using ($\mathcal{H}_2$)(b), \eqref{eq3.2} and Lemma \ref{lem2.11}, we obtain from Lebesgue's dominated convergence theorem that
	\begin{align*}
		\int_{\mathbb{R}^N} \mathcal{H}_2(t_Ru_R)\,\mathrm{d}x\to \int_{\mathbb{R}^N}\mathcal{H}_2(t_1u_0)\,\mathrm{d}x,
		\quad
		\int_{\mathbb{R}^N}\mathcal{H}'_2(t_Ru_R)t_Ru_R\,\mathrm{d}x\to \int_{\mathbb{R}^N}\mathcal{H}'_2(t_1u_0)t_1u_0\,\mathrm{d}x,
	\end{align*}
	\begin{align*}
		\int_{\mathbb{R}^N} F(t_Ru_R)\,\mathrm{d}x\to \int_{\mathbb{R}^N}F(t_1u_0)\,\mathrm{d}x
		\quad\text{and}\quad
		\int_{\mathbb{R}^N}f(t_Ru_R)t_Ru_R\,\mathrm{d}x\to \int_{\mathbb{R}^N}f(t_1u_0)t_1u_0\,\mathrm{d}x\quad\text{as }R\to \infty.
	\end{align*}
	Combining the above information, we deduce that
	$J_0(t_R u_R)\to J_0(t_1 u_0)$ and $0=\langle{J_0'(t_R u_R),t_R u_R}\rangle\to \langle{J_0'(t_1 u_0),t_1 u_0}\rangle$ as $R\to \infty$. Next, we claim that $t_1=1$. Indeed, if not, then either $t_1>1$ or  $t_1<1$. To prove the claim, we first observe that $\langle{J_0'(t_1 u_0),t_1 u_0}\rangle=0$ and hence, we have
	\begin{equation}\label{eq4.21}
		\begin{aligned}
			&\frac{1}{t_1^{N-p}}\|u_0\|^p_{W^{1,p}_{V_0}}+\|u_0\|^N_{W^{1,N}_{V_0}}+\|u_0\|^N_N+\int_{\mathbb{R}^N}\frac{\mathcal{H}'_1(t_1 u_0)u_0}{t_1^{N-1}}\,\mathrm{d}x\\
			& =\int_{\mathbb{R}^N}\frac{\mathcal{H}'_2(t_1u_0)u_0}{t_1^{N-1}}\,\mathrm{d}x+\int_{\mathbb{R}^N}\frac{f(t_1 u_0)u_0}{t_1^{N-1}}\,\mathrm{d}x.
		\end{aligned}
	\end{equation}
	On the other hand, since $u_0\in \mathcal{N}_0$, we have
	\begin{align}\label{eq4.22}
		\|u_0\|^p_{W^{1,p}_{V_0}}+\|u_0\|^N_{W^{1,N}_{V_0}}+\|u_0\|^N_N+\int_{\mathbb{R}^N}\mathcal{H}'_1( u_0)u_0\,\mathrm{d}x  =\int_{\mathbb{R}^N}\mathcal{H}'_2(u_0)u_0\,\mathrm{d}x+\int_{\mathbb{R}^N}f( u_0)u_0\,\mathrm{d}x.
	\end{align}
	Subtracting \eqref{eq4.22} from \eqref{eq4.21}, we get
	\begin{equation}\label{eq4.23}
		\begin{aligned}
			&\bigg(\frac{1}{t_1^{N-p}}-1\bigg)\|u_0\|^p_{W^{1,p}_{V_0}}+\int_{\mathbb{R}^N}\bigg(\frac{\mathcal{H}'_1(t_1 u_0)}{t_1^{N-1}}-\mathcal{H}'_1( u_0)\bigg)u_0\,\mathrm{d}x\\
			&=\int_{\mathbb{R}^N}\bigg(\frac{\mathcal{H}'_2(t_1 u_0)}{t_1^{N-1}}-\mathcal{H}'_2( u_0)\bigg)u_0\,\mathrm{d}x+\int_{\mathbb{R}^N}\bigg(\frac{f(t_1 u_0)}{t_1^{N-1}}-f( u_0)\bigg)u_0\,\mathrm{d}x.
		\end{aligned}
	\end{equation}
	Now, we consider the case when $t_1>1$. Recall that by Lemma \ref{lem4.2}, for any $0<\mathbf{t}_1<\mathbf{t}_2$, we have
	\begin{align*} \int_{\mathbb{R}^N}\bigg(\frac{\mathcal{H}'_1(\mathbf{t}_2u)u}{\mathbf{t}_2^{N-1}}-\frac{\mathcal{H}'_1(\mathbf{t}_1u)u}{\mathbf{t}_1^{N-1}}\bigg)\,\mathrm{d}x<0. \end{align*}
	Replacing $\mathbf{t}_1=1$, $\mathbf{t}_2=t_1$ and $u=u_0$, respectively in the above inequality, we infer that the left-hand side of \eqref{eq4.23} is negative. On the other hand, by using ($\mathcal{H}_2$)(c) and (f3), one sees that the right-hand side of \eqref{eq4.23} is nonnegative. This shows at once that $t_1>1$ is impossible. Similarly, when $t_1<1$, we can easily arrive at a contradiction. It follows that $t_1=1$ and hence, we obtain $J_0(t_R u_R)\to J_0( u_0)$ as $R\to \infty$. Consequently, we deduce from \eqref{eq4.20} that
	\begin{align*}
		\limsup_{\varepsilon\to0}c_\varepsilon\leq\lim_{R\to\infty} J_0(t_R u_R)=J_0( u_0)=c_0.
	\end{align*}
	The proof is now complete.
\end{proof}

Finally, we shall be able to prove the following result.

\begin{proposition}\label{prop4.6}
	The problem \eqref{main problem@@} has a positive ground state solution.
\end{proposition}

\begin{proof}
	Let $u_\varepsilon$ be a solution to \eqref{main problem@@} as in Theorem \ref{thm3.8} and hence, $J_\varepsilon(u_\varepsilon)=c_\varepsilon$ and $J'_\varepsilon(u_\varepsilon)=0$. Set $u_\varepsilon=u^+_\varepsilon-u^-_\varepsilon$. It follows that $J_\varepsilon(u_\varepsilon)=J_\varepsilon(u^+_\varepsilon)+J_\varepsilon(u^-_\varepsilon)$. Moreover, we claim that either $u^+_\varepsilon=0$ or $u^-_\varepsilon=0$. Indeed, if not, then $0=\langle{J'_\varepsilon(u_\varepsilon),u^+_\varepsilon}\rangle=\langle{J'_\varepsilon(u^+_\varepsilon),u^+_\varepsilon}\rangle$. It follows that $u^+_\varepsilon\in\mathcal{N}_\varepsilon$. Similarly, we can also prove that $u^-_\varepsilon\in\mathcal{N}_\varepsilon$. Because of Lemma \ref{lem4.5}(b), one sees that $c_\varepsilon=J_\varepsilon(u_\varepsilon)=J_\varepsilon(u^+_\varepsilon)+J_\varepsilon(u^-_\varepsilon)\geq 2c_\varepsilon$, which is a contradiction. Using that $G'_2$ and $g$ are odd functions, we can assume that $u_\varepsilon$ is a nonnegative solution of \eqref{main problem@@}. Employing a slight variant of a Moser iteration argument explored in Lemma \ref{lem4.10}, it follows that $u_\varepsilon\in L^\infty(\mathbb{R}^N)\cap C^{0,\alpha}_{\operatorname{loc}}(\mathbb{R}^N)$ for some $0 < \alpha < 1$, and therefore, from Harnack's inequality (see Trudinger \cite{Trudinger-1967}), we obtain $u_\varepsilon(x)>0$ for a.e.\,$x\in \mathbb{R}^N $. This finishes the proof.
\end{proof}

Hereafter, unless otherwise noted, the solution of \eqref{main problem@@} provided in the previous proposition is denoted by the notation $u_\varepsilon$.
The following corollary is a consequence of the Lions' compactness result that appears in Lemma \ref{lem8.99}.

\begin{corollary}\label{cor4.7}
	Let $\{u_n\}_{n\in\mathbb{N}}\subset \mathbf{Y}_\varepsilon$ be a bounded sequence in $\mathbf{Y}_\varepsilon$ verifying $\displaystyle\limsup_{n\to\infty} \|u_n\|^{N'}_{W^{1,N}}<\frac{\alpha_N}{\alpha_0}$. If there exists $R>0$ such that
	\begin{equation}
		\label{eq4.24}
		\liminf_{n\to\infty} \sup_{y\in \mathbb{R}^N} \int_{B_R(y)}|u_n|^N\,\mathrm{d}x=0,
	\end{equation}
	then the following holds:
	\begin{enumerate}
		\item[\textnormal{(a)}]
			$\displaystyle\int_{\mathbb{R}^N} g(\varepsilon x,u_n)u_n\,\mathrm{d}x\to 0$\quad and \quad $\displaystyle\int_{\mathbb{R}^N} \mathcal{G}(\varepsilon x,u_n)\,\mathrm{d}x\to 0$ as $n\to\infty$;
		\item[\textnormal{(b)}]
			$\displaystyle\int_{\mathbb{R}^N} G_2'(\varepsilon x,u_n)u_n\,\mathrm{d}x\to 0$\quad and \quad $\displaystyle\int_{\mathbb{R}^N} G_2(\varepsilon x,u_n)\,\mathrm{d}x\to 0$ as $n\to\infty$.
	\end{enumerate}
\end{corollary}

\begin{proof}
	Given $\{u_n\}_{n\in\mathbb{N}}\subset \mathbf{Y}_\varepsilon$ to be bounded in $\mathbf{Y}_\varepsilon$ and $\displaystyle\limsup_{n\to\infty} \|u_n\|^{N'}_{W^{1,N}}<\frac{\alpha_N}{\alpha_0}$. It follows from \eqref{eq4.24} and Lemma \ref{lem8.99} that $u_n\to 0$ in $L^\upsilon(\mathbb{R}^N)$ as $n\to\infty$ for any $\upsilon\in (N,+\infty)$.

	Now, by the hypothesis, there exist $m>0$ and $n_0\in\mathbb{N}$ large enough such that $\|u_n\|^{N'}_{W^{1,N}}<m<\frac{\alpha_N}{\alpha_0}$ for all $n\geq n_0$. Take $r>1$ with $r'=\frac{r}{r-1}>1$ and satisfying $\frac{1}{r}+\frac{1}{r'}=1$ . Let $r'$ close to 1 and $\alpha>\alpha_0$ close to $\alpha_0$ such that we still have $r'\alpha\|u_n\|^{N'}_{W^{1,N}}<m<\alpha_N$ for all $n\geq n_0$ and $\widetilde{u}_n = \frac{u_n}{\|u_n\|_{W^{1,N}}}$. It follows from ($\mathcal{B}$)(b), \eqref{eq3.2}, H\"older's inequality, Corollary \ref{cor2.9} and Lemma \ref{lem8.98}  that for $n$ large enough, we have
	\begin{align*}
		\Bigg|\int_{\mathbb{R}^N} g(\varepsilon x,u_n)u_n\,\mathrm{d}x\Bigg|\leq \tau S_N^{-N}\|u_n\|^N_{\mathbf{Y}_\varepsilon}+C\|u_n\|^\vartheta_{\vartheta r},
	\end{align*}
	where
	\begin{align*}
		C=\widetilde{\kappa}_\tau \bigg[\sup_{n\geq n_0}\int_{\mathbb{R}^N}\Phi\big(r'\alpha\|u_n\|^{N'}_{W^{1,N}}|\widetilde{u}_n|^{N'}\big)\,\mathrm{d}x\bigg]^{\frac{1}{r'}}<+\infty,
	\end{align*}
	thanks to Lemma \ref{lem2.11}. Consequently, due to the boundedness of the sequence $\{u_n\}_{n\in\mathbb{N}}$ in $\mathbf{Y}_\varepsilon$ and $u_n\to 0$ in $L^{\vartheta r}(\mathbb{R}^N)$ as $n\to\infty$, there exists a constant $C_1>0$ such that
	\begin{align*}
		\limsup_{n\to\infty}\Bigg|\int_{\mathbb{R}^N} g(\varepsilon x,u_n)u_n\,\mathrm{d}x\Bigg|\leq \tau S_N^{-N}C_1.
	\end{align*}
	Letting $\tau\to 0$, we obtain from the above inequality that
	\begin{align*}
		\int_{\mathbb{R}^N} g(\varepsilon x,u_n)u_n\,\mathrm{d}x\to 0\quad\text{as }n\to\infty.
	\end{align*}
	Further, by using \eqref{eq3.3} and employing similar arguments explored as above, one has
	\begin{align*}
		\int_{\mathbb{R}^N} \mathcal{G}(\varepsilon x,u_n)\,\mathrm{d}x\to 0\quad\text{as }n\to\infty.
	\end{align*}
	On the other hand, by using  ($\mathcal{H}$)(b), ($\mathcal{A}$)(b) and the fact that $u_n\to 0$ in $L^\upsilon(\mathbb{R}^N)$ as $n\to\infty$ for any $\upsilon\in (N,+\infty)$, we can easily deduce that (b) holds. This completes the proof.
\end{proof}

Now, we are ready to prove the following important compactness result.

\begin{lemma}\label{lem4.8}
	Let $\varepsilon_n\to 0$ as $n\to\infty$ and $\{u_n\}_{n\in\mathbb{N}}=\{u_{\varepsilon_n}\}_{n\in\mathbb{N}}\subset\mathbf{Y}_{\varepsilon_n}$ be a nonnegative sequence such that $J_{\varepsilon_n}(u_n)=c_{\varepsilon_n},~J'_{\varepsilon_n}(u_n)=0$ and satisfying $\displaystyle\limsup_{n\to\infty} \|u_n\|^{N'}_{W^{1,N}}<\frac{\alpha_N}{\alpha_0}$. Then, there exists a sequence $\{\widetilde{y}_n\}_{n\in\mathbb{N}}\subset\mathbb{R}^N$ such that the translated sequence
	\begin{align*}
		w_n(x)=\widetilde{u}_n(x)=u_n(x+\widetilde{y}_n)
	\end{align*}
	has a convergent subsequence in $\mathbf{Y}$. Furthermore, up to a subsequence, $y_n=\varepsilon_n\widetilde{y}_n\to y_0$ as $n\to\infty$ for some $y_0\in\Lambda$ and $V(y_0)=V_0$.
\end{lemma}

\begin{proof}
	Following the ideas in the proof of Lemma \ref{lem3.4}, one sees that $\{u_n\}_{n\in\mathbb{N}}$ is bounded in $\mathbf{Y}_{\varepsilon_n}$. Consequently, due to (V1), we conclude that $\{u_n\}_{n\in\mathbb{N}}$ is bounded in $\mathbf{Y}$. Now, we claim that there exist $R>0$, $\alpha>0$ and a sequence $\{\widetilde{y}_n\}_{n\in\mathbb{N}}\subset\mathbb{R}^N$ such that there holds
	\begin{equation} \label{eq4.25}
		\liminf_{n\to\infty} \int_{B_R(\widetilde{y}_n)}|u_n|^N\,\mathrm{d}x\geq \alpha>0.
	\end{equation}
	Indeed, if \eqref{eq4.25} does not hold, it means that \eqref{eq4.24} holds. In virtue of Lemma \ref{lem8.99}, we obtain that $u_n\to 0$ in $L^\upsilon(\mathbb{R}^N)$ as $n\to\infty$ for any $\upsilon\in (N,+\infty)$. In addition, one can notice that the results of Corollary \ref{cor4.7} also hold whenever $\varepsilon$ is replaced by $\varepsilon_n$. By the hypothesis, we have $\langle{J'_{\varepsilon_n}(u_n),u_n}\rangle=0$. This fact together with Corollary \ref{cor4.7} implies that
	\begin{align*}
		\|u_n\|_{W^{1,p}_{V_{\varepsilon_n}}}^p+ \|u_n\|_{W^{1,N}_{V_{\varepsilon_n}}}^N+\|u_n\|^N_N+\int_{\mathbb{R}^N}\mathcal{H}'_1(u_n)u_n\,\mathrm{d}x=o_n(1)\quad\text{as }n\to\infty.
	\end{align*}
	Now, employing similar arguments explored in Lemma \ref{lem4.2}, one has
	\begin{align*}
		0\leq \|u_n\|_{W^{1,p}_{V_{\varepsilon_n}}}^p+ \|u_n\|_{W^{1,N}_{V_{\varepsilon_n}}}^N+\|u_n\|^N_N+\min\big\{\|u_n\|_{\mathcal{H}_1}^l,\|u_n\|_{\mathcal{H}_1}^N\big\}\leq o_n(1)\quad\text{as }n\to\infty.
	\end{align*}
	It follows at once that $u_n\to 0$ in $\mathbf{Y}_{\varepsilon_n}$ as $n\to\infty$. Consequently, due to \eqref{eq2.333} and Corollary \ref{cor4.7}, we infer that $J_{\varepsilon_n}(u_n)=c_{\varepsilon_n}\to 0$ as $n\to\infty$, which is a contradiction because of Lemma \ref{lem4.5}(a). This shows that \eqref{eq4.25} holds. Set $w_n(x)=\widetilde{u}_n(x)=u_n(x+\widetilde{y}_n)$, then using the fact that $\|\cdot\|_{\mathbf{Y}}$ is invariant under translation, we deduce that $\{w_n\}_{n\in\mathbb{N}}$ is bounded in $\mathbf{Y}$. Thus, up to a subsequence not relabeled, we may assume that there exists $w\in \mathbf{Y}$ such that
	\begin{equation}\label{eq4.26}
		w_n\rightharpoonup w\quad\text{in }\mathbf{Y},
		\quad w_n\to w \quad\text{in }L^s(B_R)
		\quad\text{and }\quad
		w_n\to w\quad\text{a.e.\,in } \mathbb{R}^N\quad\text{as }n\to\infty,
	\end{equation}
	and also there exists $g\in L^s(B_R)$ such that $|w_n|\leq g$ a.e.\,in $\mathbb{R}^N$ for $s\in[1,+\infty)$ and for all $R>0$, thanks to Lemma \ref{lem8.98}. In view of \eqref{eq4.26}, we obtain from \eqref{eq4.25} that
	\begin{align*}
		\int_{B_R(0)}|w|^N\,\mathrm{d}x\geq \alpha>0.
	\end{align*}
	This infers that $w\neq 0$. Define $y_n=\varepsilon_n\widetilde{y}_n$. Now, our aim is to show that $\{y_n\}_{n\in\mathbb{N}}$ is a bounded sequence in $\mathbb{R}^N$. For this, we prove the following claim.\\
	\textbf{Claim I: } There holds $\lim_{n\to\infty} \operatorname{dist}(y_n,\overline{\Lambda})=0$.\\
	If the claim does not hold, then there exists $\delta>0$ and a subsequence of $\{y_n\}_{n\in\mathbb{N}}$, not relabeled, such that
	\begin{align*}
		\operatorname{dist}(y_n,\overline{\Lambda})\geq \delta\quad\text{for all }n\in\mathbb{N}.
	\end{align*}
	It follows that there exists some $r>0$ such that $B_r(y_n)\subset\Lambda^c$ for all $n\in\mathbb{N}$. Next, let $\phi\in C^\infty_{c}(\mathbb{R}^N)$ be such that $0\leq \phi\leq 1$, $\phi\equiv 1$ in $B_1$, and $\phi\equiv 0$ in $B^c_{2}$. For each $j\in\mathbb{N}$, define $\phi_j(\cdot)=\phi(\frac{\cdot}{j})$. Further, set $\psi_j(\cdot)=\psi(\frac{\cdot}{j})$ and $\psi_j(x)=\phi_j(x)w(x)$. It follows that $0\leq \psi_j\leq w$, $\psi_j\equiv w$ in $B_j$, and $\psi_j\equiv 0$ in $B^c_{2j}$. Observe that $\operatorname{supp}(\psi_j)\subset B_{2j}$. By using similar arguments explored in Lemma \ref{lem4.5}, we can obtain
	\begin{equation}\label{eq4.27}
		\psi_j\to w\quad\text{in }  \mathbf{Y} \text{ as } j\to\infty.
	\end{equation}
	Now, by fixing $j>0$ and using $\psi_j$ as a test function along with the change of variable $z\mapsto x+\widetilde{y}_n$ and the invariance by translation, we can see that
	\begin{equation}\label{eq4.28}
		\begin{aligned}
		 	& \sum_{t\in\{p,N\}}\bigg[\int_{\mathbb{R}^N}\big(|\nabla w_n|^{t-2}\nabla w_n\cdot\nabla\psi_j+V(\varepsilon_n x+y_n)w_n^{t-1}\psi_j\big)\,\mathrm{d}x\bigg]\\
		 	&+ \int_{\mathbb{R}^N}w_n^{N-1}\psi_j\,\mathrm{d}x+\int_{\mathbb{R}^N} \mathcal{H}'_1(w_n)\psi_j\,\mathrm{d}x  \\
		 	& = \int_{\mathbb{R}^N} G_2'(\varepsilon_n x+y_n,w_n)\psi_j\,\mathrm{d}x +\int_{\mathbb{R}^N} g(\varepsilon_n x+y_n,w_n)\psi_j\,\mathrm{d}x.
		\end{aligned}
	\end{equation}
	From the definitions of $G_2'$ together with ($\mathcal{A}$)(c), we can see that
	\begin{align*}
		\int_{\mathbb{R}^N} G_2'(\varepsilon_n x+y_n,w_n)\psi_j\,\mathrm{d}x & = \int_{B_{\frac{r}{\varepsilon_n}}} G_2'(\varepsilon_n x+y_n,w_n)\psi_j\,\mathrm{d}x +\int_{B^c_{\frac{r}{\varepsilon_n}}} G_2'(\varepsilon_n x+y_n,w_n)\psi_j\,\mathrm{d}x\\
		& \leq \ell \int_{B_{\frac{r}{\varepsilon_n}}} w_n^{N-1}\psi_j\,\mathrm{d}x +\int_{B^c_{\frac{r}{\varepsilon_n}}} \mathcal{H}_2'(w_n)\psi_j\,\mathrm{d}x\\
		&\leq \ell \int_{\mathbb{R}^N} w_n^{N-1}\psi_j\,\mathrm{d}x +\int_{B^c_{\frac{r}{\varepsilon_n}}} \mathcal{H}_2'(w_n)\psi_j\,\mathrm{d}x.
	\end{align*}
	Similarly, by using the definition of $g$ and ($\mathcal{B}$)(d), one has
	\begin{align*}
		\int_{\mathbb{R}^N} g(\varepsilon_n x+y_n,w_n)\psi_j\,\mathrm{d}x
		 & \leq \ell' \int_{\mathbb{R}^N} w_n^{N-1}\psi_j\,\mathrm{d}x +\int_{B^c_{\frac{r}{\varepsilon_n}}} f(w_n)\psi_j\,\mathrm{d}x.
	\end{align*}
    Observe that $\limsup_{n\to\infty} \|w_n\|^{N'}_{W^{1,N}}<\frac{\alpha_N}{\alpha_0}$ holds. In virtue of $\big|B^c_{\frac{r}{\varepsilon_n}}(0)\big|\to 0$ as $n\to\infty$, ($\mathcal{H}_2$)(b), (f1) and Lemma \ref{lem2.11}, we obtain from Lebesgue's dominated convergence theorem or Vitali's theorem that
	\begin{align*}
		\int_{B^c_{\frac{r}{\varepsilon_n}}} \mathcal{H}_2'(w_n)\psi_j\,\mathrm{d}x\to 0\quad\text{and}\quad\int_{B^c_{\frac{r}{\varepsilon_n}}} f(w_n)\psi_j\,\mathrm{d}x\to 0\quad\text{as }n\to\infty.
	\end{align*}
	Using that $\psi_j$ has compact support, \eqref{eq4.26} and ($\mathcal{H}_1$)(c), due to Lebesgue's dominated convergence theorem, we have
	\begin{equation}\label{eq4.29}
		\int_{\mathbb{R}^N} \mathcal{H}'_1(w_n)\psi_j\,\mathrm{d}x\to \int_{\mathbb{R}^N} \mathcal{H}'_1(w)\psi_j\,\mathrm{d}x\quad\text{as }n\to\infty.
	\end{equation}
	Note that $\{w_n\}_{n\in\mathbb{N}}$ is bounded in $\mathbf{Y}$. Now, arguing similarly as in Lemma \ref{lem3.6}, we can easily prove that $\nabla w_n\to \nabla w$ a.e.\,in $\mathbb{R}^N$ as $n\to\infty$ and for any $t\in\{p,N\}$ there hold
	\begin{equation}\label{eq4.30}
		\begin{aligned}
			\int_{\mathbb{R}^N} |\nabla w_n|^{t-2}\nabla w_n\cdot\nabla\psi_j&\to  \int_{\mathbb{R}^N} |\nabla w|^{t-2}\nabla w\cdot\nabla\psi_j\,\mathrm{d}x\quad\text{as }n\to\infty,\\
			\int_{\mathbb{R}^N} w_n^{t-1}\psi_j\,\mathrm{d}x&\to  \int_{\mathbb{R}^N} w^{t-1}\psi_j\,\mathrm{d}x\quad\text{as }n\to\infty.
		\end{aligned}
	\end{equation}
	Combining all the above information with (V1) and $V_0+1\geq 2(\ell+\ell')$, we obtain by sending $n\to\infty$ in \eqref{eq4.28} that
	\begin{equation}\label{eq4.31}
		\begin{aligned}
			&\int_{\mathbb{R}^N}\big(|\nabla w|^{p-2}\nabla w\cdot\nabla\psi_j+V_0 w^{p-1}\psi_j\big) \mathrm{d}x+\int_{\mathbb{R}^N}\big(|\nabla w|^{N-2}\nabla w\cdot\nabla\psi_j+\eta w^{N-1}\psi_j\big)\mathrm{d}x\\
			&+\int_{\mathbb{R}^N} \mathcal{H}'_1(w)\psi_j\mathrm{d}x\leq 0,
		\end{aligned}
	\end{equation}
	where $\eta=\frac{V_0+1}{2}$. It follows immediately from \eqref{eq4.27} that $\psi_j\rightharpoonup w $ in $L^t(\mathbb{R}^N)$ and $\nabla\psi_j\rightharpoonup \nabla w $ in $[L^t(\mathbb{R}^N)]^N$ as $j\to\infty$ for $t\in\{p,N\}$. Moreover, one has $|\nabla w|^{t-2}\nabla w\in [L^{\frac{t}{t-1}}(\mathbb{R}^N)]^N$ and $w^{t-1}\in L^{\frac{t}{t-1}}(\mathbb{R}^N)$ for $t\in\{p,N\}$. This shows that
	\begin{align*}
		\int_{\mathbb{R}^N} |\nabla w|^{t-2}\nabla w\cdot\nabla\psi_j\to  \int_{\mathbb{R}^N} |\nabla w|^{t}\,\mathrm{d}x\quad\text{and}\quad\int_{\mathbb{R}^N} w^{t-1}\psi_j\,\mathrm{d}x\to  \int_{\mathbb{R}^N} w^{t}\,\mathrm{d}x\quad\text{as }j\to\infty.
	\end{align*}
	Due to Lebesgue's dominated convergence theorem, we also have
	\begin{align*}
		\int_{\mathbb{R}^N} \mathcal{H}'_1(w)\psi_j\,\mathrm{d}x\to \int_{\mathbb{R}^N} \mathcal{H}'_1(w)w\,\mathrm{d}x\quad\text{as }j\to\infty.
	\end{align*}
	Letting $j\to\infty$ in \eqref{eq4.31} and gathering all these information, we get
	\begin{align*}
		0\leq\int_{\mathbb{R}^N}\big(|\nabla w|^{p}+V_0 w^{p}\big)\,\mathrm{d}x+\int_{\mathbb{R}^N}\big(|\nabla w|^{N}+\eta w^{N}\big)\,\mathrm{d}x+\int_{\mathbb{R}^N} \mathcal{H}'_1(w)w\,\mathrm{d}x\leq 0.
	\end{align*}
	Following similar ideas explored in Lemma \ref{lem4.2}, we deduce from the above inequality that
	\begin{align*}
		0\leq\|w\|_{W^{1,p}_{V_0}}^p+ \|w\|_{W^{1,N}_{\eta}}^N+\min\big\{\|w\|_{\mathcal{H}_1}^l,\|w\|_{\mathcal{H}_1}^N\big\} \leq 0.
	\end{align*}
	It follows that $w=0$ in $\mathbf{Y}$ is a contradiction. This completes the proof of the claim. Hence, up to a subsequence, still denoted by itself, we have $y_n\to y_0\in\overline{\Lambda}$ as $n\to\infty$.\\
	\textbf{Claim II: } $y_0\in\Lambda$.\\
	Choose $\theta\in[N,+\infty)$ and $R>0$, then the sequence $\chi_n(x)=\chi_{\Lambda}(\varepsilon_n x+y_n)$ is bounded in $L^\theta(B_R)$. Due to the reflexivity of the space $L^\theta(B_R)$, there exists $\chi_R\in L^\theta(B_R) $ such that $\chi_n\rightharpoonup \chi_R $ in $L^\theta(B_R)$ as $n\to\infty$. Suppose $0<R_1<R_2$, then the functions $\chi_{R_1}$ and $\chi_{R_2}$ are obtained in the same way of $\chi_{R}$ satisfying $\chi_{R_1}=\chi_{R_2}|_{B_{R_1}}$. Hence, we conclude there exists a measurable function $\chi\in L^\theta_{\operatorname{loc}}(\mathbb{R}^N)$ with $0\leq \chi\leq 1$ such that we have
	\begin{equation}\label{eq4.32}
		\chi_n\rightharpoonup \chi\quad\text{in }L^\theta(B_R)\quad\text{as }n\to\infty \text{ for all }\theta\in[N,+\infty)\text{ and for all }R>0.
	\end{equation}
	Fix $\psi\in C^\infty_{c}(\mathbb{R}^N)$ and $R>0$ large enough with $\operatorname{supp}(\psi)\subset B_R$. Following similar arguments as in Claim I, we can see that when $\psi_j$ is replaced by $\psi$ in \eqref{eq4.28}, then \eqref{eq4.28} is also true for any $\psi\in  C^\infty_{c}(\mathbb{R}^N)$. Similarly, we can also see that \eqref{eq4.29} and \eqref{eq4.30} also hold by replacing $\psi_j$ by $\psi$. Further, since $\psi$ has compact support and $V\in C(\mathbb{R}^N,\mathbb{R})$, therefore by \eqref{eq4.26} and Lebesgue's dominated convergence theorem, one has
	\begin{align*}
		\int_{\mathbb{R}^N} V(\varepsilon_n x+y_n)w_n^{t-1}\psi\,\mathrm{d}x\to \int_{\mathbb{R}^N} V(y_0)w^{t-1}\psi\,\mathrm{d}x\quad\text{as }n\to\infty\text{ for } t\in\{p,N\}.
	\end{align*}
	Note that $C^\infty_c(\mathbb{R}^N)$ is a dense subset of  $L^\wp(\mathbb{R}^N)$ for any $\wp\in[1,+\infty)$. Therefore, by using \eqref{eq4.26} and the growth assumptions on $\mathcal{H}'_2$ and $\widetilde{\mathcal{H}}'_2$, we obtain from Lebesgue's dominated convergence theorem that
	\begin{align*}
		\mathcal{H}'_2(w_n)\psi\to \mathcal{H}'_2(w)\psi\quad\text{and}\quad\widetilde{\mathcal{H}}'_2(w_n)\psi\to \widetilde{\mathcal{H}}'_2(w)\psi\quad\text{in }L^{\frac{q}{q-1}}(B_R)\quad\text{as }n\to\infty\text{ for all }q>N.
	\end{align*}
	This together with \eqref{eq4.32} implies that
	\begin{align*}
		\int_{\mathbb{R}^N} G_2'(\varepsilon_n x+y_n,w_n)\psi\,\mathrm{d}x\to  \int_{\mathbb{R}^N} \widetilde{G}_2'(x,w)\psi\,\mathrm{d}x\quad\text{as }n\to\infty,
	\end{align*}
	where
	\begin{align*}
		\widetilde{G}_2'(x,s)&=\chi(x)\mathcal{H}'_2(s)+(1-\chi(x))\widetilde{\mathcal{H}}'_2(s)\quad\text{for all }(x,s)\in \mathbb{R}^N\times [0,+\infty),\\
		\widetilde{G}_2'(x,s)&=-\widetilde{G}_2'(x,-s)\quad\text{for all }(x,s)\in \mathbb{R}^N\times (-\infty,0].
	\end{align*}
	Recall that $\limsup_{n\to\infty} \|w_n\|^{N'}_{W^{1,N}}<\frac{\alpha_N}{\alpha_0}$ holds. Now, by using \eqref{eq4.26}, the growth assumptions on $f$ and $\widetilde{f}$ as well as Lemma \ref{lem2.11}, we obtain from Vitali's convergence theorem that
	\begin{align*}
		f(w_n)\psi\to f(w)\psi\quad\text{and}\quad\widetilde{f}(w_n)\psi\to \widetilde{f}(w)\psi\quad\text{in }L^{\frac{q}{q-1}}(B_R)\quad\text{as }n\to\infty\text{ for all }q>N.
	\end{align*}
	Consequently, from \eqref{eq4.32}, we infer that
	\begin{align*}
		\int_{\mathbb{R}^N} g(\varepsilon_n x+y_n,w_n)\psi\,\mathrm{d}x\to  \int_{\mathbb{R}^N} \widetilde{g}(x,w)\psi\,\mathrm{d}x\quad\text{as }n\to\infty,
	\end{align*}
	where
	\begin{align*}
		\widetilde{g}(x,s)&=\chi(x) f(s)+(1-\chi(x))\widetilde{f}(s)\quad\text{for all }(x,s)\in \mathbb{R}^N\times [0,+\infty),\\
		\widetilde{g}(x,s)&=-\widetilde{g}(x,-s)\quad\text{for all }(x,s)\in \mathbb{R}^N\times (-\infty,0].
	\end{align*}
Now, replacing $\psi$ in place of $\psi_j$ in \eqref{eq4.28} and	using the above convergence results, we obtain by letting $n\to\infty$ in \eqref{eq4.28} that
	\begin{equation}\label{eq4.33}
		\begin{aligned}
			& \int_{\mathbb{R}^N}\big(|\nabla w|^{p-2}\nabla w\cdot\nabla\psi+V(y_0)w^{p-1}\psi\big)\,\mathrm{d}x\\
			&+\int_{\mathbb{R}^N}\big(|\nabla w|^{N-2}\nabla w\cdot\nabla\psi+\big(V(y_0)+1\big)w^{N-1}\psi\big)\,\mathrm{d}x+\int_{\mathbb{R}^N} \mathcal{H}'_1(w)\psi\,\mathrm{d}x  \\
			 & =\int_{\mathbb{R}^N} \widetilde{G}_2'(x,w)\psi\,\mathrm{d}x+\int_{\mathbb{R}^N} \widetilde{g}(x,w)\psi\,\mathrm{d}x.
			\end{aligned}
	\end{equation}
	Note that for $q>N$, we have
	\begin{align*}
		|\widetilde{G}_2'(x,s)|\leq \ell |s|^{N-1}+C|s|^{q-1}\quad\text{and}\quad|\widetilde{g}(x,s)|\leq |f(s)|\quad\text{for all }(x,s)\in \mathbb{R}^N\times \mathbb{R}.
	\end{align*}
	In addition, the maps $t\mapsto \frac{\widetilde{G}_2'(x,s)}{s^{N-1}}$ and $t\mapsto \frac{\widetilde{g}(x,s)}{s^{N-1}}$ are nondecreasing on $(0,+\infty)$. Applying the density of $C^\infty_c(\mathbb{R}^N)$ in $\mathbf{Y}$, we infer from \eqref{eq4.33} that $w$ is a critical point for $\widetilde{J}$, that is, $\widetilde{J}'(w)=0$, where $\widetilde{J}\colon\mathbf{Y}\to\mathbb{R}$ is defined by
	\begin{align*}
		\widetilde{J}(u)&=\frac{1}{p} \|u\|_{W^{1,p}_{V(y_0)}}^p+\frac{1}{N} \Big(\|u\|_{W^{1,N}_{V(y_0)}}^N+\|u\|^N_{N}\Big)+\int_{\mathbb{R}^N}\mathcal{H}_1(u)\,\mathrm{d}x\\
		&\quad-\int_{\mathbb{R}^N}\widetilde{G}_2(x,u)\,\mathrm{d}x-\int_{\mathbb{R}^N} \widetilde{\mathcal{G}}(x,u)\,\mathrm{d}x\quad\text{for all }u\in\mathbf{Y},
	\end{align*}
	where
	\begin{align*}
		\widetilde{G}_2(x,s)=\displaystyle\int_{0}^{s}\widetilde{G}'_2(x,t)\,\mathrm{d}t
		\quad\text{and}\quad
		\widetilde{\mathcal{G}}(x,s)=\displaystyle\int_{0}^{s}\widetilde{g}(x,t)\,\mathrm{d}t\quad\text{for all }(x,s)\in \mathbb{R}^N\times \mathbb{R}.
	\end{align*}
	We define $J_{V(y_0)}\colon\mathbf{Y}\to\mathbb{R}$ by
	\begin{align*}
		J_{V(y_0)}(u)
		&=\frac{1}{p} \|u\|_{W^{1,p}_{V(y_0)}}^p+\frac{1}{N} \Big(\|u\|_{W^{1,N}_{V(y_0)}}^N+\|u\|^N_{N}\Big)+\int_{\mathbb{R}^N}\mathcal{H}_1(u)\,\mathrm{d}x\\
		&-\int_{\mathbb{R}^N}\mathcal{H}_2(u)\,\mathrm{d}x-\int_{\mathbb{R}^N} F(u)\,\mathrm{d}x\quad\text{for all }u\in\mathbf{Y}.
	\end{align*}
	Further, let
	\begin{align*}
		\mathcal{N}_{V(y_0)}&=\{u\in\mathbf{Y}\setminus\{0\}\colon ~\langle{J_{V(y_0)}'(u),u}\rangle=0\},\\
		c_{V(y_0)}&=\inf_{u\in\mathcal{N}_{V(y_0)}}J_{V(y_0)}(u)=\inf_{u\in\mathbf{Y}\setminus\{0\}}\max_{t\geq 0} J_{V(y_0)}(tu).
	\end{align*}
	Denote $\mathcal{O}_0=\ \{u\in\mathbf{Y}_\varepsilon\colon|\operatorname{supp}(|u|)\cap \operatorname{supp}\chi|>0\}$. Now, using the growth assumptions on $\widetilde{G}'_2$ and $\widetilde{g}$ and employing similar arguments explored in Lemma \ref{lem3.2}, we can deduce that for fixed $u\in \mathcal{O}_0\setminus\{0\} $ with $u\geq 0$ a.e.\,in $\mathbb{R}^N$, $\widetilde{J}(tu)\to-\infty$ as $t\to\infty$. In virtue of $w\neq 0$ and $\widetilde{J}'(w)=0$, it follows that $w\in\mathcal{O}_0$. Consequently, we have
	\begin{align*}
		\widetilde{J}(w)=\max_{t\geq 0} \widetilde{J}(tw)\geq \max_{t\geq 0} J_{V(y_0)}(tw)\geq c_{V(y_0)}.
	\end{align*}
	Following similarly ideas as in Lemma \ref{lem3.4} and using the change of variable $z\mapsto x+\widetilde{y}_n$ together with ($\mathcal{A}$)(d) and $N\mathcal{G}(x,s)\leq sg(x,s)$ for all $(x,s)\in\mathbb{R}^N\times[0,+\infty)$, we get
	\begin{align*}
		c_{\varepsilon_n} & = J_{\varepsilon_n}(u_n)-\frac{1}{N}\langle{J'_{\varepsilon_n}(u_n),u_n}\rangle  \\
		& =\bigg(\frac{1}{p}-\frac{1}{N}\bigg)\int_{\mathbb{R}^N}\big(|\nabla w_n|^{p}+V(\varepsilon_n x+y_n)|w_n|^{p}\big)\,\mathrm{d}x\\
		&\quad +\int_{\mathbb{R}^N} \Big(\frac{1}{N}g(\varepsilon_n x+y_n,w_n)w_n-\mathcal{G}(\varepsilon_n x+y_n,w_n)\Big)\,\mathrm{d}x \\
		 & \quad +\int_{\mathbb{R}^N} \Big(\frac{1}{N}|w_n|^N+\mathcal{H}_2(w_n)-\frac{1}{N}\mathcal{H}'_2(w_n)w_n+\frac{1}{N}G'_2(\varepsilon_n x+y_n,w_n)w_n-G_2(\varepsilon_n x+y_n,w_n)\Big)\,\mathrm{d}x                                          \\
		 & \geq  \bigg(\frac{1}{p}-\frac{1}{N}\bigg)\int_{B_R}\big(|\nabla w_n|^{p}+V(\varepsilon_n x+y_n)|w_n|^{p}\big)\,\mathrm{d}x\\
		 &\quad +\int_{B_R} \Big(\frac{1}{N}g(\varepsilon_n x+y_n,w_n)w_n-\mathcal{G}(\varepsilon_n x+y_n,w_n)\Big)\,\mathrm{d}x              \\
		 & \quad +\int_{B_R} \Big(\frac{1}{N}|w_n|^N+\mathcal{H}_2(w_n)-\frac{1}{N}\mathcal{H}'_2(w_n)w_n+\frac{1}{N}G'_2(\varepsilon_n x+y_n,w_n)w_n-G_2(\varepsilon_n x+y_n,w_n)\Big)\,\mathrm{d}x
	\end{align*}
	for all $R>0$. Now, from \eqref{eq4.26} and the growth assumptions on $\mathcal{H}'_2$ and $\widetilde{\mathcal{H}}'_2$ together with  Lebesgue's dominated convergence implies that
	\begin{align*}
		\mathcal{H}'_2(w_n)w_n\to \mathcal{H}'_2(w)w\quad\text{and}\quad\widetilde{\mathcal{H}}'_2(w_n)w_n\to \widetilde{\mathcal{H}}'_2(w)w\quad\text{in } L^{\frac{q+1}{q}}(B_R)\quad\text{as }n\to\infty\text{ for all }q>N.
	\end{align*}
	The above result together with \eqref{eq4.32} gives
	\begin{align*}
		\int_{B_R} G_2'(\varepsilon_n x+y_n,w_n)w_n\,\mathrm{d}x\to  \int_{B_R} \widetilde{G}_2'(x,w)w\,\mathrm{d}x\quad\text{as }n\to\infty.
	\end{align*}
	Likewise, we can prove that
	\begin{align*}
		\int_{B_R} G_2(\varepsilon_n x+y_n,w_n)\,\mathrm{d}x\to  \int_{B_R} \widetilde{G}_2(x,w)\,\mathrm{d}x\quad\text{as }n\to\infty.
	\end{align*}
	Further, using $\limsup_{n\to\infty} \|w_n\|^{N'}_{W^{1,N}}<\frac{\alpha_N}{\alpha_0}$, \eqref{eq4.26}, the growth assumptions on $f$ and $\widetilde{f}$, and Lemma \ref{lem2.11}, we deduce from Vitali's convergence theorem that
	\begin{align*}
		f(w_n)w_n\to f(w)w\quad\text{and}\quad\widetilde{f}(w_n)w_n\to \widetilde{f}(w)w\quad\text{in } L^{\frac{q+1}{q}}(B_R)\quad\text{as }n\to\infty\text{ for all }q>N.
	\end{align*}
	Now, it follows from \eqref{eq4.32} that
	\begin{align*}
		\int_{B_R} g(\varepsilon_n x+y_n,w_n)w\,\mathrm{d}x\to  \int_{B_R} \widetilde{g}(x,w)w\,\mathrm{d}x\quad\text{as }n\to\infty.
	\end{align*}
	Similarly, we can prove that
	\begin{align*}
		\int_{B_R} \mathcal{G}(\varepsilon_n x+y_n,w_n)\,\mathrm{d}x\to  \int_{B_R} \widetilde{\mathcal{G}}(x,w)\,\mathrm{d}x\quad\text{as }n\to\infty.
	\end{align*}
	In a similar fashion, we obtain from Lebesgue's dominated convergence theorem that
	\begin{align*}
		\int_{B_R} \mathcal{H}_2(w_n)\,\mathrm{d}x\to  \int_{B_R} \mathcal{H}_2(w)\,\mathrm{d}x\quad\text{and}\quad\int_{B_R} \mathcal{H}_2'(w_n)w_n\,\mathrm{d}x\to  \int_{B_R} \mathcal{H}_2'(w)w\,\mathrm{d}x\quad\text{as }n\to\infty.
	\end{align*}
	Gathering all the above information, we obtain by using Fatou's lemma and Lemma \ref{lem4.5}(c) that
	\begin{align*}
		 & c_0   \geq  \bigg(\frac{1}{p}-\frac{1}{N}\bigg)\int_{B_R}\big(|\nabla w|^{p}+V(y_0)|w|^{p}\big)\,\mathrm{d}x+\int_{B_R} \Big(\frac{1}{N}\widetilde{g}(x,w)w-\widetilde{\mathcal{G}}(x,w)\Big)\,\mathrm{d}x \\
		 & \qquad +\int_{B_R} \Big(\frac{1}{N}|w|^N+\mathcal{H}_2(w)-\frac{1}{N}\mathcal{H}'_2(w)w+\frac{1}{N}\widetilde{G}'_2(x,w)w-\widetilde{G}_2(x,w)\Big)\,\mathrm{d}x.
	\end{align*}
	Letting $R\to\infty$ in the above inequality, we get
	\begin{align*}
		c_0
		& \geq  \bigg(\frac{1}{p}-\frac{1}{N}\bigg)\int_{\mathbb{R}^N}\big(|\nabla w|^{p}+V(y_0)|w|^{p}\big)\,\mathrm{d}x+\int_{\mathbb{R}^N} \Big(\frac{1}{N}\widetilde{g}(x,w)w-\widetilde{\mathcal{G}}(x,w)\Big)\,\mathrm{d}x\\
		& \quad +\int_{\mathbb{R}^N} \Big(\frac{1}{N}|w|^N+\mathcal{H}_2(w)-\frac{1}{N}\mathcal{H}'_2(w)w+\frac{1}{N}\widetilde{G}'_2(x,w)w-\widetilde{G}_2(x,w)\Big)\,\mathrm{d}x\\
		&= \widetilde{J}(w)-\frac{1}{N}\langle{ \widetilde{J}(w),w}\rangle=\widetilde{J}(w)\geq c_{V(y_0)}.
	\end{align*}
	It follows from the above inequality and the definitions of $c_0$ and $c_{V(y_0)}$ that $V(y_0)\leq V_0=\inf_{x\in \Lambda} V(x)$. Consequently, by (V2), we have $V(y_0)=V_0$ and $y_0\in\Lambda$. This finishes the proof of the claim.

	To finish the proof, we only have to prove $w_n\to w$ in $\mathbf{Y}$ as $n\to\infty$. For this, we have the following claim.
	\textbf{Claim III: } $w_n\to w$ in $\mathbf{Y}$ as $n\to\infty$.\\
	Notice that
	\begin{align*}
		\mathbb{R}^N=\big(\Lambda_n\cup \{w_n\leq t_1\} \big) \cup \big(\Lambda_n^c\cap\{w_n>t_1\}\big),
	\end{align*}
	where $\Lambda_n=\frac{\Lambda-y_n}{\varepsilon_n}$. From the definitions of $G_2$ and $\mathcal{H}_2$ along with ($\mathcal{B}$)(d), one has
	\begin{align*}
		&\frac{1}{N}G'_2(\varepsilon_n x+y_n,w_n)w_n-G_2(\varepsilon_n x+y_n,w_n)\geq 0\quad\text{on }\Lambda_n^c\cap\{w_n>t_1\},\\[1ex]
		&\frac{1}{N}g(\varepsilon_n x+y_n,w_n)w_n-\mathcal{G}(\varepsilon_n x+y_n,w_n)\geq 0\quad\text{on }\Lambda_n^c\cap\{w_n>t_1\},
	\end{align*}
	and
	\begin{align*}
		\frac{1}{N}|w_n|^N+\mathcal{H}_2(w_n)-\frac{1}{N}\mathcal{H}'_2(w_n)w_n>\frac{\big((N-1)\delta\big)^N}{N}>0 \quad\text{for}~w_n>t_1.
	\end{align*}
	It follows from $w_n\to w$ a.e.\,in $\mathbb{R}^N$, $\nabla w_n\to \nabla w$ a.e.\,in $\mathbb{R}^N$, $y_n\to y_0$ and $\chi_{\Lambda_n}\to 1$ a.e.\,in $\mathbb{R}^N$ as $n\to\infty$ that
	\begin{enumerate}
		\item[\textnormal{(i)}]
			$ \big(|\nabla w_n|^{p}+V(\varepsilon_n x+y_n)|w_n|^{p}\big)\chi_{\Lambda_n^c\cap\{w_n>t_1\}}\to 0$ a.e.\,in $\mathbb{R}^N$ as $n\to\infty$,\vspace{0.2cm}
		\item[\textnormal{(ii)}]
			$\Big(\frac{1}{N}G'_2(\varepsilon_n x+y_n,w_n)w_n-G_2(\varepsilon_n x+y_n,w_n)\Big)\chi_{\Lambda_n^c\cap\{w_n>t_1\}}\to 0$ a.e.\,in $\mathbb{R}^N$ as $n\to\infty$,\vspace{0.2cm}
		\item[\textnormal{(iii)}]
			$\Big(\frac{1}{N}|w_n|^N+\mathcal{H}_2(w_n)-\frac{1}{N}\mathcal{H}'_2(w_n)w_n\Big)\chi_{\Lambda_n^c\cap\{w_n>t_1\}}\to 0$ a.e.\,in $\mathbb{R}^N$ as $n\to\infty$,\vspace{0.2cm}
		\item[\textnormal{(iv)}]
			$\Big(\frac{1}{N}g(\varepsilon_n x+y_n,w_n)w_n-\mathcal{G}(\varepsilon_n x+y_n,w_n)\Big)\chi_{\Lambda_n^c\cap\{w_n>t_1\}}\to 0$ a.e.\,in $\mathbb{R}^N$ as $n\to\infty$.
	\end{enumerate}
	In virtue of $\widetilde{G}_2'(x,s)\leq \mathcal{H}'(s)$, $\widetilde{g}_2'(x,s)\leq f(s)$ for all $(x,s)\in \mathbb{R}^N\times[0,+\infty)$ and $V(y_0)=V_0$, we obtain from $\langle{\widetilde{J}'(w),w}\rangle=0$ that $\langle{J_0'(w),w}\rangle\leq 0$. Define $\xi\colon(0,+\infty)\to\mathbb{R}$ by $\xi(t)=\langle{J_0'(tw),tw}\rangle$. It follows immediately that $\xi(1)\leq 0$. Choose $0<t<1$, then by using ($\mathcal{H}_1$)(b), ($\mathcal{H}_2$)(b), \eqref{eq3.2} with $\vartheta=q>N$ and Corollary \ref{cor2.9}, we get
	\begin{align*}
		\xi(t)\geq t^N\Big(\|w\|_{W^{1,p}_{V_0}}^p+\|w\|_{W^{1,N}_{V_0}}^N+(1-\tau)\|w\|^N_{N}\Big)-t^q\bigg(C\|w\|^q_q+\widetilde{\kappa}_\tau\int_{\mathbb{R}^N}|w|^q\Phi(\alpha |w|^{N'})\,\mathrm{d}x\bigg).
	\end{align*}
	Due to the arbitrariness of $\tau>0$, we can choose $\tau>0$ small enough such that $V_0+1-\tau=:\sigma>0$. Consequently, we have
	\begin{align*}
		\xi(t)\geq t^N\Big(\|w\|_{W^{1,p}_{V_0}}^p+\|w\|_{W^{1,N}_{\sigma}}^N\Big)-t^q\bigg(C\|w\|^q_q+\widetilde{\kappa}_\tau\int_{\mathbb{R}^N}|w|^q\Phi(\alpha |w|^{N'})\,\mathrm{d}x\bigg).
	\end{align*}
	This shows that there exist constants $C_1,C_2>0$ such that
	\begin{align*}
		\xi(t)\geq C_1 t^N-C_2t^q>0\quad\text{for } t\in(0,1) \text{ sufficiently small}.
	\end{align*}
	Due to the continuity of $\xi$, we infer that there exists $t_0\in(0,1]$ such that $\xi(t_0)=0$, that is, $t_0 w\in\mathcal{N}_0$. Now, combining all the above information, we obtain from the change of variable, Lemma \ref{lem4.5}(c), Fatou's Lemma and (f3) that
    \begin{equation}\label{eq4.34}
		\begin{aligned}
		c_0
		& \geq \limsup_{n\to\infty}c_{\varepsilon_n}=\limsup_{n\to\infty}J_{\varepsilon_n}(u_n) =\limsup_{n\to\infty}\bigg[J_{\varepsilon_n}(u_n)-\frac{1}{N}\langle{J'_{\varepsilon_n}(u_n),u_n}\rangle\bigg] \\
		& =\limsup_{n\to\infty}\bigg[\bigg(\frac{1}{p}-\frac{1}{N}\bigg)\int_{\mathbb{R}^N}\big(|\nabla w_n|^{p}+V(\varepsilon_n x+y_n)|w_n|^{p}\big)\mathrm{d}x\\
		&\qquad\qquad\qquad +\int_{\mathbb{R}^N} \Big(\frac{1}{N}|w_n|^N+\mathcal{H}_2(w_n)-\frac{1}{N}\mathcal{H}'_2(w_n)w_n\Big)\,\mathrm{d}x\\
		&\qquad\qquad\qquad
		+ \int_{\mathbb{R}^N} \Big(\frac{1}{N}g(\varepsilon_n x+y_n,w_n)w_n-\mathcal{G}(\varepsilon_n x+y_n,w_n)\Big)\mathrm{d}x\\
		&\qquad\qquad\qquad  +\int_{\mathbb{R}^N} \Big(\frac{1}{N}G'_2(\varepsilon_n x+y_n,w_n)w_n-G_2(\varepsilon_n x+y_n,w_n)\Big)\mathrm{d}x\bigg]  \\
		& \geq \liminf_{n\to\infty}\bigg[\bigg(\frac{1}{p}-\frac{1}{N}\bigg)\int_{\mathbb{R}^N}\big(|\nabla w_n|^{p}+V(\varepsilon_n x+y_n)|w_n|^{p}\big)\chi_{\Lambda_n}\mathrm{d}x\\
		&\qquad\qquad\qquad +\int_{\mathbb{R}^N} \Big(\frac{1}{N}f(w_n)w_n-F(w_n)\Big)\chi_{\Lambda_n}\mathrm{d}x\\
		&\qquad\qquad\qquad +\int_{\mathbb{R}^N}\frac{1}{N}|w_n|^N\chi_{\Lambda_n}\,\mathrm{d}x\\
		&\qquad\qquad\qquad+\bigg(\frac{1}{p}-\frac{1}{N}\bigg)\int_{\mathbb{R}^N}\big(|\nabla w_n|^{p}+V(\varepsilon_n x+y_n)|w_n|^{p}\big)\chi_{\Lambda_n^c\cap\{w_n>t_1\}}\,\mathrm{d}x\\
		&\qquad\qquad\qquad +\int_{\mathbb{R}^N}\bigg(\frac{1}{N}|w_n|^N+\mathcal{H}_2(w_n)-\frac{1}{N}\mathcal{H}'_2(w_n)w_n\bigg)\chi_{\Lambda_n^c\cap\{w_n>t_1\}}\,\mathrm{d}x\\
		&\qquad\qquad \qquad+\int_{\mathbb{R}^N}\bigg(\frac{1}{N}G'_2(\varepsilon_n x+y_n,w_n)w_n-G_2(\varepsilon_n x+y_n,w_n)\bigg)\chi_{\Lambda_n^c\cap\{w_n>t_1\}}\,\mathrm{d}x\\
		&\qquad\qquad \qquad+\int_{\mathbb{R}^N} \Big(\frac{1}{N}g(\varepsilon_n x+y_n,w_n)w_n-\mathcal{G}(\varepsilon_n x+y_n,w_n)\Big)\chi_{\Lambda_n^c\cap\{w_n>t_1\}}\mathrm{d}x \bigg]
        \end{aligned}
        \end{equation}

	\begin{equation*}
		\begin{aligned}
		& \geq \bigg(\frac{1}{p}-\frac{1}{N}\bigg)\int_{\mathbb{R}^N}\big(|\nabla w|^{p}+V_0|w|^{p}\big)\,\mathrm{d}x+\frac{1}{N}\int_{\mathbb{R}^N}|w|^N\,\mathrm{d}x+\int_{\mathbb{R}^N} \Big(\frac{1}{N}f(w)w-F(w)\Big)\,\mathrm{d}x \\
		& \geq \bigg(\frac{1}{p}-\frac{1}{N}\bigg)t_0^p\int_{\mathbb{R}^N}\big(|\nabla w|^{p}+V_0|w|^{p}\big)\,\mathrm{d}x+\frac{t_0^N}{N}\int_{\mathbb{R}^N}|w|^N\,\mathrm{d}x\\
		&\qquad\qquad\qquad\qquad+\int_{\mathbb{R}^N} \Big(\frac{1}{N}f(t_0w)t_0w-F(t_0 w)\Big)\,\mathrm{d}x  \\
		& =J_{0}(t_0 w)-\frac{1}{N}\langle{J'_{0}(t_0w, t_0w}\rangle=J_{0}(t_0 w)\geq c_0.
		\end{aligned}
	\end{equation*}
	It follows that $t_0=1$. Moreover, due to the simple change of variable, we obtain the following convergences
	\begin{align*}
		\int_{\mathbb{R}^N}\big(|\nabla w_n|^{t}+V(\varepsilon_n x+y_n)|w_n|^{t}\big)\,\mathrm{d}x\to\int_{\mathbb{R}^N}\big(|\nabla w|^{t}+V_0|w|^{t}\big)\,\mathrm{d}x\quad\text{as }n\to\infty\text{ for }t\in\{p,N\}
	\end{align*}
	and
	\begin{align*}
		\int_{\mathbb{R}^N}\mathcal{H}'_1(w_n)w_n\,\mathrm{d}x\to\int_{\mathbb{R}^N}\mathcal{H}'_1(w)w\,\mathrm{d}x \quad\text{as }n\to\infty.
	\end{align*}
	Consequently, it is not difficult to show that
	\begin{align*}
		\int_{\mathbb{R}^N}|w_n|^t\,\mathrm{d}x\to\int_{\mathbb{R}^N}|w|^t\,\mathrm{d}x \quad\text{as }n\to\infty\text{ for }t\in\{p,N\}.
	\end{align*}
	This implies that $\|w_n\|_{W^{1,t}_{V_0}}\to\|w\|_{W^{1,t}_{V_0}}$ as $n\to\infty$ for $t\in\{p,N\}$. Now, arguing similarly as in Lemma \ref{lem3.7}, we can deduce that $w_n\to w$ in $\mathbf{Y}$ as $n\to\infty$. This completes the proof.
\end{proof}

\begin{remark}\label{rem4.9}
	In view of \eqref{eq4.34}, we deduce that there holds $\lim_{\varepsilon\to0}c_\varepsilon=c_0$.
\end{remark}

To study the behavior of the maximum points of the solutions, the following lemma is quite important. The proof relies on the appropriate Moser iteration argument found in Moser \cite{Moser-1960} and the notions discussed in Ambrosio \cite{Ambrosio-2024}.

\begin{lemma}\label{lem4.10}
	Let $\{w_n\}_{n\in\mathbb{N}}$ be a sequence that appears in Lemma \ref{lem4.8}. Then, $\{w_n\}_{n\in\mathbb{N}}\subset L^\infty(\mathbb{R}^N)$ and there is a constant $K>0$ such that
	\begin{align*}
		\|w_n\|_{L^\infty(\mathbb{R}^N)}\leq K
		\quad\text{for all }n\in\mathbb{N}.
	\end{align*}
	In addition, we have
	\begin{equation}\label{eq4.35}
		\lim_{|x|\to\infty} \sup_{n\in\mathbb{N}}|w_n(x)|=0.
	\end{equation}
\end{lemma}

\begin{proof}
	For all $L>0$ and $\beta>1$, we define $\gamma(w_n)=w_n w_{n,L}^{N(\beta-1)}~\text{and}~u_{n,L}=w_n w_{n,L}^{\beta-1}$, where $w_{n,L}=\min\{w_n,L\}$. Note that the function $\gamma$ is an increasing function, so we have $(a-b)(\gamma(a)-\gamma(b))\geq 0$ for all $a,b\in\mathbb{R}$. Define
	\begin{align*}
		\Lambda(t)=\frac{|t|^N}{N}
		\quad\text{and}\quad
		\Gamma(t)=\int_{0}^{t}(\gamma'(s))^{\frac{1}{N}}\,\mathrm{d}s.
	\end{align*}
	Invoking Lemma 29 of Zhang-Sun-Liang-Thin \cite{Zhang-Sun-Liang-Thin-2024}, one sees that
	\begin{align*}
		\Lambda' (a-b)(\gamma(a)-\gamma(b))\geq |\Gamma(a)-\Gamma(b)|^N\quad\text{for all }a,b\in\mathbb{R}.
	\end{align*}
	Consequently, there holds
	\begin{align*}
		\frac{1}{\beta} w_n w_{n,L}^{\beta-1}\leq \Gamma(w_n)\leq  w_n w_{n,L}^{\beta-1}.
	\end{align*}
	Let $\eta>0$ be specified later and $S$ be the best Sobolev constant of the embedding $W^{1,N}(\mathbb{R}^N)\hookrightarrow L^{N^*}(\mathbb{R}^N)$ for any $N^*>N$. Using the fact that $\|\cdot\|_{W^{1,N}_{\eta}}$ is an equivalent norm for $W^{1,N}(\mathbb{R}^N)$, the last inequality yields
	\begin{equation}\label{eq4.36}
		\|w_n w_{n,L}^{\beta-1}\|^N_{N^*}\leq C_1\beta^N\bigg[\int_{\mathbb{R^N}} w_{n,L}^{N(\beta-1)}|\nabla w_n|^N\,\mathrm{d}x+\eta\int_{\mathbb{R^N}} w_n^N w_{n,L}^{N(\beta-1)}\,\mathrm{d}x\bigg],
	\end{equation}
	where $C_1=C_1(N,\beta, S)>0$ is a constant. Due to $w_{n,L}\leq L$, we have $ \mathcal{H}'_1(w_n)\gamma(w_n)\leq L^{N(\beta-1)}\mathcal{H}'_1(w_n)w_n\in L^1(\mathbb{R}^N)$. By direct calculation, one has
	\begin{align*}
		\int_{\mathbb{R}^N}|\nabla w_n|^{p-2}\nabla w_n\cdot\nabla(\gamma(w_n))\,\mathrm{d}x=\int_{\mathbb{R}^N}w_{n,L}^{N(\beta-1)}|\nabla w_n|^p\,\mathrm{d}x+N(\beta-1)\int_{\{w_n\leq L\}}w_{n,L}^{N(\beta-1)}|\nabla w_n|^p\,\mathrm{d}x\geq 0,
	\end{align*}
	\begin{align*}
		\int_{\mathbb{R}^N}|\nabla w_n|^{N-2}\nabla w_n\cdot\nabla(\gamma(w_n))\,\mathrm{d}x
		& =\int_{\mathbb{R}^N}w_{n,L}^{N(\beta-1)}|\nabla w_n|^N\,\mathrm{d}x+N(\beta-1)\int_{\{w_n\leq L\}}w_{n,L}^{N(\beta-1)}|\nabla w_n|^N\,\mathrm{d}x  \\
		& \geq \int_{\mathbb{R}^N}w_{n,L}^{N(\beta-1)}|\nabla w_n|^N\,\mathrm{d}x,
	\end{align*}
	and
	\begin{align*}
		\int_{\mathbb{R}^N}V(\varepsilon_n x+y_n)(w_n^{p-1}+w_n^{N-1})\gamma(w_n)\,\mathrm{d}x & \geq V_0 \int_{\mathbb{R}^N}w_n^{p}w_{n,L}^{N(\beta-1)}\,\mathrm{d}x+V_0\int_{\mathbb{R}^N}w_n^{N}w_{n,L}^{N(\beta-1)}\,\mathrm{d}x  \\
		& \geq V_0\int_{\mathbb{R}^N}w_n^{N}w_{n,L}^{N(\beta-1)}\,\mathrm{d}x.
	\end{align*}
	Combining all the above information and using ($\mathcal{H}_1$)(b), ($\mathcal{A}$)(a), ($\mathcal{B}$)(b) and \eqref{eq3.2}, we obtain by taking $\gamma(w_n)$ as test function in the weak formulation of the problem solved by $w_n$ that
	\begin{equation}\label{eq4.37}
		\begin{aligned}
			&\int_{\mathbb{R^N}} w_{n,L}^{N(\beta-1)}|\nabla w_n|^N\,\mathrm{d}x+\eta\int_{\mathbb{R^N}} w_n^N w_{n,L}^{N(\beta-1)}\,\mathrm{d}x\\
			&\leq C_q \int_{\mathbb{R^N}} w_n^q w_{n,L}^{N(\beta-1)}\,\mathrm{d}x+\widetilde{\kappa}_\tau \int_{\mathbb{R^N}}w_n^\vartheta w_{n,L}^{N(\beta-1)}\Phi(\alpha w_n^{N'})\,\mathrm{d}x,
		\end{aligned}
	\end{equation}
	where we have used that $\tau\in(0,\ell']$ to be small enough and $\eta=V_0+1-(\ell+\ell')>0$, thanks to $V_0+1\geq 2(\ell+\ell')$. Consequently, we obtain from \eqref{eq4.36} and \eqref{eq4.37} that
	\begin{equation}\label{eq4.38}
		\|w_n w_{n,L}^{\beta-1}\|^N_{N^*}\leq C_1\beta^N\bigg[C_q \int_{\mathbb{R^N}} w_n^q w_{n,L}^{N(\beta-1)}\,\mathrm{d}x+\widetilde{\kappa}_\tau \int_{\mathbb{R^N}}w_n^\vartheta w_{n,L}^{N(\beta-1)}\Phi(\alpha w_n^{N'})\,\mathrm{d}x\bigg].
	\end{equation}
	By the hypothesis, we have $\limsup_{n\to\infty} \|w_n\|^{N'}_{W^{1,N}}<\frac{\alpha_N}{\alpha_0}$. Following similar arguments as in Fiscella-Pucci \cite[Theorem 1.1]{Fiscella-Pucci-2021}, up to a subsequence, not relabeled, we may suppose that
	\begin{equation}\label{eq4.39}
		\sup_{n\in\mathbb{N}} \|w_n\|^{N'}_{W^{1, N}}<\frac{\alpha_N}{\alpha_0}.
	\end{equation}
	Set $\min\{\vartheta,q\}\geq 2N$ and take $\mu,\mu',\sigma>1$ such that $\frac{1}{\mu}+\frac{1}{\mu'}=1$. It is easy to see that $\mu'(q-N)\geq N$ and $\sigma(\vartheta-N)\geq N$. Further, we select $t>1$ such that $\frac{1}{t}+\frac{1}{\sigma}+\frac{1}{\mu}=1$. In view of \eqref{eq4.39}, we can find $m>0$ such that $\|w_n\|^{N'}_{W^{1, N}}<m<\frac{\alpha_N}{\alpha_0}$ for all $n\in\mathbb{N}$. Let $t>1$ be close to $1$ and $\alpha>\alpha_0$ be close to $\alpha_0$ such that we still have $\alpha t \|w_n\|^{N'}_{W^{1, N}}<m<\alpha_N$ for all $n\in\mathbb{N}$ and $\widetilde{w}_n = \frac{w_n}{\|w_n\|_{W^{1,N}}}$. Using that $\{w_n\}_{n\in\mathbb{N}}$ is bounded in $\mathbf{Y}$, it follows from the generalized H\"older inequality, Corollary \ref{cor2.9} and Lemma \ref{lem8.98} that
	\begin{equation}\label{eq4.40}
		\begin{aligned}
			&C_q \int_{\mathbb{R^N}} w_n^q w_{n,L}^{N(\beta-1)}\,\mathrm{d}x
			+\widetilde{\kappa}_\tau \int_{\mathbb{R^N}}w_n^\vartheta w_{n,L}^{N(\beta-1)}\Phi(\alpha w_n^{N'})\,\mathrm{d}x\\
			&=C_q \int_{\mathbb{R^N}} w_n^{q-N} u_{n,L}^{N}\,\mathrm{d}x+\widetilde{\kappa}_\tau \int_{\mathbb{R^N}}w_n^{\vartheta-N} u_{n,L}^{N}\Phi(\alpha w_n^{N'})\,\mathrm{d}x  \\
			& \leq \Big[ C_q \|w_n\|^{q-N}_{\mu'(q-N)}+\widetilde{\kappa}_\tau \|w_n\|^{\vartheta-N}_{\sigma(\vartheta-N)} \bigg(\int_{\mathbb{R}^N}\Phi\big(\alpha t\|w_n\|^{N'}_{W^{1,N}}|\widetilde{w}_n|^{N'}\big)\,\mathrm{d}x\bigg)^{\frac{1}{t}}\Big] \|u_{n,L}\|^N_{N\mu}                \\
			& \leq \tilde{C} \|u_{n,L}\|^N_{N\mu}=\tilde{C} \|w_n w^{\beta-1}_{n,L}\|^N_{N\mu}\leq \tilde{C} \|w_n\|^{N\beta}_{N\beta\mu},
		\end{aligned}
	\end{equation}
	where the last inequality is obtained by using $w_{n,L}\leq w_n$ and $\tilde{C}$ is defined as
	\begin{align*}
		\tilde{C}=C_q \sup_{n\in\mathbb{N}} \|w_n\|^{q-N}_{\mu'(q-N)}+\widetilde{\kappa}_\tau \sup_{n\in\mathbb{N}} \|w_n\|^{\vartheta-N}_{\sigma(\vartheta-N)} \bigg( \sup_{n\in\mathbb{N}} \int_{\mathbb{R}^N}\Phi\big(\alpha t\|w_n\|^{N'}_{W^{1,N}}|\widetilde{w}_n|^{N'}\big)\,\mathrm{d}x\bigg)^{\frac{1}{t}}<+\infty,
	\end{align*}
	thanks to Lemma \ref{lem2.11}. In virtue of \eqref{eq4.38} and \eqref{eq4.40}, we deduce that there exists a constant $C>0$ such that
	\begin{align*}
		\|w_n w_{n,L}^{\beta-1}\|^N_{N^*}\leq C\beta^N \|w_n\|^{N\beta}_{N\beta\mu}.
	\end{align*}
	By applying Fatou's lemma as $L\to\infty$, we obtain from the above inequality by setting $N\mu= N^{* *}<N^{*}$ that
	\begin{equation}\label{eq4.41}
		\|w_n\|_{N^*\beta}\leq C^{\frac{1}{N\beta}}\beta^{\frac{1}{\beta}}\|w_n\|_{N^{* *}\beta}.
	\end{equation}
	Define $\beta=\frac{N^{*}}{N^{* *}}>1$, then one has $N^{* *}\beta^2=N^{*}\beta$. Replacing $\beta$ with $\beta^2$ in \eqref{eq4.41}, we get
	\begin{align*}
		\|w_n\|_{N^*\beta^2}
		& \leq C^{\frac{1}{N\beta^2}}\beta^{\frac{2}{\beta^2}}\|w_n\|_{N^{* *}\beta^2} =C^{\frac{1}{N\beta^2}}\beta^{\frac{2}{\beta^2}}\|w_n\|_{N^{*}\beta}  \\
		& \leq C^{\frac{1}{N}\big(\frac{1}{\beta}+\frac{1}{\beta^2}\big)}\beta^{\frac{1}{\beta}+\frac{2}{\beta^2}}\|w_n\|_{N^{*}\beta}.
	\end{align*}
	Using the fact that $\|w_n\|_{N^{*}\beta}\leq C_2$ for all $n\in\mathbb{N}$, where $C_2>0$ is a constant and iterating the formula \eqref{eq4.41}, we get
	\begin{align*}
		\|w_n\|_{N^*\beta^m}\leq C^{\displaystyle\Sigma_{i=1}^{m}\frac{1}{N\beta^i}}\beta^{\displaystyle\Sigma_{i=1}^{m}\frac{i}{\beta^i}}\|w_n\|_{N^{*}\beta}\leq C_2 C^{\displaystyle\Sigma_{i=1}^{m}\frac{1}{N\beta^i}}\beta^{\displaystyle\Sigma_{i=1}^{m}\frac{i}{\beta^i}}~\text{for all}~n,m\in\mathbb{N}.
	\end{align*}
	By d'Alembert's ratio test, the series $\displaystyle{\Sigma_{i=1}^{\infty}}N^{-1}\beta^{-i}$ and $\displaystyle\Sigma_{i=1}^{\infty}i\beta^{-i}$ are convergent. Therefore, by letting $m\to\infty$ in the above inequality, we obtain
	\begin{equation}\label{eq4.42}
		\|w_n\|_{L^\infty(\mathbb{R}^N)}\leq K~\text{for all}~n\in\mathbb{N}
	\end{equation}
	for some constant $K>0$. Now, employing the standard regularity theory for quasilinear elliptic equations (see the works in \cite{DiBenedetto-1983, He-Li-2008, Ladyzhenskaya-Uraltseva-1968, Li-1990, Serrin-1964, Tolksdorf-1984}), we conclude that $w_n\in C_{\operatorname{loc}}^{0,\alpha}(\mathbb{R}^N)$ for some $0 < \alpha < 1$ and for each $n\in\mathbb{N}$. Moreover, due to the embedding $\mathbf{Y}\hookrightarrow \mathbf{X}$, Corollary \ref{cor2.3} and \eqref{eq4.42}, we deduce that $\{w_n\}_{n\in\mathbb{N}}$ is bounded in $\mathbf{X}$, $\|w_n\|_\theta\leq C$ for all $n\in\mathbb{N}$ and $\theta\in[p,p^*]\cup[N,+\infty]$. Moreover, $w_n\to w$ in $L^\theta(\mathbb{R}^N)$, $w_n\to w$ a.e.\,in $\mathbb{R}^N$ as $n\to\infty$ and there exists $g\in L^\theta(\mathbb{R}^N)$ such that $|w_n|\leq g$ a.e.\,in $\mathbb{R}^N$ for all $\theta\in[p,p^*]\cup[N,+\infty)$. It is easy to see that $w_n$ solves (in a weak sense)
	\begin{equation}\label{eq4.4999}
		[-\Delta_p w_n+V_0 w_n^{p-1}]+[-\Delta_N w_n+\eta w_n^{N-1}]\leq C_q w_n^{q-1}+\widetilde{\kappa}_\tau w_n^{\vartheta-1} \Phi(\alpha w_n^{N'})\quad\text{in }\mathbb{R}^N.
	\end{equation}
	Define the operator $T\colon \mathbf{X}\to \mathbf{X}^*$ by
	\begin{align*}
		\langle{T(u),v}\rangle=\int_{\mathbb{R}^N}\big(|\nabla u|^{p-2}\nabla u\cdot\nabla v+V_0|u|^{p-2}uv\big)\,\mathrm{d}x+\int_{\mathbb{R}^N}\big(|\nabla u|^{N-2}\nabla u\cdot\nabla v+\eta |u|^{N-2}uv\big)\,\mathrm{d}x
	\end{align*}
	for all $u,v \in X$.
	By direct computation, we can deduce that $T$ is coercive. Further, employing a slight variant argument explored by Liu-Zheng \cite[Lemmas 3.1 and 3.2]{Liu-Zheng-2011}, it follows that $T$ is monotone. Next, we prove that $T$ is hemicontinuous, that is, the function $[0,1]\ni t\mapsto \langle{T(u+tv),w}\rangle$ is continuous for all $u,v,w\in\mathbf{X}$. For this purpose, let $\{t_n\}_{n\in\mathbb{N}}\subset[0,1]$ be such that $t_n\to t$ as $n\to\infty$. By the discrete H\"older inequality, for any $\alpha\in(0,1)$ and $a,b,c,d\geq 0$, there holds
	\begin{align*}
		a^\alpha c^{1-\alpha}+b^\alpha d^{1-\alpha}\leq (a+b)^\alpha(c+d)^{1-\alpha}.
	\end{align*}
	This together with H\"older's inequality implies that for any $\psi\in\{u,v\}$ and $k>0$, we have
	\begin{equation}
		\label{eq4.44}
		\int_{\mathbb{R}^N}\big(|\nabla \psi|^{t-1}|\nabla w |+k|\psi|^{t-1}|w|\big)\,\mathrm{d}x\leq \max\{1,k\}\|\psi\|^{t-1}_{W^{1,t}}\|w\|_{W^{1,t}}<+\infty,
	\end{equation}
	where $t\in\{p,N\}$. In virtue of \eqref{eq4.44}, for any $t\in\{p,N\}$ and $k>0$, we have
	\begin{align*}
		 & \big||\nabla (u+t_nv)|^{t-2}\nabla (u+t_n v)\cdot\nabla w+k|u+t_nv|^{t-2}(u+t_nv)w\big| \\
		 & \leq C [\big(|\nabla u|^{t-1}|\nabla w |+k|u|^{t-1}|w|\big)+\big(|\nabla v|^{t-1}|\nabla w |+k|v|^{t-1}|w|\big)]\in L^1(\mathbb{R}^N)
	\end{align*}
	for some constant $C>0$. Hence, by Lebesgue's dominated convergence theorem, we ensure that
	\begin{align*}
		\langle{T(u+t_n v),w}\rangle\to \langle{T(u+tv),w}\rangle\quad \text{as }n\to\infty.
	\end{align*}
	It follows that $T$ is hemicontinuous. Consequently, by the Browder-Minty theorem (see, for example, Zeidler \cite[Theorem 26.A]{Zeidler-1990}), the operator $T$ is surjective, that is, for all $v\in \mathbf{X}^*$, there exists $u\in\mathbf{X}$ such that $T(u)=v$.

	Note that $\frac{(r-1)N}{N-1}>N$ for any $r\in\{q,\vartheta\}$. Therefore, by using H\"older's inequality, \eqref{eq4.39}, Corollary \ref{cor2.9} and Lemma \ref{lem2.11}, one has $C_q w_n^{q-1}+\widetilde{\kappa}_\tau w_n^{\vartheta-1} \Phi(\alpha w_n^{N'})\in L^{\frac{N}{N-1}}(\mathbb{R}^N)\subset \mathbf{X}^*$, thanks to the boundedness of $\{w_n\}_{n\in \mathbb{N}}$ in $\mathbf{X}$. Consequently, there exists $v_n\in \mathbf{X}$ such that it solves (in a weak sense)
	\begin{equation}\label{eq4.45}
		[-\Delta_p v_n+V_0 |v_n|^{p-2}v_n]+[-\Delta_N v_n+\eta |v_n|^{N-2}v_n]= C_q w_n^{q-1}+\widetilde{\kappa}_\tau w_n^{\vartheta-1} \Phi(\alpha w_n^{N'})\quad\text{in }\mathbb{R}^N.
	\end{equation}
	Putting $v_n=v_n^+-v_n^-$ and testing \eqref{eq4.45} by $v_n^-$,
	we obtain by using $w_n\geq 0$ that
	\begin{align*}
		&- \|v_n^-\|^p_{W^{1,p}_{V_0}}-\|v_n^- \|^N_{W^{1,N}_{\eta}}\\
		& = \int_{\mathbb{R}^N}[\big(|\nabla v_n|^{p-2}\nabla v_n\cdot\nabla v_n^-+V_0|v_n|^{p-2}v_n v_n^-\big)+\big(|\nabla v_n|^{N-2}\nabla v_n\cdot\nabla v_n^-+\eta |v_n|^{N-2}v_nv_n^-\big)]\,\mathrm{d}x  \\
		& =-\int_{\{v_n\leq 0\}}\big(C_q w_n^{q-1}+\widetilde{\kappa}_\tau w_n^{\vartheta-1} \Phi(\alpha w_n^{N'})\big)v_n\,\mathrm{d}x\geq 0.
	\end{align*}
	This shows that $ \|v_n^-\|_{W^{1,p}_{V_0}}=\|v_n^- \|_{W^{1,N}_{\eta}}=0,~\text{that is},~\|v_n^-\|_{\mathbf{X}}=0$. It follows that $v_n^-=0$ a.e.\,in $\mathbb{R}^N$, that is, $v_n\geq 0$ a.e.\,in $\mathbb{R}^N$. Inspired by the comparison principle as used in Brasco-Prinari-Zagati  \cite[Theorem 4.1]{Brasco-Prinari-Zagati-2022} and in Corr\^{e}a-Corr\^{e}a-Figueiredo \cite[Lemma 2.2]{Correa-Correa-Figueiredo-2014}, we get from \eqref{eq4.4999} and \eqref{eq4.45} that $0\leq w_n\leq v_n$ a.e.\,in $\mathbb{R}^N$. Once more, by testing \eqref{eq4.45} with $v_n$, we have
	\begin{equation}\label{eq4.46}
		\|v_n\|^p_{W^{1,p}_{V_0}}+\|v_n\|^N_{W^{1,N}_{\eta}}=\int_{\mathbb{R}^N}\big(C_q w_n^{q-1}+\widetilde{\kappa}_\tau w_n^{\vartheta-1} \Phi(\alpha w_n^{N'})\big)v_n\,\mathrm{d}x.
	\end{equation}
	Due to the Young's inequality with $\zeta\in(0,\frac{\eta}{C_q+\widetilde{\kappa}_\tau})$, that is, $ab\leq \zeta a^N+C_\zeta b^{\frac{N}{N-1}}$ for all $a,b\geq 0$, we obtain from \eqref{eq4.46}, Corollary \ref{cor2.9} and by setting $\xi=\eta-\zeta(C_q+\widetilde{\kappa}_\tau)>0$ that
	\begin{align*}
		\|v_n\|^p_{W^{1,p}_{V_0}}+\|v_n\|^N_{W^{1,N}_{\xi}}\leq C_qC_\zeta\int_{\mathbb{R}^N}|w_n|^{\frac{(q-1)N}{N-1}}\,\mathrm{d}x+\widetilde{\kappa}_\tau C_\zeta\int_{\mathbb{R}^N} w_n^{\frac{(\vartheta-1)N}{N-1}} \Phi(N'\alpha w_n^{N'})\,\mathrm{d}x\leq \tilde{C}
	\end{align*}
	for some constant $\tilde{C} >0$, thanks to H\"older's inequality, \eqref{eq4.39}, Corollary \ref{cor2.9}, Lemma \ref{lem2.11} and the boundedness of $\{w_n\}_{n\in\mathbb{N}}$ in $\mathbf{X}$. This shows that there exists a constant $C>0$ such that $\|v_n\|_{\mathbf{X}}\leq C$ for all $n\in\mathbb{N}$. It follows, up to subsequence not relabeled, that there exists $v\in \mathbf{X}$ such that we obtain from Corollary \ref{cor2.3} that
    $$ v_n\rightharpoonup v\quad\text{in }\mathbf{X}\quad\text{and}\quad v_n\rightharpoonup v \quad\text{in }L^\vartheta (\mathbb{R}^N)\quad\text{ for all } \vartheta \in [p,p^*]\cup [N,+\infty)\quad\text{as }n\to\infty.$$
	Observe that $v$ solves (in a weak sense)
	\begin{align*}
		[-\Delta_p v+V_0 |v|^{p-2}v]+[-\Delta_N v+\eta |v|^{N-2}v]= C_q w^{q-1}+\widetilde{\kappa}_\tau w^{\vartheta-1} \Phi(\alpha w^{N'})\quad\text{in }\mathbb{R}^N.
	\end{align*}
	In particular, we have
	\begin{equation}
		\label{eq4.47}  \|v\|^p_{W^{1,p}_{V_0}}+\|v\|^N_{W^{1,N}_{\eta}}=\int_{\mathbb{R}^N}\big(C_q w^{q-1}+\widetilde{\kappa}_\tau w^{\vartheta-1} \Phi(\alpha w^{N'})\big)v\,\mathrm{d}x.
	\end{equation}
	\textbf{Claim I:} $\displaystyle\int_{\mathbb{R}^N} w_n^{q-1}v_n\,\mathrm{d}x\to \int_{\mathbb{R}^N} w^{q-1}v\,\mathrm{d}x\quad\text{as }n\to\infty.$\\
	By using the fact that $w^{q-1}\in L^{\frac{N}{N-1}}(\mathbb{R}^N)$ and $v_n\rightharpoonup v $ in $L^{N}(\mathbb{R}^N)$ as $n\to\infty$, we have
	\begin{align*}
		\int_{\mathbb{R}^N} w^{q-1}(v_n-v)\,\mathrm{d}x\to 0\quad\text{as }n\to\infty.
	\end{align*}
	On the other hand, we have
	\begin{align*}
		 |w_n^{q-1}-w^{q-1}|^{N'}\leq 2^{N'-1}\Big(g^{(q-1)N'}+|w|^{(q-1)N'}\Big)\in L^1(\mathbb{R}^N).
	\end{align*}
	 In virtue of Lebesgue's dominated convergence theorem, we have
	\begin{align*}
		\int_{\mathbb{R}^N} |w_n^{q-1}-w^{q-1}|^{N'}\,\mathrm{d}x\to 0\quad\text{as }n\to\infty.
	\end{align*}
	Using this and H\"older's inequality gives
	\begin{align*}
		\bigg|\int_{\mathbb{R}^N} (w_n^{q-1}-w^{q-1})v_n\,\mathrm{d}x\bigg|\leq \|v_n\|_N \bigg(\int_{\mathbb{R}^N}|w_n^{q-1}-w^{q-1}|^{N'}\,\mathrm{d}x\bigg)^{\frac{1}{N'}}\to 0\quad\text{as }n\to\infty.
	\end{align*}
	It follows that
	\begin{align*}
		\int_{\mathbb{R}^N} (w_n^{q-1}-w^{q-1})v_n\,\mathrm{d}x\to 0\quad\text{as }n\to\infty.
	\end{align*}
	The proof of Claim I now follows directly from the above convergences.\\
	\textbf{Claim II:} $\displaystyle\int_{\mathbb{R}^N} w_n^{\vartheta-1} \Phi(\alpha w_n^{N'})v_n\to\int_{\mathbb{R}^N} w^{\vartheta-1} \Phi(\alpha w^{N'})v\,\mathrm{d}x\quad\text{as }n\to\infty.$\\
	Due to Lemma \ref{lem2.11}, one has $w^{\vartheta-1} \Phi(\alpha w^{N'})\in L^{\frac{N}{N-1}}(\mathbb{R}^N)$. Now, it follows from $v_n\rightharpoonup v $ in $L^{N}(\mathbb{R}^N)$ as $n\to\infty$ that
	\begin{align*}
		\int_{\mathbb{R}^N} (v_n-v) w^{\vartheta-1} \Phi(\alpha w^{N'})\,\mathrm{d}x\to 0\quad\text{as }n\to\infty.
	\end{align*}
	Further, by H\"older's inequality, \eqref{eq4.39}, Corollary \ref{cor2.9} and Lemma \ref{lem2.11}, it is not difficult to see that
	\begin{align*}
		\big|w_n^{\vartheta-1} \Phi(\alpha w_n^{N'})-w^{\vartheta-1} \Phi(\alpha w^{N'})\big|^{N'}\leq 2^{N'-1}\Big(g^{(\vartheta-1)N'}\Phi(N'\alpha w_n^{N'})+|w|^{(\vartheta-1)N'}\Phi(N'\alpha w^{N'})\Big)\in L^1(\mathbb{R}^N).
	\end{align*}
	Invoking Lebesgue's dominated convergence theorem, we get
	\begin{align*}
		\int_{\mathbb{R}^N} \big|w_n^{\vartheta-1} \Phi(\alpha w_n^{N'})-w^{\vartheta-1} \Phi(\alpha w^{N'})\big|^{N'}\,\mathrm{d}x\to 0\quad\text{as }n\to\infty.
	\end{align*}
	Consequently, by H\"older's inequality, we infer that for $n\to\infty$, there holds
	\begin{align*}
		&\bigg|\int_{\mathbb{R}^N} (w_n^{\vartheta-1} \Phi(\alpha w_n^{N'})-w^{\vartheta-1} \Phi(\alpha w^{N'}))v_n\,\mathrm{d}x\bigg|\\
		&\leq \|v_n\|_N \bigg( \int_{\mathbb{R}^N} \big|w_n^{\vartheta-1} \Phi(\alpha w_n^{N'})-w^{\vartheta-1} \Phi(\alpha w^{N'})\big|^{N'}\,\mathrm{d}x\bigg)^{\frac{1}{N'}}\to 0.
	\end{align*}
	This yields
	\begin{align*}
		\int_{\mathbb{R}^N} (w_n^{\vartheta-1} \Phi(\alpha w_n^{N'})-w^{\vartheta-1} \Phi(\alpha w^{N'}))v_n\,\mathrm{d}x\to 0\quad\text{as }n\to\infty.
	\end{align*}
	Now, the proof of Claim II follows immediately by using these convergences.

	In view of Claim I and Claim II, we obtain from \eqref{eq4.46} and \eqref{eq4.47} that
	\begin{align*}
		\|v_n\|^p_{W^{1,p}_{V_0}}+\|v_n\|^N_{W^{1,N}_{\eta}}=\|v\|^p_{W^{1,p}_{V_0}}+\|v\|^N_{W^{1,N}_{\eta}}+o_n(1)\quad\text{as }n\to\infty.
	\end{align*}
	This shows that
	\begin{align*}
		\|v_n\|^p_{W^{1,p}_{V_0}}\to \|v\|^p_{W^{1,p}_{V_0}}
		\quad \text{and}\quad \|v_n\|^N_{W^{1,N}_{\eta}}\to\|v\|^N_{W^{1,N}_{\eta}}
		\quad\text{as }n\to\infty.
	\end{align*}
	 Repeating the same procedure as in Lemma \ref{lem3.7}, we get $v_n\to v$ in $\mathbf{X}$ as $n\to\infty$ and hence, $v_n\to v$ in $L^\theta(\mathbb{R}^N)$ as $n\to\infty$ for all $\theta\in[p,p^*]\cup[N,+\infty)$. This shows that the first assumption of Lemma 2.1 in Ambrosio \cite{Ambrosio-2024} is satisfied.

	On the other hand, by using the fact that $C_q w_n^{q-1}+\widetilde{\kappa}_\tau w_n^{\vartheta-1} \Phi(\alpha w_n^{N'})\leq C_q v_n^{q-1}+\widetilde{\kappa}_\tau v_n^{\vartheta-1} \Phi(\alpha w_n^{N'})$ and testing \eqref{eq4.45} by $v_nv_{n,L}^{N(\beta-1)}$, we obtain by performing a similar Moser iteration as before that
	\begin{equation}
		\label{eq4.48}
		\|v_n\|_{L^\infty(\mathbb{R}^N)}\leq K\quad\text{for all}~n\in\mathbb{N} .
	\end{equation}
	This together with the interior regularity result for quasilinear elliptic equations mentioned above implies the existence of a fixed $x_0\in \mathbb{R}^N$ such that $\{v_n\}_{n\in\mathbb{N}}\subset C^{0,\alpha}(B_{\frac{1}{2}}(x_0))$ for some $\alpha\in(0,1)$ depending on $p, N$ and independent on $n\in\mathbb{N}$ and $x_0$. Further, there also holds
	\begin{equation}
		\label{eq4.49}
		[v_n]_{C^{0,\alpha}(B_{\frac{1}{2}}(x_0))}=\underset{x,y\in B_{\frac{1}{2}}(x_0),x\neq y}{\sup}\frac{|v_n(x)-v_n(y)|}{|x-y|^\alpha}\leq C,
	\end{equation}
	where the constant $C=C(p,N)>0$ is independent on $x_0$. Now, we claim that $[v_n]_{C^{0,\alpha}(\mathbb{R}^N)}\leq \tilde{C}$, where $\tilde{C}>0$ is a constant. For this, we first fix $x,y\in \mathbb{R}^N$. Note that when $|x-y|\geq 1$, then using \eqref{eq4.48}, one has
	\begin{align*}
		|v_n(x)-v_n(y)|\leq 2  \|v_n\|_{L^\infty(\mathbb{R}^N)}\leq 2K\leq 2K|x-y|^\alpha.
	\end{align*}
	Conversely, when $|x-y|<1$, then one sees that $\Big|x-\frac{x+y}{2}\Big|=\Big|y-\frac{x+y}{2}\Big|=\frac{|x-y|}{2}<\frac{1}{2}$. Hence, by using \eqref{eq4.49}, we obtain
	\begin{align*}
		|v_n(x)-v_n(y)|\leq \Bigg|v_n(x)-v_n\bigg(\frac{x+y}{2}\bigg)\Bigg|+\Bigg|v_n(y)-v_n\bigg(\frac{x+y}{2}\bigg)\Bigg|\leq \hat{C}|x-y|^\alpha
	\end{align*}
	for some constant $\hat{C}>0$. This proves the claim. Consequently, we deduce that
	\begin{align*}
		\|v_n\|_{C^{0,\alpha}(\mathbb{R}^N)}=\|v_n\|_{L^\infty(\mathbb{R}^N)}+[v_n]_{C^{0,\alpha}(\mathbb{R}^N)}\leq C_1\quad\text{for all}~n\in\mathbb{N}
	\end{align*}
	for some constant $C_1>0$. Take $\varepsilon>0$ and choose $\delta=\big(\frac{\varepsilon}{2C_1}\big)^\alpha$, then for all $x,y\in \mathbb{R}^N$ with $|x-y|<\delta$ implies that
	\begin{align*}
		|v_n(x)-v_n(y)|\leq C_1|x-y|^\alpha<\varepsilon\quad\text{for all }n\in\mathbb{N}.
	\end{align*}
	It follows that $\{v_n\}_{n\in\mathbb{N}}$ is uniformly equicontinuous in $\mathbb{R}^N$. Therefore, by applying Lemma 2.1 of Ambrosio \cite{Ambrosio-2024}, we get
	\begin{align*}
		\lim_{|x|\to\infty} \sup_{n\in\mathbb{N}}|v_n(x)|=0.
	\end{align*}
	Due to $0\leq w_n\leq v_n$ in $\mathbb{R}^N$ for all $n\in\mathbb{N}$, it follows that \eqref{eq4.35} holds. This finishes the proof.
\end{proof}

Finally, we end this section by proving the concentration phenomena of positive solutions of \eqref{main problem}.

\begin{proof}[\bf Proof of Theorem \ref{thm1.2}]
	Let $\varepsilon_0$ be small enough. Note that if $u_\varepsilon$ is a positive solution of \eqref{main problem@@}, which is obtained by Proposition \ref{prop4.6}, then there must hold
	\begin{align}
		\label{eq4.50}
		u_\varepsilon(x)<t_1 \quad\text{for all }x\in \mathbb{R}^N\setminus\Lambda_\varepsilon \text{ and }\varepsilon\in(0,\varepsilon_0).
	\end{align}
	In fact, if \eqref{eq4.50} does not hold, let  $\{\varepsilon_n\}_{n\in\mathbb{N}}$ and a solution $u_n=u_{\varepsilon_n}$ of \eqref{main problem@@} such that $\varepsilon_n\to 0$ as $n\to\infty$ and there hold $J_{\varepsilon_n}(u_n)=c_{\varepsilon_n},~J'_{\varepsilon_n}(u_n)=0$ and
	\begin{align}\label{eq4.51}
		u_n(x)\geq t_1 \quad\text{for all }x\in \mathbb{R}^N\setminus\Lambda_{\varepsilon_n}.
	\end{align}
	In view of Lemma \ref{lem4.8}, we can find a sequence $\{\widetilde{y}_n\}_{n\in\mathbb{N}}\subset\mathbb{R}^N$ such that $w_n(\cdot)=u_n(\cdot+\widetilde{y}_n)\to w$ in $\mathbf{Y}$ and $\varepsilon_n\widetilde{y}_n\to y_0$ as $n\to\infty$ for some $y_0\in\Lambda$ and $V(y_0)=V_0$. Using that $y_0\in\Lambda$, there exists some $r>0$ such that $B_r(\varepsilon_n\widetilde{y}_n)\subset\Lambda$, that is, $B_{\frac{r}{\varepsilon_n}}(\widetilde{y}_n)\subset\Lambda_{\varepsilon_n}$ for all $n $ sufficiently large. Consequently, we obtain for these values of $n$ that
	\begin{equation}\label{eq4.52}
		\mathbb{R}^N\setminus\Lambda_{\varepsilon_n}\subset B^c_{\frac{r}{\varepsilon_n}}(\widetilde{y}_n).
	\end{equation}
	Further, since $w_n\to w$ in $\mathbf{Y}$ as $n\to\infty$, by invoking Lemma \ref{lem4.10}, one sees that \eqref{eq4.35} holds. Hence, we can find $R>0$ such that
	\begin{align*}
		w_n(x)<t_1\quad\text{for all } |x|\geq R\text{ and }n\in\mathbb{N}.
	\end{align*}
	In particular, the last inequality together with a simple change of variable yields
	\begin{align*}
		u_n(x)<t_1\quad\text{for all }x\in B^c_{R}(\widetilde{y}_n)~\text{and}~n\in\mathbb{N}.
	\end{align*}
	Consequently, there exists $n_0\in\mathbb{N}$ such that for any $n\geq n_0$ and $\frac{r}{\varepsilon_n}>R$, we deduce from \eqref{eq4.52} that
	\begin{align*}
		\mathbb{R}^N\setminus\Lambda_{\varepsilon_n}\subset B^c_{\frac{r}{\varepsilon_n}}(\widetilde{y}_n)\subset B^c_{R}(\widetilde{y}_n)\quad\text{for all }n\geq n_0.
	\end{align*}
	It follows immediately that $u_n(x)<t_1$ for all $x\in\mathbb{R}^N\setminus\Lambda_{\varepsilon_n}$ and $n\geq n_0 $, which contradicts \eqref{eq4.51} and thus, \eqref{eq4.50} holds. Now, by setting  $v_\varepsilon(x)=u_\varepsilon(\frac{x}{\varepsilon})$, we can conclude that $v_\varepsilon$ is a positive solution of \eqref{main problem}.

	Finally, we study the behavior of maximum points of $v_\varepsilon(x)$ as $\varepsilon\to 0$. For this, we assume that $\varepsilon_n\to 0$ as $n\to\infty$ and $\{u_n\}_{n\in\mathbb{N}}=\{u_{\varepsilon_n}\}_{n\in\mathbb{N}}\subset\mathbf{Y}_{\varepsilon_n}$ is a nonnegative sequence of solution for \eqref{main problem@@}. In virtue of the definition of $G_2$ and $g$, we can find $\rho\in(0,t_1)$ such that
	\begin{equation}\label{eq4.53}
		G^\prime_2(\varepsilon x,s)s \leq \ell s^N
		\quad\text{and}\quad g(\varepsilon x,s)s \leq \ell' s^N\quad\text{for all }(x,s)\in \mathbb{R}^N\times [0,\rho].
	\end{equation}
	Employing a similar argument as done above, we can find $R>0$ such that
	\begin{equation}\label{eq4.54}
		\|u_n\|_{L^\infty(B^c_{R}(\widetilde{y}_n))}<t_1\quad\text{for all }n\in\mathbb{N}.
	\end{equation}
	Note that, up to a subsequence not relabeled, we can assume that
	\begin{equation} \label{eq4.55}
		\|u_n\|_{L^\infty(B_{R}(\widetilde{y}_n))}\geq t_1\quad\text{for all }n\in\mathbb{N}.
	\end{equation}
	In fact, if \eqref{eq4.55} does not hold, then we have $\|u_n\|_{L^\infty(\mathbb{R}^N)}< t_1$ for all $n\in\mathbb{N}$. Consequently, by using \eqref{eq4.53},  $J'_{\varepsilon_n}(u_n)=0$, and arguing similarly as in Lemma \ref{lem4.2}, we deduce that
	\begin{align*}
		0\leq \|u_n\|_{W^{1,p}_{V_{\varepsilon_n}}}^p+ \|u_n\|_{W^{1,N}_{V_{\varepsilon_n}}}^N+\|u_n\|^N_N+\min\big\{\|u_n\|_{\mathcal{H}_1}^l,\|u_n\|_{\mathcal{H}_1}^N\big\}\leq 0.
	\end{align*}
	Letting $n\to\infty$ in the above inequality, we ensure that $u_n\to 0$ in $\mathbf{Y}_{\varepsilon_n}$ as $n\to\infty$. Hence, we get $J_{\varepsilon_n}(u_n)=c_{\varepsilon_n}\to 0$ as $n\to\infty$, which is a contradiction because of Lemma \ref{lem4.5}(a). It follows that \eqref{eq4.55} holds.

	Taking \eqref{eq4.54} and \eqref{eq4.55} into account, we conclude that the global maximum points $p_n\in\mathbb{R}^N$ of $u_n$ belong to $B_{R}(\widetilde{y}_n)$. It follows that $p_n=r_n+\widetilde{y}_n$, where $r_n\in B_R$. Note that the solution of \eqref{main problem} is of the type $v_n(x)=u_n(\frac{x}{\varepsilon_n})$ and thus, a maximum point $\eta_{\varepsilon_n}$ of $v_n(x)$ is of the form $\eta_{\varepsilon_n}=\varepsilon_n r_n+\varepsilon_n \widetilde{y}_n$. By using $r_n\in B_R$, $\varepsilon_n \widetilde{y}_n\to y_0$ and $V(y_0)=V_0$, we deduce from the continuity of $V$ that
	\begin{align*}
		\lim_{n\to\infty}V(\eta_{\varepsilon_n})=V(y_0)=V_0.
	\end{align*}
	Hence, the proof is completed.
\end{proof}

\section{Mutiplicity of solutions to the main problem via category theory}\label{sec5}

This section is focused on the study of the multiplicity of positive solutions to \eqref{main problem} using the Lusternik-Schnirelmann category theory. This theory is a variational technique which helps us to find critical points of a functional on a manifold, in connection with the topological properties of that manifold. For more details on this theory, we refer to the papers of Benci-Cerami \cite{Benci-Cerami-1994, Benci-Cerami-1991}, Benci-Cerami-Passaseo \cite{Benci-Cerami-Passaseo-1991}, Cingolani-Lazzo \cite{Cingolani-Lazzo-2000} and the monograph of Willem \cite{Willem-1996}.

Now, we recall some basic definitions that will be needed in the sequel.

\begin{definition}
	A closed subset $A$ is contractible in a topological space $X$, if there exists a homotopy $H\in C\big([0,1]\times A, X\big)$ such that for any $u,v\in A$, there holds $H(0,u)=u$ and $H(1,u)=H(1,v)$.
\end{definition}

\begin{definition}
	Let $A$ be a closed subset of a topological space $X$. Then, the Lusternik-Schnirelmann category of $A$ in $X$ is denoted by $\operatorname{cat}_X(A)$, which is the least number of closed and contractible sets in $X$ that cover $A$.
\end{definition}

Let $\mathcal{X}$ be a Banach space and $\Psi\colon\mathcal{X}\to \mathbb{R}$ be of class $C^1(\mathcal{X},\mathbb{R})$. We define a $C^1$-manifold $\mathcal{V}$ of the form $\mathcal{V}=\Psi^{-1}(\{0\})$, where $0$ is the regular value of $\Psi$. Now, for any functional $\mathcal{I}\colon\mathcal{X}\to \mathbb{R}$, we define the following level set
\begin{align*}
	\mathcal{I}^d=\{u\in \mathcal{V}\colon \mathcal{I}(u)\leq d\}.
\end{align*}
Recall the following result for critical points involving the Lusternik-Schnirelmann category, see Theorem 5.20 by Willem \cite{Willem-1996}.

\begin{corollary}\label{cor5.3}
	Suppose $\mathcal{I}\colon\mathcal{X}\to \mathbb{R}$ is of class $C^1(\mathcal{X},\mathbb{R})$. Further, if $\mathcal{I}\big|_{\mathcal{V}}$ is bounded from below and $\mathcal{I}$ satisfies the \textnormal{(PS)$_c$} condition for $c\in[\inf \mathcal{I}\big|_{\mathcal{V}},d]$, then $\mathcal{I}\big|_{\mathcal{V}}$ has at least $\operatorname{cat}_{\mathcal{I}^d}(\mathcal{I}^d)$ critical points in $\mathcal{I}^d$.
\end{corollary}

To implement Corollary \ref{cor5.3}, the following corollary, found in Cingolani-Lazzo \cite[Lemma 2.2]{Cingolani-Lazzo-2000}, plays a significant role in relating the topology of some sublevel of a functional to the topology of some subset of the space $\mathbb{R^N}$.

\begin{corollary}\label{cor5.4}
	Let $\Omega$, $\Omega_1$ and $\Omega_2$ be closed sets with $\Omega_1 \subset \Omega_2$, and let $\beta\colon\Omega\to\Omega_2$, $\psi\colon\Omega_1 \to\Omega$ be continuous maps such that $\beta \circ \psi$ is homotopically equivalent to the embedding $j \colon \Omega_1 \to \Omega_2$. Then $\operatorname{cat}_{\Omega}(\Omega)\geq \operatorname{cat}_{\Omega_2}(\Omega_1)$.
\end{corollary}

Let $\delta>0$ be fixed and $u_0$ be a positive ground state solution of \eqref{main problem@@@}, that is, $J_0(u_0)=c_0$ and $J'_0(u_0)=0$. Next, we consider a nondecreasing function $\eta\in C^\infty([0,+\infty),[0,1])$ satisfying $\eta\equiv 1$ in $[0,\frac{\delta}{2}],~\eta\equiv 0$ in $[\delta,+\infty)$ and $|\eta'|\leq C$ for some $C>0$. For any $y\in M$, we define
\begin{align*}
	\xi_{\varepsilon,y}(x)=\eta\big(|\varepsilon x-y|\big) u_0\bigg(\frac{\varepsilon x-y}{\varepsilon}\bigg).
\end{align*}
Moreover, let $t_\varepsilon$ be the unique positive number such that $t_\varepsilon \xi_{\varepsilon,y}\in\mathcal{N}_\varepsilon$. Note that, if $|\operatorname{supp} |\xi_{\varepsilon,y}|\cap \Lambda_\varepsilon|>0$, then $t_\varepsilon$ satisfies
\begin{align*}
	J_\varepsilon(t_\varepsilon \xi_{\varepsilon,y}) =\max_{t\geq 0}J_\varepsilon(t \xi_{\varepsilon,y}).
\end{align*}
We define $\Phi_{\varepsilon}\colon M\to \mathcal{N}_\varepsilon$ by $\Phi_{\varepsilon}(y)=t_\varepsilon \xi_{\varepsilon,y}$. By the above construction, we see that $\Phi_\varepsilon$  has compact support for any $y\in M$.

Inspired by Ambrosio-Repov\v{s} \cite[Lemma 6.1]{Ambrosio-Repovs-2021} and by Thin \cite[Lemma 13]{Thin-2022}, we have the following lemma.

\begin{lemma}\label{lem5.5}
	There holds
	\begin{align*}
		\lim_{\varepsilon\to 0} J_\varepsilon(\Phi_{\varepsilon}(y))=c_0\quad\text{uniformly in }y\in M.
	\end{align*}
\end{lemma}

\begin{proof}
	Let us assume by contradiction, there exist $\delta_0>0$, $\{y_n\}_{n\in\mathbb{N}}\subset M$ and $\varepsilon_n\to 0$ as $n\to\infty$ such that
	\begin{equation}\label{eq5.1}
		\big|J_{\varepsilon_n}(\Phi_{\varepsilon_n}(y_n))-c_0\big|\geq \delta_0.
	\end{equation}
	Invoking Lebesgue's dominated convergence theorem, we have
	\begin{equation}
		\label{eq5.2}
		\|\xi_{\varepsilon_n,y_n}\|^t_{W^{1,t}_{V_{\varepsilon_n}}} \to \|u_0\|^t_{W^{1,t}_{V_0}}
		\quad\text{and}\quad
		\|\xi_{\varepsilon_n,y_n}\|^t_t\to \|u_0\|^t_t\quad\text{as }n\to\infty \text{ for } t\in\{p, N\}.
	\end{equation}
	Setting $t_n=t_{\varepsilon_n}>0$ and using that $\langle{J'_{\varepsilon_n}(t_n \xi_{\varepsilon_n,y_n}),t_n \xi_{\varepsilon_n,y_n} }\rangle=0$, we obtain by that change of variable $z=\frac{\varepsilon_n x-y_n}{\varepsilon_n}$ that
	\begin{equation}\label{eq5.3}
		\begin{aligned}
			& t_n^p  \|\xi_{\varepsilon_n,y_n}\|^p_{W^{1,p}_{V_{\varepsilon_n}}}+  t_n^N\|\xi_{\varepsilon_n,y_n}\|^N_{W^{1,N}_{V_{\varepsilon_n}+1}}\\
			&=\int_{\mathbb{R}^N}\big(G_2'(\varepsilon_n x, t_n \xi_{\varepsilon_n,y_n} )-\mathcal{H}'_1(t_n \xi_{\varepsilon_n,y_n} )+g(\varepsilon_n x, t_n \xi_{\varepsilon_n,y_n} )\big)t_n \xi_{\varepsilon_n,y_n} \,\mathrm{d}x  \\
			& = \int_{\mathbb{R}^N}\Big(G_2'(\varepsilon_n z+y_n, t_n \eta(|\varepsilon_n z|)u_0(z) )-\mathcal{H}'_1(t_n \eta(|\varepsilon_n z|)u_0(z))\\
			&\hspace{2cm}+g(\varepsilon_n z+y_n, t_n \eta(|\varepsilon_n z|)u_0(z))\Big)t_n \eta(|\varepsilon_n z|)u_0(z)\,\mathrm{d}z,
		\end{aligned}
	\end{equation}
	where we have used $\|\cdot\|^N_{W^{1,N}_{V_{\varepsilon_n}+1}}=\|\cdot\|^N_{W^{1,N}_{V_{\varepsilon_n}}}+ \|\cdot\|^N_N$. Now, we claim that $\{t_n\}_{n\in\mathbb{N}}$ is bounded. Indeed, if not, let $t_n\to\infty$ as $n\to\infty$. Take $z\in B_{\frac{\delta}{\varepsilon_n}}$, then
	$\varepsilon_n z+y_n\in B_\delta(y_n)\subset M_\delta\subset \Lambda$. By using $G_2'\equiv \mathcal{H}'_2$ and $g\equiv f$ on $\Lambda$ together with \eqref{eq2.6} and $(f4)$, we obtain from \eqref{eq5.3} that
	\begin{equation}\label{eq5.4}
		\begin{aligned}
			&\frac{1}{t_n^{N-p}}  \|\xi_{\varepsilon_n,y_n}\|^p_{W^{1,p}_{V_{\varepsilon_n}}}+ \|\xi_{\varepsilon_n,y_n}\|^N_{W^{1,N}_{V_{\varepsilon_n}+1}}\\
			& =\log(t_n^N)\int_{B_{\frac{\delta}{\varepsilon_n}}} |\eta(|\varepsilon_n z|)u_0(z)|^N\,\mathrm{d}z+ \int_{B_{\frac{\delta}{\varepsilon_n}}}\frac{f(t_n \eta(|\varepsilon_n z|)u_0(z)) \eta(|\varepsilon_n z|)u_0(z)}{t_n^{N-1}}\,\mathrm{d}z  \\
			& \qquad+\int_{B_{\frac{\delta}{\varepsilon_n}}} |\eta(|\varepsilon_n z|)u_0(z)|^N[\log \big(|\eta(|\varepsilon_n z|)u_0(z)|^N\big)+1] \,\mathrm{d}z \\
			& \geq \log(t_n^N)\int_{B_{\frac{\delta}{2\varepsilon_n}}} |u_0(z)|^N\,\mathrm{d}z+\gamma t_n^{\mu-N} \int_{B_{\frac{\delta}{2\varepsilon_n}}} |u_0(z)|^\mu\,\mathrm{d}z+D_n,
		\end{aligned}
	\end{equation}
	where
	\begin{align*}
		D_n=\int_{B_{\frac{\delta}{\varepsilon_n}}} |\eta(|\varepsilon_n z|)u_0(z)|^N[\log \big(|\eta(|\varepsilon_n z|)u_0(z)|^N\big)+1] \,\mathrm{d}z.
	\end{align*}
	Using the fact that $u_0\in \mathbf{Y}$ is a positive solution of \eqref{main problem@@@}, $\chi_{B_{\frac{\delta}{\varepsilon_n}}}\to 1$ and $\chi_{B_{\frac{\delta}{2\varepsilon_n}}}\to 1$ a.e.\,in $\mathbb{R}^N$ as $n\to\infty$, we obtain by using Lemma \ref{lem8.98} and Lebesgue's dominated convergence theorem that
	\begin{equation}\label{eq5.5}
		\int_{B_{\frac{\delta}{2\varepsilon_n}}} |u_0(z)|^r\,\mathrm{d}z\to \int_{\mathbb{R}^N} |u_0(z)|^r\,\mathrm{d}z~\text{and}~D_n\to \int_{\mathbb{R}^N} |u_0(z)|^N[\log (|u_0(z)|^N)+1] \,\mathrm{d}z~(< +\infty)
	\end{equation}
	as $n\to\infty$, where we have used $r\in\{\mu,N\}$ and $|s|^N[\log (|s|^N)+1]=\mathcal{H}'_2(s)s-\mathcal{H}'_1(s)s $. In view of \eqref{eq5.2} and \eqref{eq5.5}, we have a contradiction by letting $n\to\infty$ in \eqref{eq5.4}. This shows that $\{t_n\}_{n\in\mathbb{N}}$ is bounded and thus, up to a subsequence, we can assume that $t_n\to t_0\geq 0$ as $n\to\infty$. Now, we prove that $t_0\neq 0$. In fact, if not, suppose $t_0=0$, that is, $t_n\to 0$ as $n\to\infty$. Since $\{t_n\}_{n\in\mathbb{N}}$ is bounded, we can find $C>0$ such that $|t_n|\leq C$. Moreover, by using Corollary \ref{cor2.9} and \eqref{eq3.2}, one sees that
	\begin{align*}
		&|f(t_n \eta(|\varepsilon_n z|)u_0(z)) t_n\eta(|\varepsilon_n z|)u_0(z)|\\
		&\leq \tau t_n^N|\eta(|\varepsilon_n z|)u_0(z)|^{N}+\widetilde{\kappa}_\tau t_n^\vartheta|\eta(|\varepsilon_n z|)u_0(z)|^{\vartheta}\Phi( C^{N'} \alpha |\eta(|\varepsilon_n z|)u_0(z)|^{N'}).
	\end{align*}
	It follows immediately from the above inequality and \eqref{eq5.4} that
	\begin{align*}
		\|\xi_{\varepsilon_n,y_n}\|^p_{W^{1,p}_{V_{\varepsilon_n}}}+t_n^{N-p} \|\xi_{\varepsilon_n,y_n}\|^N_{W^{1,N}_{V_{\varepsilon_n}+1}}
		& \leq t_n^{N-p}\Bigg[\big(\log(t_n^N)+\tau\big)\int_{B_{\frac{\delta}{\varepsilon_n}}} |\eta(|\varepsilon_n z|)u_0(z)|^N\,\mathrm{d}z+D_n\Bigg]    \\
		& \qquad +\widetilde{\kappa}_\tau t_n^{\vartheta-p} \int_{B_{\frac{\delta}{\varepsilon_n}}} |\eta(|\varepsilon_n z|)u_0(z)|^{\vartheta}\Phi( C^{N'} \alpha |\eta(|\varepsilon_n z|)u_0(z)|^{N'})\,\mathrm{d}z.
	\end{align*}
	Again by Lebesgue's dominated convergence theorem together with H\"older's inequality, Corollary \ref{cor2.9} and Lemma \ref{lem2.11}, we get
	\begin{align*}
		\int_{B_{\frac{\delta}{\varepsilon_n}}} |\eta(|\varepsilon_n z|)u_0(z)|^N\,\mathrm{d}z\to \int_{\mathbb{R}^N} |u_0(z)|^N\,\mathrm{d}z
	\end{align*}
	and
	\begin{align*}
		\int_{B_{\frac{\delta}{\varepsilon_n}}} |\eta(|\varepsilon_n z|)u_0(z)|^{\vartheta}\Phi( C^{N'} \alpha |\eta(|\varepsilon_n z|)u_0(z)|^{N'})\,\mathrm{d}z\to \int_{\mathbb{R}^N} |u_0(z)|^{\vartheta}\Phi( C^{N'} \alpha |u_0(z)|^{N'})\,\mathrm{d}z<+\infty
	\end{align*}
	as $n\to\infty$. Letting $n\to\infty$ in the above inequality and using the above convergences together with \eqref{eq5.2} and \eqref{eq5.5}, we get $\|u_0\|^p_{W^{1,p}_{V_0}}=0$, that is, $u_0=0$ a.e.\,in $\mathbb{R}^N$, which is a contradiction. This shows that $t_0\neq 0$. Repeating the same arguments used in Lemma \ref{lem4.5}(c) with simple modifications and using that $\{t_n\}_{n\in\mathbb{N}}$ is bounded, we can deduce from Lebesgue's dominated convergence theorem that
	\begin{align*}
		\int_{B_{\frac{\delta}{\varepsilon_n}}} \mathcal{H}_1(t_n \eta(|\varepsilon_n z|)u_0(z))\,\mathrm{d}z&\to \int_{\mathbb{R}^N}\mathcal{H}_1(t_0 u_0(z))\,\mathrm{d}z,\\ \int_{B_{\frac{\delta}{\varepsilon_n}}} \mathcal{H}_2(t_n \eta(|\varepsilon_n z|)u_0(z))\,\mathrm{d}z&\to \int_{\mathbb{R}^N}\mathcal{H}_2(t_0 u_0(z))\,\mathrm{d}z,\\
		\int_{B_{\frac{\delta}{\varepsilon_n}}} F(t_n \eta(|\varepsilon_n z|)u_0(z))\,\mathrm{d}z&\to \int_{\mathbb{R}^N} F(t_0 u_0(z))\,\mathrm{d}z\quad\text{as }n\to\infty.
	\end{align*}
	Similarly, we also have
	\begin{align*}
		\int_{B_{\frac{\delta}{\varepsilon_n}}} \mathcal{H}'_1(t_n \eta(|\varepsilon_n z|)u_0(z))t_n \eta(|\varepsilon_n z|)u_0(z)\,\mathrm{d}z&\to \int_{\mathbb{R}^N}\mathcal{H}'_1(t_0 u_0(z))t_0 u_0(z)\,\mathrm{d}z,\\
		\int_{B_{\frac{\delta}{\varepsilon_n}}} \mathcal{H}'_2(t_n \eta(|\varepsilon_n z|)u_0(z))t_n \eta(|\varepsilon_n z|)u_0(z)\,\mathrm{d}z&\to \int_{\mathbb{R}^N}\mathcal{H}'_2(t_0 u_0(z))t_0 u_0(z)\,\mathrm{d}z,\\
		\int_{B_{\frac{\delta}{\varepsilon_n}}} f(t_n \eta(|\varepsilon_n z|)u_0(z))t_n \eta(|\varepsilon_n z|)u_0(z)\,\mathrm{d}z&\to \int_{\mathbb{R}^N}f(t_0 u_0(z))t_0 u_0(z)\,\mathrm{d}z\quad\text{as }n\to\infty.
	\end{align*}
	Combining all the above information and sending $n\to\infty$ in \eqref{eq5.3}, we get
	\begin{align*}
		\frac{1}{t_0^{N-p}}\|u_0\|^p_{W^{1,p}_{V_0}}+\|u_0\|^N_{W^{1,N}_{V_0}}+\|u_0\|^N_N+\int_{\mathbb{R}^N}\frac{\mathcal{H}'_1(t_0 u_0)u_0}{t_1^{N-1}}\,\mathrm{d}x  =\int_{\mathbb{R}^N}\frac{\mathcal{H}'_2(t_0 u_0)u_0}{t_0^{N-1}}\,\mathrm{d}x+\int_{\mathbb{R}^N}\frac{f(t_0 u_0)u_0}{t_0^{N-1}}\,\mathrm{d}x.
	\end{align*}
	But $u_0\in \mathcal{N}_0$, therefore we get
	\begin{align*}
		\|u_0\|^p_{W^{1,p}_{V_0}}+\|u_0\|^N_{W^{1,N}_{V_0}}+\|u_0\|^N_N+\int_{\mathbb{R}^N}\mathcal{H}'_1( u_0)u_0\,\mathrm{d}x  =\int_{\mathbb{R}^N}\mathcal{H}'_2(u_0)u_0\,\mathrm{d}x+\int_{\mathbb{R}^N}f( u_0)u_0\,\mathrm{d}x.
	\end{align*}
	Subtracting these equalities, we have
	\begin{align*}
		&\bigg(\frac{1}{t_0^{N-p}}-1\bigg)\|u_0\|^p_{W^{1,p}_{V_0}}+\int_{\mathbb{R}^N}\bigg(\frac{\mathcal{H}'_1(t_0 u_0)}{t_0^{N-1}}-\mathcal{H}'_1( u_0)\bigg)u_0\,\mathrm{d}x\\
		&=\int_{\mathbb{R}^N}\bigg(\frac{\mathcal{H}'_2(t_0 u_0)}{t_0^{N-1}}-\mathcal{H}'_2( u_0)\bigg)u_0\,\mathrm{d}x+\int_{\mathbb{R}^N}\bigg(\frac{f(t_0 u_0)}{t_0^{N-1}}-f( u_0)\bigg)u_0\,\mathrm{d}x.
	\end{align*}
	Employing similar arguments as done before in Lemma \ref{lem4.5}(c), from the above equation, we conclude that $t_0=1$. Consequently, we obtain from the above convergences that
	\begin{align*}
		\lim_{n\to\infty} J_{\varepsilon_n}(\Phi_{\varepsilon_n}(y_n))=J_0(u_0)=c_0,
	\end{align*}
	which contradicts \eqref{eq5.1}. This finishes the proof.
\end{proof}

We define by $\widetilde{\mathcal{N}}_\varepsilon$  a subset of $\mathcal{N}_\varepsilon$ given by
\begin{align*}
	\widetilde{\mathcal{N}}_\varepsilon=\{u\in \mathcal{N}_\varepsilon\colon J_\varepsilon(u)\leq c_0+h(\varepsilon)\},
\end{align*}
where $h\colon\mathbb{R}^+\to \mathbb{R}^+$ is defined by $h(\varepsilon)=\sup_{y\in M}|J_\varepsilon(u)- c_0|$. From the above lemma, we have $h(\varepsilon)\to 0$ as $\varepsilon\to 0$. Moreover, we know that $\Phi_\varepsilon(y)\in \widetilde{\mathcal{N}}_\varepsilon$ for all $y\in M$ and $\varepsilon>0$. It follows that $\widetilde{\mathcal{N}}_\varepsilon\neq \emptyset$. Let for any $\delta>0$ with $\rho=\rho(\delta)>0$ be such that $M_\delta\subset B_\rho$. Let $\chi\colon\mathbb{R}^N\to \mathbb{R}^N$ be defined by
\begin{align*}
	\chi(x)=
	\begin{cases}
		x &\text{if }|x|<\rho, \\[1ex]
		\dfrac{\rho x}{|x|} &\text{if }|x|\geq\rho.
	\end{cases}
\end{align*}
Now, we define the barycenter map $\beta_\varepsilon\colon\mathcal{N}_\varepsilon\to \mathbb{R}^N $ by
\begin{align*}
	\beta_\varepsilon(u)=\displaystyle\frac{\displaystyle\int_{\mathbb{R}^N}\chi(\varepsilon x)(|u|^p+|u|^N)\,\mathrm{d}x}{\displaystyle\int_{\mathbb{R}^N}(|u|^p+|u|^N)\,\mathrm{d}x}.
\end{align*}

\begin{lemma}
	There holds
	\begin{align*}
		\lim_{\varepsilon\to 0}\beta_\varepsilon(\Phi_\varepsilon(y))=y\quad\text{uniformly in }y\in M.
	\end{align*}
\end{lemma}

\begin{proof}
	Let us assume by contradiction there exist $\delta_0>0$, $\{y_n\}_{n\in\mathbb{N}}\subset M$ and $\varepsilon_n\to 0$ as $n\to\infty$ such that
	\begin{align}\label{eq5.6}
		\big|\beta_{\varepsilon_n}(\Phi_{\varepsilon_n}(y_n))-y_n\big|\geq \delta_0.
	\end{align}
	Note that for $z\in B_{\frac{\delta}{\varepsilon_n}}$, one has
	$\varepsilon_n z+y_n\in B_\delta(y_n)\subset M_\delta\subset B_\rho.$
	Therefore, from the definitions of $\Phi_{\varepsilon_n}$, $\beta_{\varepsilon_n}$  and $\eta$ as well as the change of variable $z=\frac{\varepsilon_n x-y_n}{\varepsilon_n}$, we have
	\begin{align*}
		\beta_{\varepsilon_n}(\Phi_{\varepsilon_n}(y_n))=y_n +\displaystyle\frac{\displaystyle\int_{\mathbb{R}^N}\varepsilon_n z \big(|\eta(|\varepsilon_n z|)u_0(z)|^p+|\eta(|\varepsilon_n z|)u_0(z)|^N\big)\,\mathrm{d}z}{\displaystyle\int_{\mathbb{R}^N}\big(|\eta(|\varepsilon_n z|)u_0(z)|^p+|\eta(|\varepsilon_n z|)u_0(z)|^N\big)\,\mathrm{d}z}.
	\end{align*}
	This together with Lebesgue's dominated convergence theorem implies that
	\begin{align*}
		\big|\beta_{\varepsilon_n}(\Phi_{\varepsilon_n}(y_n))-y_n\big|=o_n(1)\quad\text{as }n\to\infty,
	\end{align*}
	which contradicts \eqref{eq5.6}. This completes the proof.
\end{proof}

The following compactness principle is crucial in order to find multiplicity of solutions for our problem.

\begin{lemma}\label{lem5.7}
	Let $\varepsilon_n\to 0$ as $n\to\infty$ and $\{u_n\}_{n\in\mathbb{N}}\subset\mathcal{N}_{\varepsilon_n}$ be a sequence such that $J_{\varepsilon_n}(u_n)\to c_0$ as $n\to\infty$ and satisfying \eqref{eq3.999}. Then, there exists a sequence $\{\widetilde{y}_n\}_{n\in\mathbb{N}}\subset\mathbb{R}^N$ such that  $w_n(x)=u_n(x+\widetilde{y}_n)$ has a convergent subsequence in $\mathbf{Y}$. Moreover, there holds $y_n=\varepsilon_n\widetilde{y}_n\to y_0$ as $n\to\infty$ for some $y_0\in M$.
\end{lemma}

\begin{proof}
	By applying the same idea as in Alves-da Silva \cite[Lemma 6.5]{Alves-daSilva-2023}, the lemma can be proved.
\end{proof}

Now, we shall discuss the multiplicity of solutions to \eqref{main problem@@} by using the Lusternik-Schnirelmann category theory.

\begin{proposition}\label{prop5.8}
	Assume that \textnormal{(V1)--(V2)} hold and let $\delta>0$ be small enough. Then, the problem \eqref{main problem@@} has at least $\operatorname{cat}_{M_\delta} (M)$ solutions for $\varepsilon$ small enough such that $\varepsilon\in(0,\tilde{\varepsilon}_1)$ for some $\tilde{\varepsilon}_1>0$.
\end{proposition}

\begin{proof}
	Choose $\tilde{\varepsilon}_1>0$ be sufficiently small and fix $\varepsilon\in(0,\tilde{\varepsilon}_1)$. Let us set $\mathcal{I}=J_\varepsilon$, $\mathcal{V}=\mathcal{N}_\varepsilon$, $d=c_0+h(\varepsilon)$ and $\mathcal{I}^d=J^d_\varepsilon=\tilde{\mathcal{N}}_\varepsilon$. In view of Proposition \ref{prop4.4}, one sees that $J_\varepsilon\big|_{\mathcal{N}_\varepsilon}$ satisfies the \textnormal{(PS)} condition and thus, by Corollary \ref{cor5.3}, it implies that $J_\varepsilon\big|_{\mathcal{N}_\varepsilon}$ has at least  $\operatorname{cat}_{\widetilde{\mathcal{N}}_\varepsilon} (\widetilde{\mathcal{N}}_\varepsilon)$ critical points in $\widetilde{\mathcal{N}}_\varepsilon$. Consequently, by Proposition \ref{prop4.3}, we deduce that $J_\varepsilon$ has at least $\operatorname{cat}_{\widetilde{\mathcal{N}}_\varepsilon} (\widetilde{\mathcal{N}}_\varepsilon)$ critical points. This shows that \eqref{main problem@@} has at least $\operatorname{cat}_{\widetilde{\mathcal{N}}_\varepsilon} (\widetilde{\mathcal{N}}_\varepsilon)$ critical points. To complete the proof, we only have to show that $\operatorname{cat}_{\widetilde{\mathcal{N}}_\varepsilon} (\widetilde{\mathcal{N}}_\varepsilon)\geq \operatorname{cat}_{M_\delta} (M)$. One can notice that $\Phi_\varepsilon(M)\subset \widetilde{\mathcal{N}}_\varepsilon$ for $\varepsilon$ small enough, thanks to Lemma \ref{lem5.5}. Moreover, the following diagram
	\begin{align*}
		M\overset{\Phi_\varepsilon}{ \longrightarrow}\widetilde{\mathcal{N}}_\varepsilon \overset{\beta_\varepsilon}{ \longrightarrow} M_\delta
	\end{align*}
	is well-defined for $\varepsilon$ small enough. It follows that the map $\beta_\varepsilon \circ \Phi_\varepsilon\colon M\to M_\delta$ is well-defined for $\varepsilon$ small enough. Define the map $H\colon[0,1]\times M\to M_\delta$ by $H(t,y)=(1-t)y+t\beta_\varepsilon(\Phi_\varepsilon(y))$ for all $(t,y)\in [0,1]\times M $. This shows that $H(0,y)=y$ and $H(1,y)=\beta_\varepsilon(\Phi_\varepsilon(y))$. Therefore, we infer that the map $\beta_\varepsilon \circ\Phi_\varepsilon$ is homotopically equivalent to the embedding $j\colon M\to M_\delta$ and thus, by Corollary \ref{cor5.4} implies that $\operatorname{cat}_{\widetilde{\mathcal{N}}_\varepsilon} (\widetilde{\mathcal{N}}_\varepsilon)\geq \operatorname{cat}_{M_\delta} (M)$.
\end{proof}

\begin{proposition}\label{prop5.9}
	Let $V$ satisfies \textnormal{(V1)--(V2)}  and let $\delta>0$ sufficiently small, then there exists $\tilde{\varepsilon}_2>0$ such that for $\varepsilon\in(0,\tilde{\varepsilon}_2)$, we have the following assertions:
	\begin{enumerate}
		\item[\textnormal{(a)}]
			 problem \eqref{main problem@@} has at least $\frac{\operatorname{cat}_{M_\delta} (M)}{2}$ positive solutions, whenever $\operatorname{cat}_{M_\delta} (M)$ is an even number;
		\item[\textnormal{(b)}]
			 problem \eqref{main problem@@} has at least $\frac{\operatorname{cat}_{M_\delta} (M)+1}{2}$ positive solutions, whenever $\operatorname{cat}_{M_\delta} (M)$ is an odd number.
	\end{enumerate}
\end{proposition}

\begin{proof}
	Let $\tilde{\varepsilon}_2>0$ be small enough and fix $\varepsilon\in(0,\tilde{\varepsilon}_2)$. Further, assume that $w_\varepsilon$ is a critical point of $J_\varepsilon$ and $J_\varepsilon(w_\varepsilon)\leq c_0+h(\varepsilon)$ holds. Then either $w^+_\varepsilon=0$ or $w^-_\varepsilon=0$. Indeed, if not, then repeating a similar procedure as in Proposition \ref{prop4.6}, one has $w^+_\varepsilon,~w^-_\varepsilon\in\mathcal{N}_\varepsilon$. Consequently, we deduce that
	\begin{align*}
		c_0+h(\varepsilon)\geq J_\varepsilon(w_\varepsilon)=J_\varepsilon(w^+_\varepsilon)+J_\varepsilon(w^-_\varepsilon)\geq 2c_\varepsilon.
	\end{align*}
	Letting $\varepsilon\to 0$ in the above inequality and using Remark \ref{rem4.9}, we get a contradiction. Once more by similar arguments as in Proposition \ref{prop4.6}, we infer that either $w_\varepsilon>0$ or $w_\varepsilon<0$. Denote $\varrho=\operatorname{cat}_{M_\delta}(M)$. Let $\varrho$ be an even number and $w_1,\cdots,w_\varrho$ be the solutions of \eqref{main problem@@} as in Proposition \ref{prop5.8}. If at least $\frac{\varrho}{2}$ of the solutions $w_1,\cdots,w_\varrho$ are positive, then (a) is proved. In fact, if not, suppose that at least $\frac{\varrho}{2}$ of the solutions $w_1,\cdots,w_\varrho$ are negative and denote these negative solutions by $v_1,\cdots,v_{\frac{\varrho}{2}}$. Observe that $G_2'(x,\cdot)+g(x,\cdot)-\mathcal{H}'_1(\cdot)$ is an odd function, therefore $-v_1,\cdots,-v_{\frac{\varrho}{2}}$ are positive solutions of \eqref{main problem@@}. The assertion in (a) follows. In a similar way, one can prove the statement (b).
\end{proof}

Now, we shall be able to prove Theorem \ref{thm1.3}.

\begin{proof}[\bf Proof of Theorem \ref{thm1.3}]
	Suppose $w_\varepsilon$ is a critical point of $J_\varepsilon$ satisfying  $J_\varepsilon(w_\varepsilon)\leq c_0+h(\varepsilon)$. In order to complete the proof, it is sufficient to show that there exists $\varepsilon_1>0$ small enough with $\varepsilon\in(0,\varepsilon_1)$ such that
	\begin{equation}\label{eq5.7}
		0<w_\varepsilon(x)<t_1\quad\text{for all }x\in\mathbb{R}^N\setminus \Lambda_\varepsilon
	\end{equation}
	holds, where each solution $w_\varepsilon$ of \eqref{main problem@@} is given in (a) and (b) of Proposition \ref{prop5.9}. Let us assume by contradiction that there exists a sequence $\{\varepsilon_n\}_{n\in\mathbb{N}}$ such that $\varepsilon_n\to 0$ as $n\to\infty$, $\{w_n\}_{n\in\mathbb{N}}=\{w_{\varepsilon_n}\}_{n\in\mathbb{N}}$ be a sequence of solution for $(\widetilde{\mathcal{S}}_{\varepsilon_n})$ and \eqref{eq3.999} holds but \eqref{eq5.7} is not satisfied. It is easy to see that $\{w_n\}_{n\in\mathbb{N}}$ satisfies the assumptions of Lemma \ref{lem5.7} and thus, \eqref{eq4.35} holds. Now, using a similar procedure as in the proof of Theorem \ref{thm1.2}, we get a contradiction. It follows that \eqref{eq5.7} holds true. Consequently, we deduce that \eqref{main problem@} satisfies (a) and (b) of Theorem \ref{thm1.3}. Hence, the result follows immediately by using a simple change of variable. This completes the proof.
\end{proof}

\noindent{\bf Acknowledgements:}

Deepak Kumar Mahanta wishes to convey his sincere appreciation for the DST INSPIRE Fellowship with reference number DST/INSPIRE/03/2019/000265 funded by the Government of India. Tuhina Mukherjee thanks the CSIR-HRDG grant; sanction number 25/0324/23/EMR-II.

%
%
%
%
%


\begin{thebibliography}{99}

\bibitem{Adimurthi-Yang-2010}
	Adimurthi, Y. Yang,
	{\it An interpolation of Hardy inequality and Trudinger-Moser inequality in $\mathbb{R}^N$ and its applications},
	Int. Math. Res. Not. IMRN {\bf 2010} (2010), no. 13, 2394--2426.

\bibitem{Alves-Ambrosio-2024}
	C.O. Alves, V. Ambrosio,
	{\it Concentrating solutions for a fractional $p$-Laplacian logarithmic Schr\"{o}dinger equation},
	Anal. Appl. (Singap.) {\bf 22} (2024), no. 2, 311--349.

\bibitem{Alves-daSilva-2024}
	C.O. Alves, I.S. da Silva,
	{\it Existence of a positive solution for a class of Schr\"{o}dinger logarithmic equations on exterior domains},
	Z. Angew. Math. Phys. {\bf 75} (2024), no. 3, Paper No. 77.

\bibitem{Alves-daSilva-2023}
	C.O. Alves, I.S. da Silva,
	{\it Existence of multiple solutions for a Schr\"{o}dinger logarithmic equation via Lusternik-Schnirelmann category},
	Anal. Appl. (Singap.) {\bf 21} (2023), no. 6, 1477--1516.

\bibitem{Alves-deMoraisFilho-2018}
	C.O. Alves, D.C. de Morais Filho,
	{\it Existence and concentration of positive solutions for a Schr\"{o}dinger logarithmic equation},
	Z. Angew. Math. Phys. {\bf 69} (2018), no. 6, Paper No. 144, 22 pp.

\bibitem{Alves-Figueiredo-2009}
	C.O. Alves, G.M. Figueiredo,
	{\it On multiplicity and concentration of positive solutions for a class of quasilinear problems with critical exponential growth in $\mathbb{R}^N$},
	J. Differential Equations {\bf 246} (2009), no. 3, 1288--1311.

\bibitem{Alves-Ji-2020}
	C.O. Alves, C. Ji,
	{\it Existence and concentration of positive solutions for a logarithmic Schr\"{o}dinger equation via penalization method},
	Calc. Var. Partial Differential Equations {\bf 59} (2020), no. 1, Paper No. 21, 27 pp.

\bibitem{Alves-Ji-2024}
	C.O. Alves, C. Ji,
	{\it Multi-peak positive solutions for a logarithmic Schr\"{o}dinger equation via variational methods},
	Israel J. Math. {\bf 259} (2024), no. 2, 835--885.

\bibitem{Ambrosio-2024}
	V. Ambrosio,
	{\it On the uniform vanishing property at infinity of $W^{s,p}$-sequences},
	Nonlinear Anal. {\bf 238} (2024), Paper No. 113398, 17 pp.

\bibitem{Ambrosio-Repovs-2021}
	V. Ambrosio, D.D. Repov\v{s},
	{\it Multiplicity and concentration results for a $(p,q)$-Laplacian problem in $\mathbb{R}^N$},
	Z. Angew. Math. Phys. {\bf 72} (2021), no. 1, Paper No. 33, 33 pp.

\bibitem{Aris-1994}
	R. Aris,
	``Mathematical Modelling Techniques'',
	Dover Publications, Inc., New York, 1994.

\bibitem{Autuori-Pucci-2013}
	G. Autuori, P. Pucci,
	{\it Existence of entire solutions for a class of quasilinear elliptic equations},
	NoDEA Nonlinear Differential Equations Appl. {\bf 20} (2013), no. 3, 977--1009.

\bibitem{Bartolo-Candela-Salvatore-2016}
	R. Bartolo, A.M. Candela, A. Salvatore,
	{\it Multiplicity results for a class of asymptotically $p$-linear equations on $\mathbb{R}^N$},
	Commun. Contemp. Math. {\bf 18} (2016), no. 1, 1550031, 24 pp.

\bibitem{Bartsch-Wang-1995}
	T. Bartsch, Z.Q. Wang,
	{\it Existence and multiplicity results for some superlinear elliptic problems on $\mathbb{R}^N$},
	Comm. Partial Differential Equations {\bf 20} (1995), no. 9-10, 1725--1741.

\bibitem{Beckner-1993}
	W. Beckner,
	{\it Sharp Sobolev inequalities on the sphere and the Moser-Trudinger inequality},
	Ann. of Math. (2) {\bf 138} (1993), no. 1, 213--242.

\bibitem{Benci-Cerami-1994}
	V. Benci, G. Cerami,
	{\it Multiple positive solutions of some elliptic problems via the Morse theory and the domain topology},
	Calc. Var. Partial Differential Equations {\bf 2} (1994), no. 1, 29--48.

\bibitem{Benci-Cerami-1991}
	V. Benci, G. Cerami,
	{\it The effect of the domain topology on the number of positive solutions of nonlinear elliptic problems},
	Arch. Rational Mech. Anal. {\bf 114} (1991), no. 1, 79--93.

\bibitem{Benci-Cerami-Passaseo-1991}
	V. Benci, G. Cerami, D. Passaseo,
	{\it On the number of the positive solutions of some nonlinear elliptic problems},
	Nonlinear analysi, 93--107, Sc. Norm. Super. di Pisa Quaderni,  Scuola Norm. Sup., Pisa, 1991.

\bibitem{Bialynicki-Birula-Mycielski-1975}
	I. Bia\l ynicki-Birula, J. Mycielski,
	{\it Wave equations with logarithmic nonlinearities},
	Bull. Acad. Polon. Sci. S\'{e}r. Sci. Math. Astronom. Phys. {\bf 23} (1975), no. 4, 461--466.

\bibitem{Brasco-Prinari-Zagati-2022}
	L. Brasco, F. Prinari, A.C. Zagati,
	{\it A comparison principle for the Lane-Emden equation and applications to geometric estimates},
	Nonlinear Anal. {\bf 220} (2022),  Paper No. 112847, 41 pp.

\bibitem{Carles-Gallagher-2018}
	R. Carles, I. Gallagher,
	{\it Universal dynamics for the defocusing logarithmic Schr\"{o}dinger equation},
	Duke Math. J. {\bf 167} (2018), no. 9, 1761--1801.

\bibitem{Carvalho-Figueiredo-Furtado-Medeiros-2021}
	J.L. Carvalho, G.M. Figueiredo, M.F. Furtado, E. Medeiros,
	{\it On a zero-mass $(N,q)$-Laplacian equation in $\mathbb{R}^N$ with exponential critical growth},
	Nonlinear Anal. {\bf 213} (2021), Paper No. 112488, 14 pp.

\bibitem{Cazenave-1983}
	T. Cazenave,
	{\it Stable solutions of the logarithmic Schr\"{o}dinger equation},
	Nonlinear Anal. {\bf 7} (1983), no. 10, 1127--1140.

\bibitem{Cazenave-Lions-1982}
	T. Cazenave, P.-L. Lions,
	{\it Orbital stability of standing waves for some nonlinear Schr\"{o}dinger equations},
	Comm. Math. Phys. {\bf 85} (1982), no. 4, 549--561.

\bibitem{Chang-Yang-2003}
	S.-Y.A. Chang, P.C. Yang,
	{\it The inequality of Moser and Trudinger and applications to conformal geometry},
	Comm. Pure Appl. Math. {\bf 56} (2003), no. 8, 1135--1150.

\bibitem{Chen-Chen-2016}
	C. Chen, Q. Chen,
	{\it Infinitely many solutions for $p$-Kirchhoff equation with concave-convex nonlinearities in $\mathbb{R}^N$},
	Math. Methods Appl. Sci. {\bf 39} (2016), no. 6, 1493--1504.

\bibitem{Chen-Fiscella-Pucci-Tang-2020}
	S. Chen, A. Fiscella, P. Pucci, X. Tang,
	{\it Coupled elliptic systems in $\mathbb{R}^N$ with $(p,N)$ Laplacian and critical exponential nonlinearities},
	Nonlinear Anal. {\bf 201} (2020),  112066, 14 pp.

\bibitem{Chen-Lu-Zhu-2023-b}
	L. Chen, G. Lu, M. Zhu,
	{\it Existence and non-existence of ground states of bi-harmonic equations involving constant and degenerate Rabinowitz potentials},
	Calc. Var. Partial Differential Equations {\bf 62} (2023), no. 2, Paper No. 37, 29 pp.

\bibitem{Chen-Lu-Zhu-2023}
	L. Chen, G. Lu, M. Zhu,
	{\it Existence of extremals for Trudinger-Moser inequalities involved with a trapping potential},
	Calc. Var. Partial Differential Equations {\bf 62} (2023), no. 5, Paper No. 150, 35 pp.

\bibitem{Chen-Lu-Zhu-2020}
	L. Chen, G. Lu, M. Zhu,
	{\it Ground states of bi-harmonic equations with critical exponential growth involving constant and trapping potentials},
	Calc. Var. Partial Differential Equations {\bf 59} (2020), no. 6, Paper No. 185, 38 pp.

\bibitem{Chen-Lu-Zhu-2023-c}
	L. Chen, G. Lu, M. Zhu,
	{\it Least energy solutions to quasilinear subelliptic equations with constant and degenerate potentials on the Heisenberg group},
	Proc. Lond. Math. Soc. (3) {\bf 126} (2023), no. 2, 518--555.

\bibitem{Chen-Lu-Zhu-2021}
	L. Chen, G. Lu, M. Zhu,
	{\it Sharp Trudinger-Moser inequality and ground state solutions to quasi-linear Schr\"{o}dinger equations with degenerate potentials in $\mathbb{R}^n$},
	Adv. Nonlinear Stud. {\bf 21} (2021), no. 4, 733--749.

\bibitem{Chen-Wang-Zhu-2023}
	L. Chen, B. Wang, M. Zhu,
	{\it Improved fractional Trudinger-Moser inequalities on bounded intervals and the existence of their extremals},
	Adv. Nonlinear Stud. {\bf 23} (2023), no. 1, Paper No. 20220067, 17 pp.

\bibitem{Cingolani-Lazzo-2000}
	S. Cingolani, M. Lazzo,
	{\it Multiple positive solutions to nonlinear Schr\"{o}dinger equations with competing potential functions},
	J. Differential Equations {\bf 160} (2000), no. 1, 118--138.

\bibitem{Cohn-Lu-2001}
	W.S. Cohn, G. Lu,
	{\it Best constants for Moser-Trudinger inequalities on the Heisenberg group},
	Indiana Univ. Math. J. {\bf 50} (2001), no. 4, 1567--1591.

\bibitem{Correa-Correa-Figueiredo-2014}
	F.J.S.A. Corr\^{e}a, A.S.S. Corr\^{e}a,G.M. Figueiredo,
	{\it Positive solution for a class of $p\&q$-singular elliptic equation},
	Nonlinear Anal. Real World Appl. {\bf 16} (2014), 163--169.

\bibitem{dAvenia-Montefusco-Squassina-2014}
	P. d'Avenia, E. Montefusco, M. Squassina,
	{\it On the logarithmic Schr\"{o}dinger equation},
	Commun. Contemp. Math. {\bf 16} (2014), no. 2, 1350032, 15 pp.

\bibitem{Deng-He-Pan-Zhong-2023}
	Y. Deng, Q. He, Y. Pan, X. Zhong,
	{\it The existence of positive solution for an elliptic problem with critical growth and logarithmic perturbation},
	Adv. Nonlinear Stud. {\bf 23} (2023), no. 1, Paper No. 20220049, 22 pp.

\bibitem{DelPino-Dolbeault-2003}
	M. Del Pino, J. Dolbeault,
	{\it The optimal Euclidean $L^p$-Sobolev logarithmic inequality},
	J. Funct. Anal. {\bf 197} (2003), no. 1, 151--161.

\bibitem{DiBenedetto-1983}
	E. DiBenedetto,
	{\it $C^{1+\alpha }$ local regularity of weak solutions of degenerate elliptic equations},
	Nonlinear Anal. {\bf 7} (1983), no. 8, 827--850.

\bibitem{do-Lu-Ponciano-2024}
	J.M. do \'{O}, G. Lu, R. Ponciano,
	{\it Sharp Sobolev and Adams-Trudinger-Moser embeddings on weighted Sobolev spaces and their applications},
	Forum Math. {\bf 36} (2024), no. 5, 1279--1320.

\bibitem{Duy-Phi-2022}
	N.T. Duy, L.L. Phi,
	{\it Finsler Trudinger-Moser inequalities on $\mathbb{R}^2$},
	Sci. China Math. {\bf 65} (2022), no. 9, 1803--1826.

\bibitem{Fife-1979}
	P.C. Fife,
	``Mathematical Aspects of Reacting and Diffusing Systems'',
	Springer-Verlag, Berlin-New York, 1979.

\bibitem{Fiscella-Pucci-2021}
	A. Fiscella, P. Pucci,
	{\it $(p,N)$ equations with critical exponential nonlinearities in $\mathbb{R}^N$},
	J. Math. Anal. Appl. {\bf 501} (2021), no. 1, Paper No. 123379, 25 pp.

\bibitem{He-Li-2008}
	C. He, G. Li,
	{\it The regularity of weak solutions to nonlinear scalar field elliptic equations containing $p\&q$-Laplacians},
	Ann. Acad. Sci. Fenn. Math. {\bf 33} (2008), no. 2, 337--371.

\bibitem{Ji-Szulkin-2016}
	C. Ji, A. Szulkin,
	{\it A logarithmic Schr\"{o}dinger equation with asymptotic conditions on the potential},
	J. Math. Anal. Appl. {\bf 437} (2016), no. 1, 241--254.

\bibitem{Jiang-Xu-Zhang-Zhu-2025}
	R. Jiang, W. Xu, C. Zhang, M. Zhu,
	{\it Adams-type inequalities with logarithmic weights in fractional dimensions and the existence of extremals},
	Bull. Sci. Math. {\bf 200} (2025), Paper No. 103586, 27 pp.

\bibitem{Ladyzhenskaya-Uraltseva-1968}
	O.A. Ladyzhenskaya, N.N. Ural'tseva,
	``Linear and Quasilinear Elliptic Equations'',
	Academic Press, New York-London, 1968.

\bibitem{Lam-Lu-2013}
	N. Lam, G. Lu,
	{\it A new approach to sharp Moser-Trudinger and Adams type inequalities: a rearrangement-free argument},
	J. Differential Equations {\bf 255} (2013), no. 3, 298--325.

\bibitem{Lam-Lu-2012}
	N. Lam, G. Lu,
	{\it Existence and multiplicity of solutions to equations of $N$-Laplacian type with critical exponential growth in $\mathbb{R}^N$},
	J. Funct. Anal. {\bf 262} (2012), no. 3, 1132--1165.

\bibitem{Lam-Lu-2012-b}
	N. Lam, G. Lu,
	{\it Sharp Moser-Trudinger inequality on the Heisenberg group at the critical case and applications},
	Adv. Math. {\bf 231} (2012), no. 6, 3259--3287.

\bibitem{Li-1990}
	G.B. Li,
	{\it Some properties of weak solutions of nonlinear scalar field equations},
	Ann. Acad. Sci. Fenn. Ser. A I Math. {\bf 15} (1990), no. 1, 27--36.

\bibitem{Li-Lu-Zhu-2021}
	J. Li, G. Lu, M. Zhu,
	{\it Concentration-compactness principle for Trudinger-Moser's inequalities on Riemannian manifolds and Heisenberg groups: a completely symmetrization-free argument},
	Adv. Nonlinear Stud. {\bf 21} (2021), no. 4, 917--937.

\bibitem{Li-Peng-Shuai-2022}
	Q. Li, S. Peng, W. Shuai,
	{\it On fractional logarithmic Schr\"{o}dinger equations},
	Adv. Nonlinear Stud. {\bf 22} (2022), no. 1, 41--66.

\bibitem{Li-Ruf-2008}
	Y. Li, B. Ruf,
	{\it A sharp Trudinger-Moser type inequality for unbounded domains in $\mathbb{R}^n$},
	Indiana Univ. Math. J. {\bf 57} (2008), no. 1, 451--480.

\bibitem{Liu-Peng-Zou-2025}
	T. Liu, X. Peng, W. Zou,
	{\it Normalized solution for the logarithmic Schr\"{o}dinger system},
	J. Geom. Anal. {\bf 35} (2025), no. 7, Paper No. 215, 58 pp.

\bibitem{Liu-Zheng-2011}
	C. Liu, Y. Zheng,
	{\it Existence of nontrivial solutions for $p$-Laplacian equations in $\mathbb{R}^N$},
	J. Math. Anal. Appl. {\bf 380} (2011), no. 2, 669--679.

\bibitem{Mahanta-Mukherjee-Sarkar-2025}
	D.K. Mahanta, T. Mukherjee, A. Sarkar,
	{\it Degenerate Schr\"{o}dinger-Kirchhoff $(p,N)$-Laplacian problem with singular Trudinger-Moser nonlinearity in $\mathbb{R}^N$},
	Forum Math. {\bf 37} (2025), no. 2, 459--483.

\bibitem{Mahanta-Winkert-2026}
	D.K. Mahanta, P. Winkert,
	{\it On $(p,n)$-Laplace Schr\"{o}dinger equations with Stein-Weiss convolution parts in $\mathbb{R}^n$},
	Appl. Math. Lett. {\bf 172} (2026), Paper No. 109700, 7 pp.

\bibitem{Moser-1960}
	J. Moser,
	{\it A new proof of De Giorgi's theorem concerning the regularity problem for elliptic differential equations},
	Comm. Pure Appl. Math. {\bf 13} (1960), 457--468.

\bibitem{Murray-1993}
	J.D. Murray,
	``Mathematical Biology'',
	Springer-Verlag, Berlin, 1993.

\bibitem{Myers-Beaghton-Vvedensky-1989}
	A.K. Myers-Beaghton, D.D. Vvedensky,
	{\it Chapman-Kolmogorov equation for Markov models of epitaxial growth},
	J. Phys. A {\bf 22} (1989), no. 11, L467--L475.

\bibitem{Pucci-Serrin-1985}
	P. Pucci, J. Serrin,
	{\it A mountain pass theorem},
	J. Differential Equations {\bf 60} (1985), no. 1, 142--149.

\bibitem{Pucci-Xiang-Zhang-2015}
	P. Pucci, M. Xiang, B. Zhang,
	{\it Multiple solutions for nonhomogeneous Schr\"{o}dinger-Kirchhoff type equations involving the fractional $p$-Laplacian in $\mathbb{R}^N$},
	Calc. Var. Partial Differential Equations {\bf 54} (2015), no. 3, 2785--2806.

\bibitem{Royden-1988}
	H.L. Royden,
	``Real Analysis'',
	Macmillan Publishing Company, New York, 1988.

\bibitem{Serrin-1964}
	J. Serrin,
	{\it Local behavior of solutions of quasi-linear equations},
	Acta Math {\bf 111} (1964), 247--302.

\bibitem{Shen-Squassina-2025}
	L. Shen, M. Squassina,
	{\it Existence and concentration of normalized solutions for $p$-Laplacian equations with logarithmic nonlinearity},
	J. Differential Equations {\bf 421} (2025), 1--49.

\bibitem{Simon-1977}
	J. Simon,
	{\it R\'{e}gularit\'{e} de la solution d'une \'{e}quation non lin\'{e}aire dans $\mathbb{R}^N$},
	Journ\'{e}es d'Analyse Non Lin\'{e}aire (Proc. Conf., Besan\c{c}on, 1977), pp. 205--227, Lecture Notes in Math. {\bf 665}, Springer, Berlin, 1978.

\bibitem{Squassina-Szulkin-2015}
	M. Squassina, A. Szulkin,
	{\it Multiple solutions to logarithmic Schr\"{o}dinger equations with periodic potential},
	Calc. Var. Partial Differential Equations {\bf 54} (2015), no. 1, 585--597.

\bibitem{Squassina-Szulkin-2017}
	M. Squassina, A. Szulkin,
	{\it Erratum to: {M}ultiple solutions to logarithmic Schr\"{o}dinger equations with periodic potential [{MR}3385171]},
	Calc. Var. Partial Differential Equations {\bf 56} (2017), no. 3, Paper No. 56, 4 pp.

\bibitem{Szulkin-1986}
	A. Szulkin,
	{\it Minimax principles for lower semicontinuous functions and applications to nonlinear boundary value problems},
	Ann. Inst. H. Poincar\'{e} Anal. Non Lin\'{e}aire {\bf 3} (1986), no. 2, 77--109.

\bibitem{Szulkin-Weth-2010}
	A. Szulkin, T. Weth,
	{\it The method of Nehari manifold},
	Handbook of nonconvex analysis and applications, 597--632, Int. Press, Somerville, MA, 2010.

\bibitem{Tanaka-Zhang-2017}
	K. Tanaka, C. Zhang,
	{\it Multi-bump solutions for logarithmic Schr\"{o}dinger equations},
	Calc. Var. Partial Differential Equations {\bf 56} (2017), no. 2, Paper No. 33, 35 pp.

\bibitem{Tang-2013}
	X.H. Tang,
	{\it Infinitely many solutions for semilinear Schr\"{o}dinger equations with sign-changing potential and nonlinearity},
	J. Math. Anal. Appl. {\bf 401} (2013), no. 1, 407--415.

\bibitem{Thin-2022}
	N.V. Thin,
	{\it Multiplicity and concentration of solutions to a fractional $(p,p_1)$-Laplace problem with exponential growth},
	J. Math. Anal. Appl. {\bf 506} (2022), no. 2, Paper No. 125667, 46 pp.

\bibitem{Tolksdorf-1984}
	P. Tolksdorf,
	{\it Regularity for a more general class of quasilinear elliptic equations},
	J. Differential Equations {\bf 51} (1984), no. 1, 126--150.

\bibitem{Trudinger-1967}
	N.S. Trudinger,
	{\it On Harnack type inequalities and their application to quasilinear elliptic equations},
	Comm. Pure Appl. Math. {\bf 20} (1967), 721--747.

\bibitem{Wang-2025}
	X. Wang,
	{\it Singular Trudinger--Moser inequalities for the Aharonov--Bohm magnetic field},
	Adv. Nonlinear Stud. {\bf 25} (2025), no. 3, 807--821.

\bibitem{Wang-Zhang-2019}
	Z.-Q. Wang, C. Zhang,
	{\it Convergence from power-law to logarithm-law in nonlinear scalar field equations},
	Arch. Ration. Mech. Anal. {\bf 231} (2019), no. 1, 45--61.

\bibitem{Wilhelmsson-1987}
	H. Wilhelmsson,
	{\it Explosive instabilities of reaction-diffusion equations},
	Phys. Rev. A (3) {\bf 36} (1987), no. 2, 965--966

\bibitem{Willem-1996}
	M. Willem,
	``Minimax Theorems'',
	Birkh\"{a}user Boston, Inc., Boston, MA, 1996.

\bibitem{Xue-Zhang-Zhu-2025}
	J. Xue, C. Zhang, M. Zhu,
	{\it Trudinger-Moser type inequalities with logarithmic weights in fractional dimensions},
	Adv. Nonlinear Stud. {\bf 25} (2025), no. 1, 152--170.

\bibitem{Yang-2012}
	Y. Yang,
	{\it Existence of positive solutions to quasi-linear elliptic equations with exponential growth in the whole Euclidean space},
	J. Funct. Anal. {\bf 262} (2012), no. 4, 1679--1704.

\bibitem{Zeidler-1990}
	E. Zeidler,
	``Nonlinear Functional Analysis and Its Applications. II/B'',
	Springer-Verlag, New York, 1990.

\bibitem{Zhang-Sun-Liang-Thin-2024}
	X. Zhang, X. Sun, S. Liang, N.V. Thin,
	{\it Existence and concentration of solutions to a Choquard equation involving fractional $p$-Laplace via penalization method},
	J. Geom. Anal. {\bf 34} (2024), no. 3, Paper No. 90, 59 pp.

\bibitem{Zhang-Wang-2020}
	C. Zhang, Z.-Q. Wang,
	{\it Concentration of nodal solutions for logarithmic scalar field equations},
	J. Math. Pures Appl. (9) {\bf 135} (2020),  1--25.

\bibitem{Zhang-Zhu-2024}
	C. Zhang, M. Zhu,
	{\it Existence of ground states to quasi-linear Schr\"{o}dinger equations with critical exponential growth involving different potentials},
	Adv. Nonlinear Stud. {\bf 24} (2024), no. 3, 616--636.

\bibitem{Zloshchastiev-2010}
	K.G. Zloshchastiev,
	{\it Logarithmic nonlinearity in theories of quantum gravity: origin of time and observational consequences},
	Gravit. Cosmol. {\bf 16} (2010), no. 4, 288--297.

\end{thebibliography}
\end{document}